\documentclass[a4paper,11pt, reqno]{amsart}
\usepackage[T1]{fontenc}
\usepackage[english]{babel}
\usepackage[margin=1in]{geometry}
\usepackage{amsmath,amssymb,amsthm}
\usepackage{bbm,cite}
\usepackage{mathrsfs}
\usepackage{enumitem}
\usepackage{mathtools}
\usepackage{graphicx}
\usepackage[usenames,dvipsnames]{xcolor}
\usepackage[colorlinks=true,linkcolor=NavyBlue,urlcolor=RoyalBlue,citecolor=PineGreen,bookmarks=true,bookmarksopen=true,bookmarksopenlevel=2,unicode=true,linktocpage]{hyperref}
\usepackage{float}
\usepackage{tikz}
\usepackage{pgfplots}
\usetikzlibrary{calc,math}
\pgfplotsset{ticks=none}
\pgfplotsset{compat=1.16}

\makeatletter 
\@mparswitchfalse%
\makeatother

\reversemarginpar
\usepackage[colorinlistoftodos]{todonotes}

\DeclareMathAlphabet{\mathpzc}{OT1}{pzc}{m}{it}

\newtheorem{theorem}{Theorem}[section]

\newtheorem{corollary}[theorem]{Corollary}
\newtheorem{proposition}[theorem]{Proposition}
\newtheorem{lemma}[theorem]{Lemma}

\newtheoremstyle{named}{}{}{\itshape}{}{\bfseries}{.}{.5em}{\thmname{#1}\thmnumber{ #2}%
\thmnote{#3}}

\theoremstyle{named}
\newtheorem*{namedtheorem}{Theorem} 

\numberwithin{equation}{section}

\theoremstyle{definition}

\theoremstyle{remark}
\newtheorem{remark}[theorem]{Remark}
\newtheorem{remarks}[theorem]{Remarks}
\newtheorem*{remark*}{Remark}

\newcommand{\R}{{\mathbb R}}
\newcommand{\RP}{{\mathbb R}_+}

\newcommand{\N}{{\mathbb N}}

\newcommand{\Sp}{{\mathbb S}}
\newcommand{\BB}{{\mathbb B}}

\newcommand{\tra}{\top}
\newcommand{\os}{\overline{\sigma}^2}
\newcommand{\nvar}{\Upsilon}
\newcommand{\cDtot}{\widehat{\cD}}
\newcommand{\phitot}{\widehat{\phi}}
\newcommand{\sigmatot}{\widehat{\sigma}}
\newcommand{\cZtot}{\widehat{\cZ}}

\newcommand{\TV}{\mathrm{TV}}
\newcommand{\prok}{\mathrm{Prok}}
\newcommand{\indep}{\perp\!\!\!\!\perp} 

\def\d{\,{\rm d}}

\newcommand{\cC}{\mathcal{C}}
\newcommand{\cD}{\mathcal{D}}
\newcommand{\cM}{\mathcal{M}}
\newcommand{\cW}{\mathcal{W}}
\newcommand{\cX}{\mathcal{X}}
\newcommand{\cY}{\mathcal{Y}}
\newcommand{\cZ}{\mathcal{Z}}
\newcommand{\cH}{\mathcal{H}}
\newcommand{\cG}{\mathcal{G}}
\newcommand{\cF}{\mathcal{F}}

\DeclareMathOperator{\Diag}{Diag}
\DeclareMathOperator{\tr}{tr}
\DeclareMathOperator{\Tr}{Tr}

\makeatletter
\def\namedlabel#1#2{\begingroup
    (#2)%
    \def\@currentlabel{#2}%
    \phantomsection\label{#1}\endgroup
}
\makeatother

\title[CLT for superdiffusive reflected Brownian motion]{Central limit theorem for superdiffusive reflected Brownian motion} 

\date{\today}

\begin{document}

\author{Aleksandar Mijatovi\'c}
\address{Department of Statistics, University of Warwick \& The Alan Turing Institute, UK}
\email{a.mijatovic@warwick.ac.uk}

\author{Isao Sauzedde}
\address{Department of Statistics, University of Warwick, UK}
\email{isao.sauzedde@warwick.ac.uk}

\author{Andrew Wade}
\address{Department of Mathematical Sciences, Durham University, UK}
\email{andrew.wade@durham.ac.uk}

\begin{abstract}
We study the second-order asymptotics 
around the superdiffusive strong law~\cite{MMW}
 of a multidimensional driftless diffusion
 with oblique reflection from the boundary in a generalised parabolic domain.
In the unbounded direction we prove the limit is Gaussian with the usual diffusive scaling, while in the appropriately scaled cross-sectional slice we establish convergence to the invariant law of a reflecting diffusion in a unit ball. Using the separation of time scales, we also show asymptotic independence between these two components.
The parameters of the limit laws are explicit in the growth rate of the boundary and the asymptotic diffusion matrix and reflection vector field. A phase transition occurs when the domain becomes too narrow, in which case we prove that the central limit theorem for the unbounded component fails.
\end{abstract}

\maketitle

 \medskip


\section{Introduction}
\label{sec:introduction}
This paper quantifies, via distributional limit theorems, fluctuations of
driftless multidimensional diffusions with multiplicative noise in domains in $\R^{1+d}$ with an unbounded direction and oblique reflections at the boundary. For our class of models, the geometry of the domain and the structure of the reflection vector field is such that the first-order asymptotics are superdiffusive (possibly super-ballistic) in the unbounded direction. We now  describe their multidimensional fluctuations  on appropriate scales.

More precisely, we consider  the unique strong solution $(Z,L)$ to the following 
\emph{stochastic differential equation with reflection} (SDER)
\begin{equation}
\label{eq:SDER} 
Z_t=z_0+\int_0^t \sigma(Z_s) \d W_s+\int_0^t\phi(Z_s)\d L_s, \quad  L_t=\int_0^t \mathbbm{1}_{Z_s\in \partial \cD} \d L_s, \quad t \in \RP,
\end{equation}
driven by a standard Brownian motion $W$ on $\R^{1+d}$,
on the generalised parabolic domain 
\begin{equation}
\label{eq:D-def}
\cD \coloneqq \{(x,y)\in \RP \!\times \R^d: |y|_d\leq b(x) \}.
\end{equation}
Here $|\cdot|_d$ is the standard Euclidean norm on $\R^d$, $\RP\coloneqq[0,\infty)$ and $b:\RP\to\RP$ a $\cC^2$-function.
Assume that, as $x\to \infty$, the reflection vector field $\phi$ on the boundary $\partial \cD$ is asymptotically oblique  with an eventually positive component in the $x$-direction, the diffusion matrix $\sigma$ is asymptotically constant  and  
$b(x) \sim a_\infty x^\beta$ for some $\beta\in (-1,1)$ and $a_\infty>0$. Then, by~\cite[Thm~2.2]{MMW},
the superdiffusive strong law 
$X_t\sim c_1 t^{1/(1+\beta)}$ holds as $t\to\infty$, where $Z=(X,Y)$ and $c_1>0$ is an explicit constant.
In this paper we prove the 
following central limit theorem (CLT).
\begin{namedtheorem}[ (joint distributional convergence)]
\label{thm:1}
 Suppose that $\beta \in (-\frac{1}{3},1)$. Under suitable conditions (see Subsection~\ref{sec:assumptions} below), there exists a distribution $\mu$ on the unit ball $\BB^d$ in $\R^d$ such that, for a centred Gaussian distribution $\mathcal{N}$ on $\R$,  the weak convergence holds:
 \begin{equation}
  \label{eq:mainCVIntro}
  \Big(\frac{X_t-c_1 t^{\frac{1}{1+\beta}}}{\sqrt{t}} , \frac{Y_t}{a_\infty c_1^\beta t^\frac{\beta}{1+\beta}} \Big)  \underset{t\to \infty}\longrightarrow\mathcal{N}\otimes \mu.
  \end{equation}
\end{namedtheorem}
The variance of $\mathcal{N}$ is given explicitly in Equation~\eqref{eq:N-var} below.
For any $d\in\N$, the distribution $\mu$ is given as the stationary distribution to SDER~\eqref{eq:limitEqJoint} on the $d$-dimensional unit ball $\BB^d$. The precise statement of this result is in Theorem~\ref{th:main1} below. 

A key step in the proof of~\eqref{eq:mainCVIntro} consists of understanding the local behaviour of $Z$ in a moving time window, inside of which we establish the convergence  (as a process) of $Z$, properly rescaled and recentred, to a process in the infinite cylinder $\R\times\mathbb{B}^d$ (see Theorem~\ref{th:main2} below for details).  

\subsection*{A new phase transition at {\texorpdfstring{$\beta=-1/3$}{β=-1/3}}}
Under the asymptotically oblique assumption on the reflection vector field $\phi$ with a positive component in the $x$-direction and stable $\sigma$ as $x\to\infty$,
the asymptotic behaviour of the process $Z$ depends on the parameter $\beta\in\R$ (and not on the dimension $d$). 
Transitions at $\beta\in\{-1,0,1\}$ were previously known.
If $\beta<-1$, the domain is shrinking so fast that the process $Z$ explodes in finite time~\cite[Thm~2.2(i) \& Ex.~2.4]{MMW}. As mentioned above, by~\cite[Thm~2.2(ii)]{MMW}, if  $\beta\in(-1,1)$ then the first component $X$ of $Z$ satisfies a superdiffusive strong law. Moreover,  the value $\beta=0$  corresponds to the transition between expanding and shrinking domains, making the process $X$ either sub-ballistic or super-ballistic~\cite[Thm~2.2(ii)]{MMW}. For $\beta>1$, it is not hard to see that the domain $\cD$ is expanding sufficiently fast that the process does not ``feel the boundary'', making it diffusive. 

This paper identifies a further transition at $\beta=-1/3$  (corresponding to the almost sure behaviour $X_t\sim c_1t^{3/2}$ as $t\to\infty$). Beyond the distributional convergence for $\beta\in(-1/3,1)$  in~\eqref{eq:mainCVIntro} above, we prove that the CLT for $X_t-c_1 t^{1/(1+\beta)}$ does \emph{not} hold for $\beta\in(-1,-1/3)$, see Proposition~\ref{prop:nogo} below.
In this range of $\beta$, the proof of Proposition~\ref{prop:nogo} suggests that the process $X_t-c_1 t^{1/(1+\beta)}$ is of order $t^{1/(1+\beta)-1}\gg t^{1/2}$ as $t\to\infty$.  Furthermore, we expect that the fluctuations of $X_t-c_1 t^{1/(1+\beta)}$ are no longer Gaussian  for $\beta<-1/3$, but described by almost sure convergence to a non-Gaussian random variable. Evidence for a non-Gaussian strong limit in this regime is provided by a toy diffusion model in Subsection~\ref{sec:heuristic} below.
Intuitively, for $\beta\in(-1,-1/3)$, the second order term in the Taylor expansion overwhelms the ergodic fluctuations that drive the limit in~\eqref{eq:mainCVIntro} above  in the regime $\beta\in(-1/3,1)$.
We expect that this phase transition is universal, in that it is typical of any stochastic process with drift  asymptotically  proportional to the value of the process to the power $-\beta$ at large times and noise of order one with tails that are not too heavy.\footnote{It was pointed out to us by Andrey Pilipenko during the Isaac Newton Institute programme on \emph{Stochastic systems for anomalous diffusion} (July--December 2024) that, if the noise is an $\alpha$-stable L\'evy process, we should expect the transition to occur at $\beta=-1/(1+\alpha)$ rather than $\beta=-1/3$. We are grateful for this  observation and his detailed comments on an earlier version of the paper.}
The fact that we establish this phase transition in the case of reflected diffusions, where the trajectories of the drift are not even absolutely continuous,  can be viewed  as evidence for such universality.

In contrast to the behaviour of the first component of $Z$ discussed in the previous paragraph, the local convergence of appropriately rescaled and recentred $Z$ in Theorem~\ref{th:main2} below holds for all $\beta\in(-1,1)$. In particular, this implies that  the convergence of $Y_t/(a_\infty c_1^\beta t^\frac{\beta}{1+\beta})$ in~\eqref{eq:mainCVIntro} remains valid for $\beta\in(-1,-1/3]$. Figure \ref{fig:phases} summarises various behaviours of the process $Z=(X,Y)$.

The \href{https://youtu.be/SVcVtZafrVY?si=_9DPQXjErCcg7sej}{YouTube} presentation
in~\cite{YouTube_talk} describes our main results, including the new phase transition at $\beta=-1/3$, and discusses elements of the proofs.  

\begin{figure}[!ht]
\includegraphics{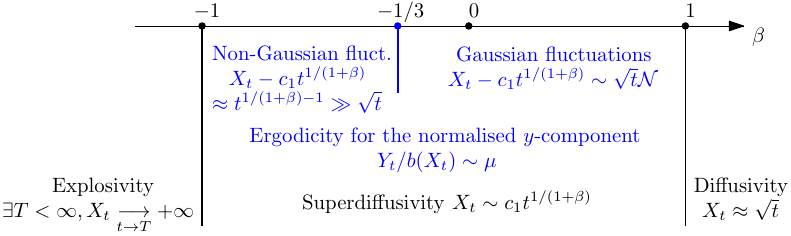}
\caption{\label{fig:phases}
The phase diagram describes the behaviour of the reflected process $Z$ following~\eqref{eq:SDER} at large times. The behaviour of $Z$ is a function of the growth/decay exponent $\beta$ of the boundary function $b$ in the definition of the domain $\mathcal{D}$ in~\eqref{eq:D-def}. The results of the present paper (Theorem~\ref{th:main1} and Proposition~\ref{prop:nogo} below) are in blue (the statements in black  in the diagram are from~\cite{MMW}).}\end{figure}

\subsection*{Literature and motivation}
\label{sec:literature}

Reflected diffusions have been studied extensively
as prototypical examples of stochastic processes with constrained (or controlled, confined) dynamics. The literature is large; we mention~\cite{KPR,lr19,varadhan-williams,lions-sznitman,kang2014characterization,kwon-williams,costantini-kurtz-2019,kang2017submartingale,ernst2021asymptotic,Burdzy17,Burdzy2006} for a selection of classical as well as more recent papers. 
 Motivation comes from heavy-traffic limits of queueing systems~\cite{harrison,hw87,rr}, communication networks~\cite{foschini}, or financial models~\cite{hhl,ipbkf}, for example.
Modern applications include sampling or optimization algorithms in computational statistics and machine learning~\cite{ahlw}, numerical methods for solving Neumann or mixed boundary-value problems~\cite{lst}, or estimation of an unknown set via observations of a reflected diffusion therein~\cite{cflp}.

The most well-studied examples of reflected diffusions in unbounded domains are orthants or cones~\cite{hobson1993recurrence,williams,franceschi2019integral,varadhan-williams}, where, typically, boundary reflections are infrequent and large-scale behaviour remains diffusive. Recurrence and transience for \emph{normally} reflected Brownian motion on generalized parabolic domains of the type specified via~\eqref{eq:D-def} was first studied in~\cite{pinsky2009transcience}. For reflection vector fields that are asymptotically normal in an appropriate sense, the case where the process is stable, with heavy-tailed stationary distributions, is studied in~\cite{bmw}.  A discrete (random walk) relative was studied in~\cite{MMW1}. A structurally similar, but much simpler, CLT for discrete-time process on $\RP$ with drift at $x$ of order $x^{-\beta}$, $\beta \in (0,1)$, is given in~\cite{MW-lamperti,MPW-book}. We note that $\beta <0$ is not considered in these papers.

The direct  antecedent work to the present paper is~\cite{MMW}, which was motivated to consider natural families of domains on which obliquely-reflected diffusion occurs sufficiently frequently
to drive \emph{anomalous diffusion}~\cite{bg,mjcb,oflv}. 
For statistical sampling and learning, processes that explore space faster than ordinary diffusions, or that
adapt their behaviour according to previous learning, can lead to more efficient algorithms.
Part of the motivation behind the present work is to understand more deeply the anomalous diffusion exhibited in~\cite{MMW}. 

Several aspects of our problem and approach appear to be related to the ideas in~\cite{costantini-kurtz-2019,costantini-kurtz-2024}, where reflected multidimensional diffusions in domains with one singularity are also studied. In our setting the singularity is at infinity rather than a cusp at the origin as in~\cite{costantini-kurtz-2024}. 
Furthermore, in our case as in~\cite{costantini-kurtz-2024}, the ergodicity of an embedded process plays a crucial role in the proofs.
Finally we note that~\cite{MR1661309}, 
for tube-like domains with variable widths, 
whether a minimal harmonic function remains a minimal parabolic function depends greatly on the rate of thinning of such domains (see~\cite[Thm~1.6]{MR1661309}).
This behaviour is similar in spirit\footnote{We thank Chris Burdzy for bringing  this analogy to our attention.} to our main result, which shows that the asymptotic behaviour of a reflected Brownian motion in tube-like domains also depends greatly on the rate of thinning in such domains.

\section{Main results}
\label{sec:description}

In this section we present formally the assumptions we use throughout the remainder of the paper (Subsection~\ref{sec:assumptions}) and state precisely the main results (Subsection~\ref{sec:theorems}), including a rigorous statement of the theorem in Section~\ref{sec:introduction}. In Subsection~\ref{sec:heuristic}  we 
discuss a heuristic leading to a one-dimensional toy model,  providing intuition for the key phenomena discussed above. 
Subsection~\ref{sec:outline} presents the structure of the proof and the organisation of the paper.

\subsection{Assumptions}
\label{sec:assumptions}

Recall that $|y|_d$ is the norm of $y\in\R^d$, i.e.~$|y|_d\coloneqq (y_1^2+\dots+y_d^2)^{\frac{1}{2}}$. We denote by $\BB^d$ the (closed) unit ball $\BB^d\coloneqq \{y\in \R^d: |y|_d\leq 1 \}$, and by $\Sp^{d-1}$ the unit sphere $\Sp^{d-1} \coloneqq\partial \BB^d$.  
Let $\cD\subset \R^{1+d}$ be the domain defined in~\eqref{eq:D-def}
for a given $b: \RP\to \RP$. 
In this domain,  $(Z,L,W)$ is (when it uniquely exists) the strong solution to SDER~\eqref{eq:SDER} above.
The process $L$, referred to as the \emph{local time of $Z$} at the boundary $\partial \cD$, is continuous, non-decreasing and starts from~$L_0=0$. 
Both $L$ and $Z$ are adapted to the filtration generated by $W$,   and $Z$ takes values in $\cD$ started at $Z_0 = z_0 \in \cD$. 

Assume that the function $b$ that specifies $\cD$ via~\eqref{eq:D-def} satisfies the following condition, ensuring that the boundary $\partial \cD\coloneqq \{(x,y)\in \RP \!\times \R^d: |y|_d= b(x) \}$ of the domain is $\cC^2$~\cite[Lem.~4.3]{MMW}. The conditions at $0$ in \eqref{hyp:D1} are equivalent to $\partial \mathcal{D}$ being $\mathcal{C}^2$ at the origin.
\begin{description}
\item
[\namedlabel{hyp:D1}{D$_1$}] 
The function $b:\RP\to \RP$ is continuous with $b(0)=0$ and $b(x)>0$ for $x>0$. Suppose $b$ is twice continuously differentiable on $(0,\infty)$, such that $\liminf_{x\to 0} b(x)b'(x)>0$, and $\lim_{x\to 0} b''(x)/b'(x)^3$ exists in $(-\infty,0]$. 
\end{description}

We furthermore assume that the function $b$ satisfies the following at infinity.
\begin{description}
\item
[\namedlabel{hyp:D3}{D$_2^+$}] 
The function $b:\RP\to \RP$ can be expressed as $b(x)= a_\infty x^\beta+f(x)$, where
 $\beta\in (-1,1)$, $a_\infty>0$ and, as $x\to \infty$, $f(x)=o(x^{\frac{3\beta-1}{2} })$, $f'(x)=o(x^{\frac{3\beta-3}{2} })$ and $f''(x)=o(x^{\frac{3\beta-5}{2} })$.
\end{description}
(The name~\eqref{hyp:D3} indicates that this 
is a strengthened version of Assumption~(D$_2$) in~\cite{MMW}.) Domains considered in this paper can either expand or shrink: see Figure~\ref{fig:grow_shrink} below.

Next we impose conditions on the diffusion matrix $\sigma$. From here on, we use the same convention as in \cite{MMW}, that for an  element in $\R\times \R^d$, we use the index $0$ for the coordinate in $\R$ which plays a distinguished role, and we use the indices $1,\dots, d$ for the coordinates in $\R^d$. 
For example, a matrix $M$ in the set of positive symmetric matrices $\cM^{+}_{1+d}$ is expressed as 
$M=(M_{i,j})_{i,j\in \{0,\dots, d\} }$.
Throughout 
$f(x,y)\underset{x\to \infty}\longrightarrow 0$ denotes 
$\lim_{x \to \infty} \sup_{y: (x,y)\in \cD} \|f(x,y)\| = 0$ if $f$ is defined on $\cD$, or $\lim_{x\to\infty} \sup_{y: (x,y)\in \partial \cD} \|f(x,y)\|= 0$ if $f$ is defined on $\partial \cD$ (for an appropriate norm $\|\cdot\|$).
\begin{description}
\item
[\namedlabel{hyp:C}{C$^+$}]
The diffusion matrix $\sigma: \cD\to \cM^{+}_{1+d}$ is bounded, globally Lipschitz and uniformly elliptic.\footnote{Without loss of generality, we assume that $\sigma$ is symmetric. In fact, we could have allowed the driving Brownian motion $W$ in SDER~\eqref{eq:SDER} to be of dimension $m\geq 1+d$, with dispersion matrix $\sigma$ mapping $\R^m$ onto $\R^{1+d}$, without modifying the results or significantly changing the proofs in the paper.} Furthermore, there exists $\Sigma^\infty\in\cM^{+}_{1+d}$ such that $\Sigma\coloneqq \sigma^2: \cD\to \cM^{+}_{1+d}$
satisfies
\[
\Sigma( x,y)-\Sigma^\infty \underset{x\to \infty}\longrightarrow 0
\qquad 
\text{ and }
\qquad
x^{ \frac{1-\beta}{2}} \Big(\sum_{i=1}^d 
 \Sigma_{i,i}(x,y) - \os \Big)\underset{x\to \infty}\longrightarrow 0,
\]
where $\os \coloneqq \Tr(\Sigma^\infty   )-\Sigma^\infty_{0,0}=\sum_{i=1}^d \Sigma^\infty_{i,i}$.
\end{description}

Finally, we impose conditions on the boundary vector field $\phi:\partial \cD\to \R^{1+d}$.  
We write $\phi=(\phi_0,\dots,\phi_d) \in\R^{1+d}$ and $\phi^{(d)}\coloneqq (\phi_1,\dots,\phi_d)\in\R^d$.
Throughout the paper $\langle\,\cdot\,,\cdot\,\rangle$ denotes the standard inner product on a Euclidean space of appropriate dimension. 
Let $n:\partial \cD\to \R^{1+d}$ denote the
vector field orthogonal to $\partial \cD$, inward-pointing with unit norm.

\begin{description}
\item[\namedlabel{hyp:V+}{V$^+$}] The vector field $\phi:\partial \cD\to \R^{1+d}$ is $\cC^2$ and bounded. For all $z\in \partial \cD$, $\langle \phi(z),n(z)\rangle >0$.
There exists a constant $s_0>0$ and a function
$\phi^{(d)}_\infty: \Sp^{d-1}\to \R^d$ such that
\[
\sup_{u\in \Sp^{d-1} }  x^{ \frac{1-\beta}{2} }  \left| \phi(x,b(x) u) - (s_0, \phi^{(d)}_\infty(u) )  \right|_{d+1} \underset{x\to \infty}\longrightarrow 0.
\]
There exists a constant $c_0>0$ such that  $c_0= \langle \phi^{(d)}_\infty(u), -u \rangle>0$ for all $u\in \Sp^{d-1}$.
\end{description}

Conditions \eqref{hyp:C} and \eqref{hyp:V+} imply Assumption  ({\sc A}) from \cite{MMW} (the constant we named $\os$ is $\sigma^2$ in \cite{MMW}).
Conditions \eqref{hyp:D3}, \eqref{hyp:C}, \eqref{hyp:V+},  imply respectively the conditions ({\sc D$_2$}), ({\sc C}), and ({\sc V}) from \cite{MMW}, with the same value of $\beta$.
The conditions here are stronger, mostly in that they require
control of  the rates of convergence of the problem data, not necessary for the first-order asymptotics in~\cite{MMW}.
Under the  assumptions
\eqref{hyp:D1}, \eqref{hyp:D3}, \eqref{hyp:C},  and \eqref{hyp:V+}, Theorem~A.1 and  Theorem~2.2(ii) in \cite{MMW} ensure that the strong solution $(Z,L,W)$ of the SDER \eqref{eq:SDER} exists for all time, is pathwise unique, and that almost surely,
\begin{equation} 
\label{def:c1}
t^{-\frac{1}{1+\beta}} X_t \underset{t\to \infty}\longrightarrow c_1\coloneqq 
 \Bigl( \frac{(1+\beta) s_0\os}{2 a_\infty c_0}\Bigr)^{\frac{1}{1+\beta}}.
 \end{equation}
 
In order to facilitate the $\cC^1$-convergence of the \emph{recentred and rescaled} versions of the coefficients $\sigma$ and $\phi$  defined in~\eqref{def:phix} below, we will use the following additional assumption.
\begin{description}
  \item[\namedlabel{hyp:S}{S}]
  There exists $\epsilon>0$ such that for all $i,j ,k \in \{ 0,\dots, d\}$,
  \[
  \sup_{u\in \mathbb{B}^{d} }x^{\beta+\epsilon} | \partial_i \sigma_{j,k} (x,b(x)u)|  \underset{x\to \infty}\longrightarrow 0
  \quad 
  \mbox{and}
  \quad  \sup_{u\in \Sp^{d-1} } x^{\beta+\epsilon} | \operatorname{grad} \phi_{j} (x,b(x)u) | \underset{x\to \infty}{\longrightarrow} 0,
  \]
  where $\operatorname{grad} \phi_{j}$ is the gradient taken along $\partial \mathcal{D}$.
\end{description}

Assumptions \eqref{hyp:D3}, \eqref{hyp:C} and \eqref{hyp:V+} 
are essential  for the CLT for $X$. Indeed,  removing any of them would in general give rise to additional terms of order smaller than the deterministic first-order approximation of $X$ (given by $c_1t^{1/(1+\beta)}$) but greater than the stochastic second-order correction (proportional to $\sqrt{t}$).
 As a simple example, if $\tilde b(x)=a_\infty x^{\beta}+  x^{\beta(1-\epsilon)}$ for all large $x$ and any $\epsilon\in(0,(1-\beta)/2)$, then there exists $c'_1> 0$ 
 such that almost surely,\footnote{This asymptotic equivalence follows by observing that $\tilde X_t\sim X_t+X_t^{1-\epsilon}$, where $X$ corresponds to the boundary function $b(x)=a_\infty x^\beta$, It\^o's formula and 
the integration-by-parts formula for local time in Subsection~\ref{subsec:loc_time}.} $\tilde X_t-c_1t^{\frac{1}{1+\beta}}\sim c'_1 t^{\frac{1-\epsilon}{1+\beta}}\gg \sqrt{t}$ as $t\to\infty$, violating the CLT.

Assumption \eqref{hyp:S} is of a technical nature. It is used to prove existence and uniqueness of a strong solution to SDER \eqref{eq:limitEqJoint} below, and to guarantee certain continuity properties exploited in our proofs (see Section~\ref{sec:local}). It is in fact possible to replace this condition with the weaker condition with $\epsilon=0$, but assuming $\epsilon>0$ allows us to choose some explicit control functions rather than non-explicit ones, simplifying the exposition.

\begin{figure}[!ht]
\includegraphics[width=\textwidth]{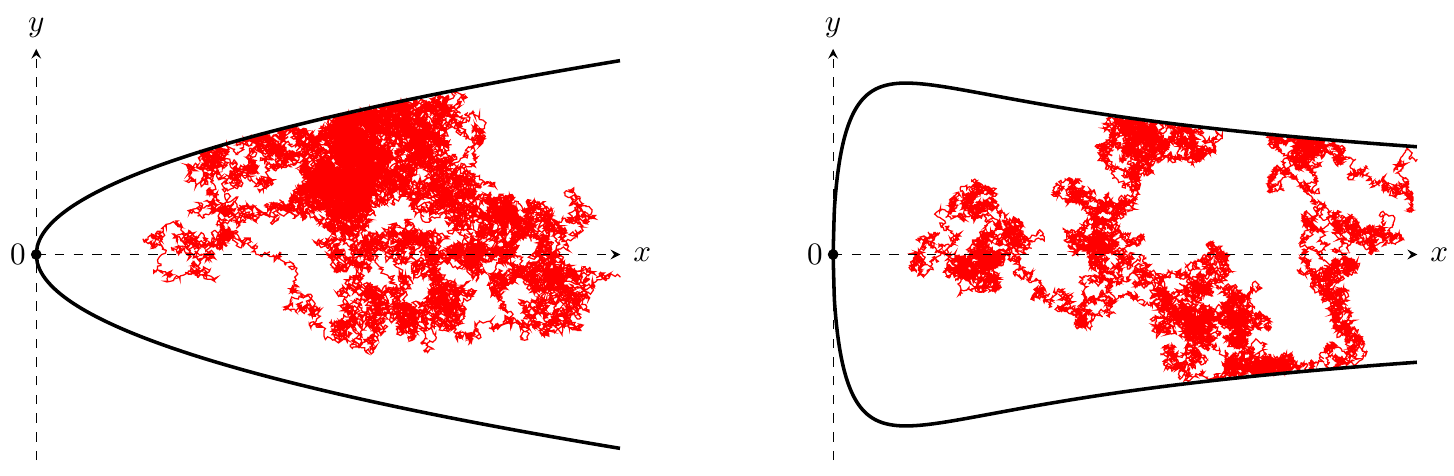}
\caption{\label{fig:grow_shrink}
Domain $\mathcal{D}$ satisfying assumptions~\eqref{hyp:D1} and~\eqref{hyp:D3} can either expand or shrink: the boundary function $b$ on the left (resp. right) is proportional to $b(x)\sim x^{1/2}$ (resp. $b(x)\sim x^{-1/4}$) as $x\to\infty$.
A simulated path of the solution of the SDER in~\eqref{eq:SDER} with the identity dispersion matrix and oblique reflection satisfying assumptions \eqref{hyp:V+} and~\eqref{hyp:S} is also depicted in both graphs.}
\end{figure} 

\subsection{Limit theorems}
\label{sec:theorems}
    To state our main theorem, we introduce the SDER on $ \BB^d$ given by
    \begin{equation} \label{eq:limitEq} \d \cY_{t} = \pi_d \sigma^\infty \d W_{t}+\phi^{(d)}_\infty(\cY_{t}) \d L^{\cY}_{t},\end{equation} 
    where $\sigma^\infty$ is a square root of $\Sigma^\infty$ in \eqref{hyp:C}, the vector field $\phi^{(d)}_\infty$ on $\Sp^{d-1}$ is given in~\eqref{hyp:V+}, $W$ is a standard Brownian motion  on $\R^{1+d}$ and $\pi_d: \R\times\R^d\to \R^d$, $\pi_d(x,y)=y$, is the canonical projection. Here is the precise version of the informal theorem in Section~\ref{sec:introduction} above. 
    \begin{theorem}
    \label{th:main1}
    Suppose that hypotheses \eqref{hyp:D1}, \eqref{hyp:D3}, \eqref{hyp:C}, \eqref{hyp:V+}, and~\eqref{hyp:S} hold with $\beta\in(-\frac{1}{3},1)$. 
    SDER~\eqref{eq:limitEq} admits a unique strong solution with a unique invariant measure $\mu$ on $\BB^d$.
    
    Let $Z=(X,Y)$ be the solution of the SDER in~\eqref{eq:SDER} and recall the constant $c_1$ from~\eqref{def:c1}.  In distribution, as $t \to \infty$, we have
    \[ \biggl( \frac{X_t-c_1 t^{\frac{1}{1+\beta}} }{\sqrt{t}}, \frac{Y_t}{a_\infty c_1^\beta t^\frac{\beta}{1+\beta}} \biggr)\longrightarrow \mathcal{N}\otimes \mu, \]
    where $\mathcal{N}$ is a centred Gaussian law on $\R$ with variance $\nvar$ given by 
    \begin{equation}
        \label{eq:N-var}
    \nvar\coloneqq  \frac{1+\beta}{1+3\beta} \Big( \Sigma^\infty_{0,0}+ 2\frac{ s_0}{c_0} \sum_{j=1}^d  \Sigma_{0,j}^{\infty} \int_{\BB^d} y_j
    \d \mu_y+\frac{s_0^2}{c_0^2} \sum_{j,k=1}^d \Sigma^\infty_{j,k} \int_{\BB^d} y_j y_k
    \d \mu_y\Big).
    \end{equation}
\end{theorem}

The proof of Theorem~\ref{th:main1} is the conclusion of Section~\ref{sec:central} below. We note in passing that in the second component of the weak limit in Theorem~\ref{th:main1},
the space is scaled either down or up depending on the sign of  $\beta$, while in the first component it is scaled down in all cases.

\begin{remark}
\label{rem:BM_limit_on_disk}
Reflected Brownian motion is the special
        case in which $\Sigma$ is the identity, in which case $\nvar$ given by~\eqref{eq:N-var} simplifies to 
        \[ \nvar = \frac{1+\beta}{1+3\beta} \left(1 + \frac{s_0^2}{c_0^2} \int_{\BB^d} | y |_d^2 \mu (\mathrm{d} y) \right) .\]
        If furthermore the projection in the $y$-direction of the reflection vector field $\phi$ is asymptotically normal, i.e.~$\phi^{(d)}_\infty(u)=-c_0 u$ for all $u\in\mathbb{S}^{d-1}$, then  $\mu$ is uniform on $\BB^d$ and 
        \[ \int_{\BB^d} | y |_d^2 \mu (\mathrm{d} y) = \frac{\int_0^1 r^{d+1} \d r}{\int_0^1 r^{d-1} \d r} = \frac{d}{d+2}. \]
        If we further specify $\beta=0$ (corresponding to a half-cylinder), we get
        $\nvar =  1 + \frac{s_0^2}{c_0^2} \frac{d}{d+2}$.
\end{remark}
It is essential for our proof of Theorem~\ref{th:main1} to describe the asymptotic dynamics of $Z$ at the scale where this dynamics is stochastic and non-trivial and can be approximated by a reflected diffusion in an infinite cylinder. 
To this end, we introduce a process $\cZ^T$ which is constructed from the process $Z$ started at a large time $T$ via a $T$-dependent time-change. We centre the process by subtracting (likely, large) $X_T$, to observe local behaviour. Then, as we want the boundary of the domain to stay at an approximately constant distance in the $y$-direction, we rescale by a multiplicative factor $1/ b(X_T) \approx T^{-\frac{\beta}{1+\beta}}$. Finally, as we want stochastic fluctuations to be visible on this space scale, we time change  by a factor $b(X_T)^2 \approx T^{\frac{2\beta}{1+\beta}} \ll T$. Transforming $Z=(X,Y)$ in this way, for any $T\in(0,\infty)$, we arrive at the following process:
\begin{equation}
    \label{eq:Z^T-def}
 \cZ^T \coloneqq (\cZ^T_t)_{t \geq 0},\quad\text{where}\quad  \cZ^T_{t}\coloneqq 
\frac{1}{b(X_T)}  \bigl( X_{T+ b(X_T)^2 t}   - X_T,    Y_{T+ b(X_T)^2 t}  \bigr).
 \end{equation}

\begin{theorem}
\label{th:main2}
    Suppose that hypotheses \eqref{hyp:D1}, \eqref{hyp:D3}, \eqref{hyp:C}, \eqref{hyp:V+}, and~\eqref{hyp:S} hold with $\beta\in(-1,1)$. 
Let $Z=(X,Y)$ be the strong solution to the SDER~\eqref{eq:SDER}.
For every $s>0$, the process
  $(\cZ^T_\mathbf{t})_{\mathbf{t}\in[0,s]}$
  converges weakly (in the uniform topology), as $T\to \infty$. The limit is the law of the strongly unique solution $\cZ=(\cX,\cY)$ of the SDER on $\R\times \BB^d$,
  \begin{equation} 
\label{eq:limitEqJoint}
  \d \cZ_\mathbf{t} =\sigma^\infty \d W_\mathbf{t}+ \bigl( s_0,\phi^{(d)}_\infty(\cY_\mathbf{t}) \bigr) \d L^{\cY}_\mathbf{t}, \qquad \mathbf{t}\in [0,s], \end{equation}
   where $\cX_0=0$ and $\cY_0$, following the invariant law  $\mu$ of~\eqref{eq:limitEq} on $\BB^d$, is independent of $W$ (note that $L^\cY$, the local time of $\cZ$ on $\R\times \Sp^{d-1}$, depends only on $\cY$ but not on $\cX$).
\end{theorem}
  Theorem~\ref{th:main2} is proved in Section~\ref{sec:local}; see Proposition~\ref{prop:seq} and Remark~\ref{rem:proof_of_main2} below.

\begin{remarks}\ 
\label{rem:beginning}
\begin{enumerate}[label=(\alph*)] 
\item 
\label{rem:beginning(a)}
Equation \eqref{eq:limitEqJoint} projects in the $y$-direction into the autonomous SDER in~\eqref{eq:limitEq}. The component $\cX$ can then be recovered via a stochastic integral as $\cX$ does not feature on the right-hand side of~\eqref{eq:limitEqJoint}. Indeed, we have  $\cX=\pi_0\sigma^\infty W+s_0L^{\cY}$,
where the canonical projection $\pi_0:\R\times\R^d\to\R$ is given by $\pi_0(x,y)=x$, making $\cX$ the sum of a scalar Brownian motion and the local time of $\cY$ at $\partial \mathbb{B}^d$ scaled by the constant $s_0$.
\item 
In the definition of $\cZ^T$ in~\eqref{eq:Z^T-def}, we may replace $b(X_T)$ with $a_\infty c_1^\beta T^{\beta/(1+\beta)}\sim  b(X_T)$ without altering the conclusion of Theorem~\ref{th:main2}. Note also that in the scaling limit of $\cZ^T$ in the theorem, for $\beta>0$
(resp.\ $\beta<0$), we accelerate (resp.\ decelerate) time and scale down (resp.\ up) space. 
\item Unlike Theorem~\ref{th:main1}, Theorem~\ref{th:main2} holds for all $\beta\in(-1,1)$. 
\item For $\beta > -1/3$, the spatial scale in Theorem~\ref{th:main2} on which we observe the dynamics of $X$ around $X_T$ is $T^{\beta/(1+\beta)}$. This scale is much smaller than the scale of the typical fluctuations of $X_T$ itself, which are of order $\sqrt{T}$ by  Theorem~\ref{th:main1}.  Figure~\ref{fig:path} illustrates the various scales  in Theorems~\ref{th:main1} and~\ref{th:main2}.
    \end{enumerate}
\end{remarks}

\begin{figure}[!h]
\includegraphics[width=\textwidth]{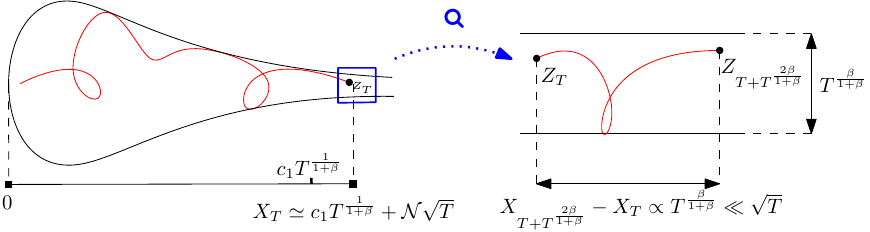}
\caption{\label{fig:path}
The limit in~\eqref{def:c1} from \cite{MMW} states that the typical position of $X_T$ is around $c_1 T^\frac{1}{1+\beta}$.
Theorem~\ref{th:main1} yields the picture on the left (here $\beta>-\frac{1}{3}$): the typical error $X_T- c_1 T^\frac{1}{1+\beta}$ is, at its main order, Gaussian with scale proportional to $T^{\frac{1}{2}}$.
Theorem~\ref{th:main2} describes the behaviour of $Z$ in the magnified picture on the right (which depicts in red the rescaled process $\cZ^T$). Note that since $T$ is finite, the domain in the magnified picture is only  \emph{approximately} a cylinder.}
\end{figure}

\subsection{A one-dimensional heuristic}
\label{sec:heuristic}

Using an estimate (based on a renewal argument in~\cite[p.679]{pinsky2009transcience}) of the effective horizontal drift accumulated via reflection, a heuristic in~\cite[p.1815]{MMW} suggests that the toy model
\begin{equation}
    \label{eq:SDE-toy}
 \d \widetilde{X}_t = c' \tilde{X}_t^{-\beta} \d t + \d \widetilde W_t,
 \end{equation}
 for some constant $c'>0$ and  a Brownian motion $\widetilde W$ in $\R$, ought to have long-time behaviour analogous to the first component $X$ of the reflected diffusion $Z$.
SDE~\eqref{eq:SDE-toy} falls into the class of one-dimensional diffusions studied in~\cite{go13} and indeed exhibits a behaviour very similar to that of $X$, with the same transitions at $\beta\in\{1,-1\}$ and the superdiffusive law of large numbers $\widetilde{X}_t\sim c t^{1/(1+\beta)}$ as $t\to\infty$, where  $c\coloneqq (c'(1+\beta))^{1/(1+\beta)}>0$~\cite[Thm~4.10(i)]{go13}.

The main results in the present paper, Theorems~\ref{th:main1} and~\ref{th:main2},  refine the asymptotics of $X$ in the superdiffusive but non-explosive regime, that is for $\beta\in(-1,1)$.
We now sketch an analysis of fluctuations for the toy model~\eqref{eq:SDE-toy}, which exhibits some of the relevant intuition for our main results, including the significance of the critical value $\beta = -1/3$. We stress however that the intuitive reasoning we are about to sketch for the one-dimensional diffusion $\widetilde X$ is a long way from a rigorous proof of our limit theorems, which crucially depend on the analysis of the fast mixing and ergodicity of the $d$-dimensional component $Y$ of the reflected process $Z=(X,Y)$.

It\^o's formula yields
 $\widetilde{X}_t^{1+\beta} = \widetilde{X}_0^{1+\beta} + c' (1+\beta)t + (\beta(1+\beta)/2)  A_t  + (1+\beta) M_t$,
where the process $A$ and the quadratic variation $[M]$ of the local martingale $M$ are respectively given by 
\[ A_t = \int_0^t \widetilde{X}_s^{\beta-1} \d s  \qquad\text{and}\qquad 
[M]_t =  \int_0^t \widetilde{X}_s^{2\beta} \d s.\]
Since  $\widetilde{X}_t\sim c t^{1/(1+\beta)}$ as $t\to\infty$ almost surely, for $\beta\in(0,1)$ we have 
$-1<(\beta-1)/(\beta+1)<0$
 and hence $A_t \approx t^{\frac{2\beta}{\beta+1}}\ll t$, while for $\beta\in(-1,0)$  we get $(\beta-1)/(\beta+1)<-1$ implying a finite limit $A_t \to A_\infty < \infty$ as $t\to\infty$. Thus
$A_t/t=O(t^{\frac{2\beta^+-\beta-1}{1+\beta}})$ for all $\beta\in(-1,1)$, where $\beta^+=\max\{\beta,0\}$.

For $\beta\in(-1/3,1)$
we get $[M]_t\sim c^{2\beta}\frac{1+\beta}{1+3\beta}t^{\frac{1+3\beta}{1+\beta}}$ as $t\to\infty$ almost surely. Since (by the Dambis--Dubins-–Schwarz theorem) $M$ is equal in law to a Brownian motion time-changed by its quadratic variation $[M]$, we  get weak convergence to a centred Gaussian with variance $c^{2\beta}  \frac{1+\beta}{1+3\beta}$:
\begin{equation}
\label{eq:Mart_conv_toy}
     t^{-\frac{1+3\beta}{2(1+\beta)}} M_t \to \mathcal{N} \big( 0, c^{2\beta} (1+\beta)/(1+3\beta)\big) \quad \text{ for all } -1/3 < \beta <1.
\end{equation}
In particular, in this case we have $M_t/t\to0$ in probability as $t\to\infty$.
For $\beta\in(-1,-1/3)$, the quadratic variation $[M]_t$ converges as $t\to \infty$ almost surely and hence $M_t$ tends almost surely to a non-degenerate random variable $M_\infty$.

The inequality $-1\leq \frac{2\beta^+-\beta-1}{1+\beta}$ for $\beta\in(-1,1)$ implies $t^{-1}=o(t^{\frac{2\beta^+-\beta-1}{1+\beta}})$. 
The semimartingale decomposition of $\widetilde X^{1+\beta}$, $A_t/t=O(t^{\frac{2\beta^+-\beta-1}{1+\beta}})$ for $\beta\in(-1,1)$ and the calculations above yield 
\begin{align*} \widetilde{X}_t  & =  c t^{\frac{1}{1+\beta}}  \Bigl( 1 +  t^{-1}  M_t/c' + O \bigl(t^{\frac{2\beta^+-\beta-1}{1+\beta}} \bigr) \Bigr)^{\frac{1}{1+\beta}} \\
& =  c t^{\frac{1}{1+\beta}} +  t^{-\frac{\beta}{1+\beta}} M_t/c^\beta + O \bigl( t^{\frac{2\beta^+-\beta}{1+\beta}} \bigr)\qquad\text{as $t\to\infty$.}
\end{align*}
Note that $\frac{1}{2}+\frac{\beta}{1+\beta}=\frac{1+3\beta}{2(1+\beta)}$. 
For $-1/3 < \beta <1$   we have 
$\frac{2\beta^+-\beta}{1+\beta}-\frac{1}{2}<0$. Thus~\eqref{eq:Mart_conv_toy} yields 
\[ t^{-1/2} \Bigl( \widetilde{X}_t - c' t^{\frac{1}{1+\beta}} \Bigr)= t^{-\frac{1}{2}-\frac{\beta}{1+\beta}} M_t /c^\beta + O \bigl( t^{\frac{2\beta^+-\beta}{1+\beta}-\frac{1}{2}} \bigr) \underset{t\to \infty}\longrightarrow \mathcal{N} \big( 0,(1+\beta)/(1+3\beta) \big)\]
in distribution.

If $-1 < \beta < -1/3$, then the almost sure limits $A_\infty$ and $M_\infty$ play a role:  as $t\to\infty$ we get
\begin{align*} 
  \widetilde{X}^{1+\beta}_t -   c'(1+\beta)  t  & =   C_{0,\infty}+o(1),\quad\text{where  $C_{0,\infty}\coloneqq \widetilde{X}_0^{1+\beta} + \frac{\beta (1+\beta)}{2} A_\infty + (1+\beta) M_\infty$.}
\end{align*}
Hence $\widetilde{X}_t=ct^{\frac{1}{1+\beta}}(1+\frac{C_{0,\infty}}{c^{1+\beta}}t^{-1}+o(t^{-1}))^{\frac{1}{1+\beta}}=ct^{\frac{1}{1+\beta}}(1+\frac{C_{0,\infty}}{(1+\beta)c^{1+\beta}}t^{-1}+o(t^{-1}))$ as $t\to\infty$, 
implying almost sure convergence to a random variable:
\[ t^{\frac{\beta}{1+\beta}} \Bigl( \widetilde{X}_t -  c t^{\frac{1}{1+\beta}} \Bigr) \underset{t\to \infty}\longrightarrow 
C_{0,\infty}c^{-\beta}/(1+\beta). \]
Since for $-1 < \beta < -1/3$ we have $-1/2<\beta/(1+\beta)$, the quantity  $(\widetilde X_t-c t^{\frac{1}{1+\beta}})/\sqrt{t}$ does not converge weakly to a non-degenerate limit law (as discussed in Section~\ref{sec:introduction}, by Proposition~\ref{prop:nogo} below, the same holds true for $X$).

\noindent\subsection{Skeleton of the proof of Theorem~\ref{th:main1} and structure of the paper}
\label{sec:outline}
The component $X$ of the reflected process $Z=(X,Y)$ satisfies a superdiffusive strong law $X_t/t^{\frac{1}{1+\beta}}\to c_1$ as $t\to\infty$. As indicated in Figure~\ref{fig:path}, the proof that $(X_t-c_1t^{\frac{1}{1+\beta}})/\sqrt{t}$ converges weakly to a centred Gaussian is all about
identifying a time window around $c_1 t^{\frac{1}{1+\beta}}$ of appropriate length smaller than $\sqrt{t}$, during which the domain does not change significantly, while the $Y$ component mixes so that it is near stationarity. More precisely, the proof consists of the following five steps.

\smallskip
\noindent\underline{\emph{Step 1}}. 
Subsection~\ref{sec:cylinder-sde}
of 
Section~\ref{sec:limiting}  constructs the limiting process $\mathcal{Z}=(\mathcal{X},\mathcal{Y})$ satisfying SDER~\eqref{eq:limitEqJoint} in the infinite cylinder and establishes the convergence to the unique invariant measure $\mu$ for the component $\mathcal{Y}$ in the ball $\mathbb{B}^d$.  Subsection~\ref{sec:cylinder-coupling} constructs the (maximal) coupling of the process $\mathcal{Y}$  started at an arbitrary distribution supported on $\BB^d$ with the process started at the invariant measure $\mu$. This coupling construction is designed to extend the original probability space (supporting $Z$ and $W$ in~\eqref{eq:SDER}) and possess certain (conditional) independence properties, see Proposition~\ref{prop:coupling} below. These properties are essential for an effective comparison (in Step~3 of the proof) of the stationary process in the infinite cylinder with the stochastically rescaled (both in space and time) and recentered process $\cZ^T$ in~\eqref{eq:Z^T-def}.

\smallskip
\noindent \underline{\emph{Step 2}}. 
The solutions to the SDER (given in~\eqref{eq:SDER_Skorkhod_Prob} of Subsection~\ref{subsec:def}; see also~\eqref{eq:ZT=ZZT}) satisfied by the rescaled process $\mathcal{Z}^T$ in~\eqref{eq:Z^T-def},  started at a fixed point, are proved to converge in distribution to the distribution of the limiting process $\mathcal{Z}$ as $T\to\infty$. 
The key issue here is that, not only do the reflected processes $\mathcal{Z}^T$ and $\cZ$ satisfy different SDERs, but moreover they also live in different domains.
This issue is resolved in Subsection~\ref{sec:cv} by reducing the continuity of an SDER with respect to a parametric family of coefficients to continuity with respect to the starting point of a single $(d+2)$-dimensional SDER, using the parameter as an additional variable (recall that both $\mathcal{Z}^T$ and $\cZ$ live in $(d+1)$-dimensional domains). Modulo localisation, Subsection~\ref{sec:cv} applies the seminal result from~\cite{Ishii} (which does not require ellipticity) to conclude the continuity in the starting point of the extended system,
implying the desired convergence of the rescaled  process $\cZ^T$ to the limiting process $\cZ$ in the infinite cylinder.

\smallskip
\noindent \underline{\emph{Step 3}}.
 The second step enables a comparison of the rescaled process $\mathcal{Z}^T$ and the limiting process $\mathcal{Z}$ in the infinite cylinder, provided they start at the same point. In Subsection~\ref{subsec:techPart}, we remove this restriction and prove that $\mathcal{Z}^T$ can be coupled to the limiting process $\mathcal{Z}$ in stationarity. This is achieved by first allowing sufficient time for the $\mathcal{Y}^T$ component of  $\mathcal{Z}^T$ to mix, so that it is almost in stationarity. This construction,  based on Proposition~\ref{prop:coupling} discussed in Step~1 above and carried out in the proof of Proposition~\ref{prop:seq} in Subsection~\ref{subsec:techPart} below, constitutes the most technical part of the paper. It enables us to conclude that the second component in the weak limit in Theorem~\ref{th:main1} above converges to the invariant measure $\mu$. 
 The asymptotic independence in Theorem~\ref{th:main1}, established in Subsection~\ref{subsec:Asymptotic_indep},
 essentially follows from the fact that the second component mixes much faster than  the first component fluctuates.

 \smallskip
\noindent \underline{\emph{Step 4}}.
Section~\ref{sec:ergodicity} establishes a limit theorem for certain additive functionals of the process $Z=(X,Y)$, which appear naturally in the proof of the CLT (discussed in Step~5) for the first component in Theorem~\ref{th:main1}. 
The key challenge  is that these additive functionals depend on the superdiffusive component $X$ as well as on the ergodic process $Y/b(X)$.
The strategy of the proof is as follows: we split the additive functional into a sum of integrals over shorter time periods, each of which can be can be controlled via the coupling from Step~3 above (cf.~Proposition~\ref{prop:seq}) between the rescaled process $\mathcal{Z}^T$ and the stationary version of the process $\mathcal{Z}$, together with an application of the ergodic theorem to the process $\mathcal{Z}$.

\smallskip
\noindent  \underline{\emph{Step 5}}.
 The second-order behaviour of $X_T$ as $T\to\infty$ is analysed in Section~\ref{sec:central}. We apply the Lyapunov function from \cite{MMW} to the process $Z$ and use It\^o's formula to find the finite variation and local martingale parts of the transformed process. The finite variation part, including the local-time term driven by $L$, is controlled via  the laws of large numbers for $X_T$ and $L_T$ in~\cite{MMW}. The martingale part leads to the Gaussian contribution in the central limit theorem. The growth rate of its quadratic variation, which is one of the functionals analysed in Step 4,  characterises the asymptotic variance 
 of the first component in the limit of Theorem~\ref{th:main1}.

\section{The limiting process in the infinite cylinder}
\label{sec:limiting}

\subsection{General properties}
\label{sec:cylinder-sde}

In this section we introduce a process $\cZ^{\infty,\mathpzc{z}}$
on the infinite cylinder $\R \times \BB^d$, that will serve as a large-time local approximation to the process $Z$ satisfying~\eqref{eq:SDER}.  
For any  $\mathpzc{z}\in \R\times \BB^d$,
consider the process $\cZ^{\infty,\mathpzc{z}}=(\cX^{\infty,\mathpzc{z}},\cY^{\infty,\mathpzc{z}})$ satisfying
\begin{equation}
\label{eq:limitEqJoint_Integral}
      \cZ^{\infty,\mathpzc{z}}_t =(0,\mathpzc{y})+\int_0^t \sigma^\infty \d W_s+\int_0^t(s_0,\phi^{(d)}_\infty(\cY^{\infty,\mathpzc{z}}_s)) \d L^{\cY}_s, \qquad t\in\RP,
\end{equation}
where $\sigma^\infty$ is a square root of the positive matrix $\Sigma^\infty\in\cM^{+}_{1+d}$ in Assumption~\eqref{hyp:C}, $\phi^{(d)}_\infty$ is the vector field in~\eqref{hyp:V+} mapping $\Sp^{d-1}$ into $\R^d$  and $W$ is a standard Brownian motion in $\R^{1+d}$. The process $\cZ^{\infty,\mathpzc{z}}$
is indexed by $\mathpzc{z}=(\mathpzc{x},\mathpzc{y})\in \R\times \BB^d$,
but it only depends on $\mathpzc{z}$ through the initial condition
$\cZ^{\infty,\mathpzc{z}}_0 = (0,\mathpzc{y})$; 
nevertheless, we retain the $\mathpzc{x}$ (as a component of $\mathpzc{z}$) for compatibility with other notation as it facilitate the comparison in Section~\ref{sec:local} with the rescaled and recentred process $\mathcal{Z}^T$ defined in~\eqref{eq:Z^T-def} above.\footnote{Cursive script $\mathpzc{x}, \mathpzc{y},\mathpzc{z}, \cX,\cY,\cZ,\dots$ will later be used for the ``rescaled'' quantities, while Roman letters $x,y,z,X,Y,Z,\dots$ will be used for the ``original'' quantities. By extension, in Section~\ref{sec:limiting} we also use cursive script for the limits of the rescaled quantities.}  Sometimes it is practical to consider a strong solution to \eqref{eq:limitEqJoint_Integral} driven by a given Brownian motion $W$, in which case this Brownian motion is written as an extra superscript (e.g.~$\cZ^{\infty,\mathpzc{z},W}$).
The process $\cY^{\infty,\mathpzc{z}}$ satisfies the SDER in~\eqref{eq:limitEq}, started at $\mathpzc{y}$, and 
the local time $L^{\cY}$ of
$\cY^{\infty,\mathpzc{z}}$ at $\partial \BB^d=\Sp^{d-1}$ clearly equals the local time $L^{\cZ}$ of
$\cZ^{\infty,\mathpzc{z}}$ at $\partial (\R\times \BB^d)=\R\times \Sp^{d-1}$.

\begin{lemma}
    \label{le:strongSol}
    Suppose that~\eqref{hyp:C} and~\eqref{hyp:V+} hold. Then, for any  $\mathpzc{z}=(\mathpzc{x},\mathpzc{y})\in \R\times \BB^d$, SDER~\eqref{eq:limitEqJoint_Integral} admits a unique strong solution $\cZ^{\infty,\mathpzc{z}}=(\cX^{\infty,\mathpzc{z}},\cY^{\infty,\mathpzc{z}})$.
\end{lemma}
\begin{proof}
We first consider 
$\cY^{\infty,\mathpzc{y}}$; restriction of SDER~\eqref{eq:limitEqJoint_Integral} shows that $\cY^{\infty,\mathpzc{y}}$ ought to satisfy  SDER~\eqref{eq:limitEq} started at $\mathpzc{y}\in\BB^d$,
and thus we construct $\cY^{\infty,\mathpzc{y}}$ first. 
  Since the domain $\BB^d$ is compact and all the data (the boundary of the domain, the diffusion matrix, and the reflection vector field) are $\cC^1$, we can for example apply Corollary~5.2 (case 2) from~\cite{Ishii}, the result of which is indeed strong existence and uniqueness for
  SDER~\eqref{eq:limitEq}.
  
  Let $\cY^{\infty,\mathpzc{y},W }$ be the unique strong solution to~\eqref{eq:limitEq}, and let $L^{\cY}$ be the associated local time. 
Denote by $\pi_0: \R\times \R^d\to \R$ the canonical projection on the first coordinate $\pi_0(x,y)=x$.
  Then $\cZ^{\infty,\mathpzc{z},W}\coloneqq ( \pi_0 \sigma^\infty W+ s_0  L^{\cY  },  \cY^{\infty,\mathpzc{y},W })$ is by construction a strong solution to \eqref{eq:limitEqJoint_Integral}, associated with the same local time. This establishes existence of the solution of SDER~\eqref{eq:limitEqJoint_Integral}.
 
  Furthermore, if $\widetilde{\cZ}^{\infty,\mathpzc{z},W}=(\widetilde{\cX}^{\infty,\mathpzc{z},W},\widetilde{\cY}^{\infty,\mathpzc{z},W})$ is any solution to \eqref{eq:limitEqJoint_Integral} with local time $\widetilde{L}$, then $\widetilde{\cY}^{\infty,\mathpzc{z},W}$
   is a strong solution to  \eqref{eq:limitEq} with local time $\widetilde{L}$. Hence $(\widetilde{\cY}^{\infty,\mathpzc{z},W},\widetilde{L})$ is almost surely equal to $(\cY^{\infty,\mathpzc{y},W }, L^{\cY})$. Moreover, since $\widetilde{\cZ}^{\infty,\mathpzc{z}}$ solves \eqref{eq:limitEqJoint_Integral}, we have
   \[  \widetilde{\cX}^{\infty,\mathpzc{z}}_t= \pi_0 \widetilde{\cZ}^{\infty,\mathpzc{z}}_t=\pi_0 \sigma^\infty W_t +s_0  \widetilde{L}_t
   =\pi_0 \sigma^\infty W_t +s_0  {L}^{\cY}_t
   =\cX^{\infty,\mathpzc{z}}_t \quad\text{
   for all $t\in\RP$,}\]
implying $\widetilde{\cX}^{\infty,\mathpzc{z}}={\cX}^{\infty,\mathpzc{z}}$ almost surely. This establishes uniqueness.
\end{proof}

\begin{remark}[Notation] 
\label{rem:Notation}
When we want to stress that $\cZ^{\infty,\mathpzc{z}}$ is a strong solution of SDER~\eqref{eq:limitEqJoint_Integral}
associated with specific Brownian motion $W$ (see Lemma~\ref{le:strongSol} above), we use the notation $\cZ^{\infty,\mathpzc{z},W}=\cZ^{\infty,\mathpzc{z}}$. 
Furthermore, recalling that there is no dependence on $\mathpzc{x}$ in~\eqref{eq:limitEqJoint_Integral}, we will often write $\cY^{\infty,\mathpzc{y}}$ for $\cY^{\infty,\mathpzc{z}}$, which is itself a strong Markov process on $\BB^d$, satisfying  SDER~\eqref{eq:limitEq}.
 We write $\cZ^{\infty,\nu}$  and $\cY^{\infty,\nu}$ for the process $\cZ^{\infty,(0,\tilde Y_0),W}$ and $\cY^{\infty,\tilde Y_0,W}$, respectively where $\tilde Y_0$ is random, distributed according to a probability measure $\nu$ supported on $\BB^d$ and independent from $W$. 
For any $t\in\RP$, $\mathpzc{y} \in \BB^d$, and Borel set $A\subseteq\BB^d$,
let  $P^\infty_t(\mathpzc{y},A)\coloneqq \mathbb{P}( \cY^{\infty,\mathpzc{y}}_t\in A)$ 
be the Markov kernel associated to $\cY^{\infty,\mathpzc{y}}$.
\end{remark}

The main result of this subsection is the following minorization condition,
which establishes that $\BB^d$ is a \emph{small set} in the sense of~\cite[p.~111]{MCMS} for $P^\infty_t$ for all  $t>0$, i.e., 
$
\inf_{\mathpzc{y}\in \BB^d} P^\infty_t (\mathpzc{y},A) \geq \xi (A)$ for every Borel $A \subseteq \BB^d$
 and some non-trivial measure $\xi$.

\begin{proposition}
    \label{prop:small-set}
    Suppose that~\eqref{hyp:C} and~\eqref{hyp:V+} hold.   For every $\epsilon>0$,
there exists a measure $\xi\neq 0$ supported on $\BB^d$ such that 
    \begin{equation}
    \label{eq:small_set}
 \inf_{t \in[\epsilon,\infty)}   \inf_{\mathpzc{y}\in \BB^d} P^\infty_t (\mathpzc{y},A) \geq \xi (A), \quad \text{for every Borel } A \subseteq \BB^d.
    \end{equation}
    \end{proposition}
\begin{remark}
\label{rem:small-set}
    We expect that in~\eqref{eq:small_set} one can take $\xi(A)$ to be a constant multiple of Lebesgue measure, to show a  Doeblin-type mixing condition. However, we only need the (weaker) small set property, for which the proof is simpler.
    It is clear from the proof of Proposition~\ref{prop:small-set} at the end of the present subsection that the measure $\xi$ depends on the choice of $\epsilon>0$. 
\end{remark}
Before giving the proof of Proposition~\ref{prop:small-set} via  Lemmas~\ref{le:minorization} and~\ref{le:reach-core} below, we use it to deduce 
existence and uniqueness of the stationary measure 
and show that $\cY^{\infty,\mathpzc{y}}$ is in fact  uniformly  ergodic. 
Recall  the total variation distance $d_{\TV}(\nu_1, \nu_2)\coloneqq\sup\{|\nu_1(A)-\nu_2(A)|:\text{$A$ Borel}\}$.

\begin{corollary} 
  \label{cor:mu-existence-convergence} 
    Suppose that~\eqref{hyp:C} and~\eqref{hyp:V+} hold.   
There exists a unique probability measure $\mu$ supported on $\BB^d$, which is invariant for the kernel $P^\infty$ associated to SDER~\eqref{eq:limitEq} and defined above, i.e.~$\mu P^\infty_t=\mu$ for all $t\in\R_+$. 
There exist constants $\lambda\in(0,1)$, $C_0\in(0,\infty)$, such that for 
all probability distributions $\nu$ supported on $\BB^d$ we have
  \[ d_{\TV} ( \nu P^\infty_{t} , \mu ) \leq  C_0 \lambda^t\quad\text{for all $t\in\RP$,}\quad\text{where }\nu P^\infty_{s}(\d \mathpzc{y}) \coloneqq\int_{\BB^d} \nu(\mathrm{d} y')P^\infty_{s}(y',\mathrm{d} \mathpzc{y}).\]
\end{corollary}


\begin{remarks}
\
\begin{enumerate}
    \item[(a)]
 In the special case of Remark~\ref{rem:BM_limit_on_disk} above, when the projection in the $y$-direction of the reflection vector field $\phi$ is asymptotically normal ($\phi^{(d)}_\infty(u)=-c_0 u$ for $u\in\mathbb{S}^{d-1}$) and $\sigma^\infty$ is the identity matrix, the uniform ergodicity in Corollary~\ref{cor:mu-existence-convergence} (i.e.~independence of the constant $C_0$ on the initial condition $\nu$) is well known~\cite{Burdzy2006,MR4064292}. 
  \item[(b)] In the general case when the vector field $\phi$ satisfies the assumptions of Corollary~\ref{cor:mu-existence-convergence}, the results in~\cite[Thm~3 \& Cor.~4]{kang2014characterization} suggest that a density of the measure $\mu$ exists and satisfies the adjoint linear, second-order partial differential equation (PDE) on the ball $\mathbb{B}^d$ with Neumann boundary conditions. Formalising this is beyond the scope of the paper, as our primary interest here is the existence and uniqueness of the invariant probability measure $\mu$. However, we note that this PDE is central for the quantification of the asymptotic variance $\Upsilon$ in~\eqref{eq:N-var} of our CLT in Theorem~\ref{th:main1}. 
\end{enumerate}
\end{remarks}

\begin{proof}[Proof of Corollary~\ref{cor:mu-existence-convergence}]
By Proposition~\ref{prop:small-set} above, for any $t\in(0,\infty)$, $\BB^d$ is a small set for $P^\infty_t$.
By~\cite[Thm~8.7]{MCMS},  there exists a measure $\mu_t$, satisfying  $\mu_tP^\infty_{nt}=\mu_t$ for all $n\in\N$, and a constant $\lambda_t\in(0,1)$ such that for all probability measures $\nu$ supported on $\mathbb{B}^d$, we have
\[ 
\d_{TV}(\nu P^\infty_{nt},\mu_t )\leq 2 \lambda_t^n\quad\text{for all $n\in\N$.}
\]
The measures $(\mu_t )_{t>0}$ are in fact all equal (i.e.~they do not depend on the parameter $t$). Indeed, fix arbitrary times $t,t'>0$ and pick any $\delta>0$. Let  $m\in\N$ be such that $2 \lambda_{t'}^m<\delta$ and let $n\in\N$ satisfy $nt>mt'$. Then, since 
    $\mu_t P_{nt}^\infty=\mu_t$ for all $n\in\N$ and $d_{\TV}(\mu_t P^\infty_s,\mu_{t'}P^\infty_s)\leq d_{\TV}(\mu_t,\mu_{t'}) $ for all $s\in\RP$,
  the semigroup property of $P^\infty$ and the definition of $\mu_{t'}$ imply
    \[ 
    d_{\TV}(\mu_t,\mu_{t'})=   d_{\TV}(\mu_tP^\infty_{nt},\mu_{t'} )=  d_{\TV}((\mu_t P^\infty_{nt-mt'}) P^\infty_{mt'},\mu_{t'} )\leq 2\lambda_{t'}^m\leq \delta.
    \]
Since $\delta >0$ was arbitrary, we conclude that $\mu_t = \mu_{t'}$. Denote  $\mu\coloneqq \mu_t$ for any $t>0$ and note that $\mu P^\infty_t=\mu$ for all $t
\in\RP$. 
Since $\mu=\mu_1$ and the integer part $\lfloor t\rfloor$ of $t>0$ satisfies $t-1\leq\lfloor t\rfloor$, for any probability measure $\nu$ supported in $\BB^d$ and $C_0\coloneqq 2 /\lambda_1$ we have
\[
    d_{\TV}(\nu P^\infty_t,\mu)=
    d_{\TV}(\nu P^\infty_{\lfloor t \rfloor} P^\infty_{t-\lfloor t \rfloor} ,\mu P^\infty_{t-\lfloor t \rfloor} )
    \leq 
    d_{\TV}(\nu P^\infty_{\lfloor t \rfloor},\mu  )
\leq 2 \lambda_1^{\lfloor t \rfloor}\leq C_0 \lambda_1^t.
\qedhere
\]
\end{proof}

We now work towards the proof of Proposition~\ref{prop:small-set}.
The following result gives a lower bound on $P_t^\infty(\mathpzc{y}, \, \cdot\, )$, as is required for~\eqref{eq:small_set},
but only
for starting points $\mathpzc{y}$ not too close to the boundary.
An estimate to show that the process exits quickly a neighbourhood of the boundary,
Lemma~\ref{le:reach-core} below, will provide the other key ingredient in our proof of Proposition~\ref{prop:small-set}.

\begin{lemma}
  \label{le:minorization}
    Assume~\eqref{hyp:C} and~\eqref{hyp:V+} hold. Pick $\epsilon\in(0,\infty)$.
  Let $C\coloneqq r\BB^d$ for a given $r\in(0,1)$. Then there exists a  measure $\xi_r\neq 0$ supported on $\BB^d$  such that for any Borel measurable $A\subseteq \BB^d$,
  \[
  \inf_{\mathpzc{y}\in C } \inf_{t\in [\epsilon/2,\epsilon]} P^\infty_t(\mathpzc{y},A) \geq \xi_r(A).
  \]
\end{lemma}
\begin{proof}
Recall $\pi_d: \R\times\R^d\to \R^d$, $\pi_d(x,y)=y$.
  On the event
  $\{\forall s\in [0,t],\ \mathpzc{y}+\pi_d \sigma^{\infty}  W_s\in  \BB^d\}$, the processes $\cY^{\infty,\mathpzc{y}}$ and $\mathpzc{y}+ \pi_d \sigma^{\infty} W$ are almost surely equal on $[0,t]$, since $\mathpzc{y}+ \pi_d \sigma^{\infty} W$ exits the (closed) ball $\BB^d$ immediately after its first hitting time of $\Sp^{d-1}$. Thus, it suffices to compare $\cY^{\infty,\mathpzc{y}}$ with the process $\mathpzc{y}+\pi_d\sigma^{\infty}  W$ killed at the boundary $\partial \BB^d$. 
  Recall that $\Sigma^\infty\in\cM_{1+d}^+$ is positive definite  by~\eqref{hyp:C}, implying by Sylvester's criterion that the principal submatrix $\pi_d \Sigma^{\infty}  \pi_d^\tra $ is also positive definite.
  Let $\hat{\sigma}^\infty \in \cM^{+}_{d}$ be such that $\hat{\sigma}^\infty ( \hat{\sigma}^\infty)^\tra = \pi_d \Sigma^{\infty}  \pi_d^\tra $.  Then the process $\pi_d \sigma^{\infty}  W$ is equal in distribution to 
   $\hat{\sigma}^\infty W'$, for a $d$-dimensional standard Brownian motion $W'$. 
   Let $D$ be the convex set \[ D\coloneqq(\hat{\sigma}^\infty)^{-1}  {\BB^d}.\]
  
  For  $t>0$ and all $u, v \in D$, let  $p^D_t(u,v)$ be the Dirichlet heat kernel associated with $D$, that is the unique function continuous in $v$ such that for all $t>0$, all $u \in D$, and all Borel $A \subset \BB^d$,
  \begin{equation}
  \label{eq:dirichlet-kernel}
\mathbb{P}( u + W_t' \in A \mbox{ and } \forall s\in [0,t], u +W_s' \in D ) = \int_A p^D_t(u,v) \d v,
  \end{equation}
  where $W'$ is a standard Brownian motion in $\R^d$ started from $0$.
Let also 
\begin{equation}
\label{eq:delta-D-defs}
\delta_u \coloneqq \inf\{ | u -v |_d : v \in \partial D\}, ~\text{for}~ u \in \BB^d, \text{ and } \delta(A) \coloneqq \inf \{\delta_u : u \in A \}.
  \end{equation}
  Note $\delta(A)>0$ if the closure of  $A$ is contained in the interior of $D$, e.g.~if $A=(\hat{\sigma}^\infty)^{-1}C$. Moreover, for $u, v \in D$ we have 
  $|u-v|_d^2=|(\hat{\sigma}^\infty)^{-1}a-(\hat{\sigma}^\infty)^{-1}b|_d^2\leq \ell_\infty |a-b|_d^2\leq
  4 \ell_\infty$
  for some $a,b\in\BB^d$, where $\ell_\infty$ denotes the largest eigenvalue of the positive definite matrix $(\hat{\sigma}^\infty)^{-1}$.

Fix arbitrary $\epsilon>0$. Then the transition density 
 $p_t (u,v) \coloneqq (2\pi t)^{-d/2} \exp \bigl( - \frac{| u-v|_d^2}{2t} \bigr)$ of a standard Brownian motion on $\R^d$ satisfies
$$p_t (u,v) \geq (2\pi\epsilon)^{-d/2} \mathrm{e}^{-4\ell_\infty/\epsilon} =: q_d\qquad\text{for all $u, v \in D$ and $t \in [\epsilon/2,\epsilon]$.}$$
It is known (see e.g.~\cite{Serafin}) that
  \[
p^D_t(u,v)\geq p_t (u,v) f(t,\min( \delta_u,\delta_v )), \text{ for all } u, v \in D \text{ and } t \in \RP, 
  \]
  where  $f:\RP \times \RP\to \RP $ is a continuous function satisfying  $f>0$ on $(0,\infty)\times(0,\infty)$, with slices $\delta\mapsto f(t,\delta)$ which are non-decreasing for every $t>0$.  We thus obtain
\begin{equation}
      \label{eq:kernel-bound}
p^D_t(u,v)\geq q_d f(t,\min( \delta_u,\delta_v )), \text{ for all } u, v \in D \text{ and } t \in [\epsilon/2,\epsilon]. 
  \end{equation}

For a Borel set $A\subset\BB^d$, $t\in \RP$ and $\mathpzc{y}\in \BB^d$,
the definition of $p^D$ in~\eqref{eq:dirichlet-kernel} 
implies
  \begin{align*}
  P^\infty_t(\mathpzc{y},A)
  &\geq \mathbb{P}( \mathpzc{y}+\hat{\sigma}^\infty W'_t\in A \mbox{ and } \forall s\in [0,t], \mathpzc{y}+\hat{\sigma}^\infty  W'_s\in \BB^d )\\
  &=\mathbb{P}( (\hat{\sigma}^\infty)^{-1} \mathpzc{y}+  W'_t\in (\hat{\sigma}^\infty)^{-1} A \mbox{ and } \forall s\in [0,t], (\hat{\sigma}^\infty)^{-1} \mathpzc{y}+ W'_s\in D )\\
  &= \int_{(\hat{\sigma}^\infty)^{-1} A} p^D_t((\hat{\sigma}^\infty)^{-1} \mathpzc{y},v) \d v.
  \end{align*}
  Recall from the statement of the lemma that $r \in (0,1)$, $C = r \BB^d$ and
  set $\delta := \delta( (\hat{\sigma}^\infty)^{-1} C ) >0$.
   The lower bound in~\eqref{eq:kernel-bound} and the monotonicity of $f(t,\cdot)$  imply
  \begin{align*}
\inf_{\mathpzc{y} \in C}  P^\infty_t(\mathpzc{y},A)
  &\geq q_d \int_{(\hat{\sigma}^\infty)^{-1} A} f ( t, \min ( \delta, \delta_v ))   \d v \\
  & \geq 
  q_d \int_{(\hat{\sigma}^\infty)^{-1} A} f_C ( \delta_v )   \d v \quad\text{ for all } t \in [\epsilon/2,\epsilon],
  \end{align*}
  where $f_C ( \delta') := \inf_{t \in [\epsilon/2,\epsilon]} f ( t,  \min ( \delta, \delta' ))$
  is positive for $\delta' > 0$ (since $f(t,\min(\delta,\delta'))>0$ for all $t\in(0,\infty)$).
With the change of variable $b = \hat{\sigma}^\infty v$, we get
\[ 
\inf_{t \in [\epsilon/2,\epsilon]}
\inf_{\mathpzc{y} \in C}  P^\infty_t(\mathpzc{y},A)
\geq \frac{q_d}{\det ( \hat{\sigma}^\infty) } \int_A f_C (\delta_{(\hat{\sigma}^\infty)^{-1} b} ) \d b.
\]
 The measure $\xi_r$ with density $\frac{q_d}{\det(\hat{\sigma}^\infty) } f_C (\delta_{(\hat{\sigma}^\infty)^{-1} }\ \cdot \ )$ with respect to the Lebesgue measure thus satisfies the minorization property stated in the lemma.
\end{proof}

The next lemma shows that the process $\cY^{\infty,\mathpzc{y}}$ visits the ball $\tfrac{1}{2}\BB^d$ before any time $\epsilon>0$ with positive probability, uniformly bounded from below for all starting points in $\mathpzc{y}\in\BB^d$.

\begin{lemma}
  \label{le:reach-core}
    Suppose that~\eqref{hyp:C} and~\eqref{hyp:V+} hold. For every $\epsilon\in(0,\infty)$,
  there exists $\epsilon_0>0$, depending only on $\Sigma^\infty$ and $\epsilon$,
  such that
  \[ \inf_{\mathpzc{y}\in  \BB^d} \mathbb{P}\big( \exists t\in [0,\epsilon]: \cY^{\infty,\mathpzc{y}}_t\in  \tfrac{1}{2}\BB^d \big) \geq \epsilon_0.
  \]
\end{lemma}
\begin{proof} 
  Define $f: \BB^d \to [0,1]$ by $f(\mathpzc{y}) :=1-|\mathpzc{y}|_{d}^2$ and
  fix arbitrary $\mathpzc{y} \in \BB^d$.   
  Recall from ~\eqref{hyp:C} that
  $\sum_{i=1}^d \Sigma_{i,i}^\infty = \os$ and from Assumption~\eqref{hyp:V+} that $\langle u,\phi^{(d)}_\infty(u)\rangle \leq 0$ for all $u \in \Sp^{d-1}$.
  Since $f$ is $\cC^2$, by It\^o's formula and SDER~\eqref{eq:limitEq}  we obtain
  \begin{align*}
  f( \cY^{\infty,\mathpzc{y}}_t)
  &=f(\mathpzc{y}) -2 \int_0^t \langle \cY^{\infty,\mathpzc{y}}_s, \pi_d\sigma^\infty \d W_s\rangle -2 \int_0^t \langle\cY^{\infty,\mathpzc{y}}_s,\phi^{(d)}_\infty(\cY^{\infty,\mathpzc{y}}_s)\rangle \d L^{\cY^{\infty,\mathpzc{y}} }_s- t \sum_{i=1}^d \Sigma^\infty_{i,i} \\
  &\geq -2 \int_0^t \langle \cY^{\infty,\mathpzc{y}}, \pi_d\sigma^\infty \d W_s\rangle -   \os t =: M_t- \os t \quad{\text{for all $t\in\RP$,}}
  \end{align*}
  where $\pi_d: \R^{1+d}\to \R^d$ is the canonical projection.
  By the Dambis--Dubins--Schwarz theorem there exists a Brownian motion $B$ such that  the continuous local martingale $M$ is equal to $B_{[ M] }$ with quadratic variation $[ M]=4\int_0^\cdot |(\pi_d\sigma^\infty)^\tra \cY^{\infty,\mathpzc{y}}_s|_d^2 \d s$.

Pick any $\epsilon>0$.  Define $\tau := \inf\{ t\in\RP :  \cY^{\infty,\mathpzc{y}}_t\in \tfrac{1}{2}\BB  \}\in [0,\infty]$ (with convention $\inf\emptyset=
  \infty$). On the event $\{t\leq \tau\}$ we have  $|\cY^{\infty,\mathpzc{y}}_s|_d^2\geq \frac{1}{4}$ for all $s \in [0,t]$, implying $[ M]_t\geq \ell_0 t$,
  where $\ell_0>0$ is the smallest eigenvalue of the symmetric matrix $\pi_d\sigma^\infty(\pi_d\sigma^\infty)^\tra=\pi_d \Sigma^\infty \pi_d^\tra$, which is positive by Sylvester's criterion applied to the matrix $\Sigma^\infty \in \cM^{+}_{1+d}$ from Assumption~\eqref{hyp:C}. 
Thus,
\begin{align*}
\{\tau \geq \epsilon\}
& \subseteq \bigl\{[ M]_\epsilon\geq \epsilon \ell_0 \bigr\}\cap \bigl\{\forall t \in [0,\epsilon],  f( \cY^{\infty,\mathpzc{y}}_t )\leq 3/4 \bigr\}\\
&\subseteq \bigl\{[ M]_\epsilon\geq \epsilon\ell_0 \bigr\}\cap \bigl\{ \forall t \in [0,\epsilon], B_{[ M]_t}\leq 3/4+ \os t \bigr\}\\
& \subseteq \bigl\{[ M]_\epsilon\geq \epsilon \ell_0 \bigr\} \cap \bigl\{  \sup_{t \in [0,\epsilon]} B_{[ M]_t}\leq 3/4+ \os \epsilon\bigr\}\subseteq \bigl\{ B_{\epsilon \ell_0}\leq  3/4+ \os\epsilon  \bigr\},
\end{align*}
where the last inclusion uses the continuity of the quadratic variation $[ M]$.
$\mathbb{P}(B_{\epsilon \ell_0}\leq  3/4+ \os\epsilon )$
 is the probability
that a standard Gaussian random variable is not more than $(3/4 +\os\epsilon)/\sqrt{\epsilon\ell_0}$, and does not depend on the starting point $\mathpzc{y}\in\BB^d$.
Setting 
$\epsilon_0\coloneqq 1-\mathbb{P}(B_{\epsilon \ell_0}\leq  3/4+ \os\epsilon )>0$, the inclusion above implies
$\mathbb{P}(\tau\leq \epsilon)=1-\mathbb{P}(\tau> \epsilon)\geq \epsilon_0$ for all $\mathpzc{y}\in\BB^d$, establishing the lemma.
\end{proof}

Equipped with Lemmas~\ref{le:minorization} and~\ref{le:reach-core}, we can complete the proof of Proposition~\ref{prop:small-set}.

\begin{proof}[Proof of Proposition~\ref{prop:small-set}]
Let $\epsilon>0$.
For each $t \in \RP$ and $\mathpzc{y} \in \BB^d$, 
define the stopping time  $\tau_t^\mathpzc{y} := \inf\{ s \geq t :  \cY^{\infty,\mathpzc{y}}_s \in \tfrac{1}{2}\BB^d  \}$ (with convention $\inf\emptyset=\infty$), taking values in $[t,\infty]$,
and let
    $\mathbb{P}^{\tau,\cY^\infty}_{\mathpzc{y},t}$ denote the probability measure on $[t,t+\epsilon/2] \times \tfrac{1}{2}\BB^d$ given by
    \[ 
   \mathbb{P}^{\tau,\cY^\infty}_{\mathpzc{y},t}\coloneqq \mathbb{P}\bigl( (\tau_t^\mathpzc{y},\cY^{\infty,\mathpzc{y}}_{\tau_t^\mathpzc{y}}) \in \cdot \mid \tau_t^\mathpzc{y} \leq t+\epsilon/2 \bigr).\]
Here we used the fact $\mathbb{P}(\tau^\mathpzc{y}_t \leq t+\epsilon/2)>0$, which holds  by the Markov property at time $t$ and Lemma~\ref{le:reach-core}.
For every $t\in\RP$ and Borel set $A \subseteq \BB^d$,
the strong Markov property of $\cY^{\infty,\mathpzc{y}}$ (see~Theorem~A.1 of~\cite{MMW}) applied at the stopping time $\tau_t^\mathpzc{y}$ yields
    \begin{align*}
    \mathbb{P}(\cY^{\infty,\mathpzc{y}}_{t+\epsilon} \in A)
    &\geq \mathbb{P}(\tau^\mathpzc{y}_t \leq t+\epsilon/2 \text{ and }  \cY^{\infty,\mathpzc{y}}_{t+\epsilon} \in A)\\
    &=\mathbb{P}(\tau^\mathpzc{y}_t \leq t+\epsilon/2) \int_{[t,t+\epsilon/2]\times \tfrac{1}{2}\BB^d} \mathbb{P}(\cY^{\infty,\mathpzc{y}'}_{t+\epsilon-s} \in A) \d \mathbb{P}^{\tau,\cY^{\infty}}_{\mathpzc{y},t} (s,\mathpzc{y}')\\
    &\geq \mathbb{P}(\tau_t^\mathpzc{y}\leq t+\epsilon/2) 
    \inf_{u \in [\epsilon/2,\epsilon],\,  \mathpzc{y}'\in  \tfrac{1}{2}\BB^d}  \mathbb{P}(\cY^{\infty,\mathpzc{y}'}_{u} \in A) \\
    &= \mathbb{P}(\tau_t^\mathpzc{y}\leq t+\epsilon/2) 
    \inf_{u \in [\epsilon/2,\epsilon],\,  \mathpzc{y}'\in  \tfrac{1}{2}\BB^d}  P^\infty_u(\mathpzc{y}', A) \geq   \mathbb{P}(\tau_t^\mathpzc{y}\leq t+\epsilon/2)\xi_{1/2}(A),
    \end{align*}
    where $\xi_{1/2}$ is the measure in Lemma~\ref{le:minorization} (for $r=1/2$ and the present~$\epsilon$). The Markov property at time~$t$ and Lemma~\ref{le:reach-core} (applied with $\epsilon/2$)
    imply that there exists $\epsilon_0>0$ such that  $\mathbb{P}(\tau_t^\mathpzc{y}\leq t+\epsilon/2) \geq \epsilon_0$
    for all $\mathpzc{y} \in \BB^d$ and $t \in \RP$. Thus setting $\xi (A) := \epsilon_0 \xi_{1/2} (A)$ yields~\eqref{eq:small_set}.
    \end{proof}

\subsection{Coupling the process in the cylinder to its stationary version}
\label{sec:cylinder-coupling}
Let $\nu$ be an arbitrary probability measure supported in $\BB^d$. Recall that $\mu$ in Corollary~\ref{cor:mu-existence-convergence}
is the invariant measure of $\mathcal{Y}^{\infty,\mathpzc{y}}$ and, from Remark~\ref{rem:Notation} above, that $\cZ^{\infty, \nu}$ and $\cZ^{\infty, \mu}$
are solutions of SDER~\eqref{eq:limitEqJoint_Integral}, where $\mathcal{Y}^{\infty,\nu}_0$ and $\mathcal{Y}^{\infty,\mu}_0$
follow distributions $\nu$ and $\mu$, respectively.
In the present subsection we construct a coupling of the processes $\cZ^{\infty, \nu}$ and $\cZ^{\infty, \mu}$, i.e.~a probability measure on a measurable space that supports a pair $(\widetilde \cZ^{\infty, \nu},\widetilde\cZ^{\infty, \mu})$, such that   the processes $\widetilde \cZ^{\infty, \nu}$ and $\widetilde\cZ^{\infty, \mu}$ follow the laws of $\cZ^{\infty, \nu}$ and $\cZ^{\infty, \mu}$, respectively. 
If there is no ambiguity, we abuse the notation slightly by referring to the coupled processes again as 
$(\cZ^{\infty, \nu},\cZ^{\infty, \mu})$.
Recall also from Remark~\ref{rem:Notation} above that  $\cZ^{\infty, \mu,\mathcal{W}'}$  denotes the strong solution of SDER~\eqref{eq:limitEqJoint_Integral}, driven by a Brownian motion $\cW'$ with the $\cY$-component started according to the stationary law $\mu$ of Corollary~\ref{cor:mu-existence-convergence}.

\begin{proposition}
  \label{prop:coupling}
  Let $\tilde Y_0$ be a random vector following 
  a probability law $\nu$ supported on $\BB^d$ and let $\cW$ be a Brownian motion independent of $\tilde Y_0$. Denote by $\cZ^{\infty,\nu,\cW}$ the strong solution of SDER~\eqref{eq:limitEqJoint_Integral} started at the random initial point $(0,\tilde Y_0)$ and driven by $\cW$.
  Recall that the measure $\mu$, supported on $\BB^d$, is invariant for  SDER~\eqref{eq:limitEq} and let the constants $\lambda\in(0,1)$ and $C_0\in(0,\infty)$ satisfy the conclusion of Corollary~\ref{cor:mu-existence-convergence} above.
  
  For arbitrary
  $\mathpzc{s}\in(0,\infty)$, 
  there exists a coupling $\mathbb{P}^{\mathpzc{s}}$ of two Brownian motions $\mathcal{W}, \mathcal{W}'$, which extends the initial probability space  $(\Omega_\mathrm{init}, \mathcal{F}_\mathrm{init},\mathbb{P}_{\mathrm{init}})$ supporting $(\tilde Y_0, \cW)$, such that: 
\begin{enumerate}  
\item  the processes $\cZ^{\infty, \nu,\mathcal{W}}$ started at $(0,\tilde{Y}_0)$ and the solution $\cZ^{\infty, \mu,\mathcal{W}'}$  of SDER~\eqref{eq:limitEqJoint_Integral}  satisfy
  \begin{equation}
  \label{eq:bound39}
  \mathbb{P}^{\mathpzc{s}}(\exists x\in \R: \forall t\geq \mathpzc{s}, \cZ^{\infty,\nu,\mathcal{W}}_t-\cZ^{\infty,  \mu,\mathcal{W}'}_t=  (x, 0_{\R^d}  ) )\geq 1-2 C_0 \lambda^\mathpzc{s};
  \end{equation}
    \item the increments of $\mathcal{W}$ and $\mathcal{W}'$ after the time $\mathpzc{s}$ are equal:  
   $\mathbb{P}^\mathpzc{s}((\mathcal{W}'_t- \mathcal{W}'_\mathpzc{s})_{t\in[\mathpzc{s},\infty)}=(\mathcal{W}_t- \mathcal{W}_\mathpzc{s})_{t\in[\mathpzc{s},\infty)})=1$ (making 
  $(\mathcal{W}'_t- \mathcal{W}'_\mathpzc{s})_{t\in[\mathpzc{s},\infty)}$ and $(\tilde Y_0,(\mathcal{W}_t)_{t\in[0,\mathpzc{s}]})$  independent); 
  \item under $\mathbb{P}^{\mathpzc{s}}$, $\tilde Y_0$ is independent from the couple 
  $(\mathcal{W}', 
  \cZ^{\infty, \mu,\mathcal{W}'}_0 )$;
  \item under $\mathbb{P}^{\mathpzc{s}}$, given $(\tilde Y_0,\cW)$, the couple 
  $(\mathcal{W}', \cZ^{\infty, \mu,\mathcal{W}'}_0 )$ is independent of $\mathcal{F}_\mathrm{init}$. 
  \end{enumerate}
\end{proposition}

\begin{proof}
It suffices to construct a 
quadruple $(\tilde{Y}_0, \hat{Y}_0,\cW,\cW')$
on some probability space, such that
$\mathcal{W}$ and $\mathcal{W}'$
are coupled Brownian motions and  $\tilde{Y}_0$, $\hat{Y}_0$ are random vectors in $\BB^d$, distributed respectively as $\nu$ and $\mu$, satisfying properties (1), (2) and~(3) in the proposition (with $\mathcal{Z}^{\infty, \mu,\mathcal{W}'}$ and $\mathcal{Z}^{\infty, \nu,\mathcal{W}}$ following SDER~\eqref{eq:limitEqJoint_Integral}, started from $(0,\hat{Y}_0)$ and $(0,\tilde{Y}_0)$, respectively). 
Indeed, the fact that we can then choose this coupling to extend the initial probability space 
$(\Omega_\mathrm{init}, \mathcal{F}_\mathrm{init},\mathbb{P}_{\mathrm{init}})$
 on which $(\tilde{Y}_0,\cW)$ is originally defined, in such a way that property (4) also holds, follows directly by constructing a regular conditional probability of the coupling (with respect to $(\tilde Y_0,\cW)$)~\cite[Thm~2.3]{Kallenberg} and applying the extension lemma~\cite[Lem.~6.9]{Kallenberg}.
 The remainder of the proof is dedicated to the construction of the coupling  $(\tilde{Y}_0, \hat{Y}_0,\cW,\cW')$ with properties (1), (2) and (3).

By Corollary~\ref{cor:mu-existence-convergence} and the triangle inequality, for any two starting points $\mathpzc{y}_0, \mathpzc{y}_0'\in \mathbb{B}^d$, we have $\mathrm{d}_\TV( \delta_{\mathpzc{y}_0}P^\infty_{\mathpzc{s}}, \delta_{\mathpzc{y}_0'}P^\infty_{\mathpzc{s}})\leq 2 C_0 \lambda^{-\mathpzc{s}}$, where $\delta_{\mathpzc{y}}$ is the Dirac measure at $\mathpzc{y}$. By the existence of maximal couplings under total variation, there exists a probability measure  $\Gamma_{\mathpzc{y}_0,\mathpzc{y}'_0}$, supported on $\mathbb{B}^d\times \mathbb{B}^d$, with marginal distributions 
$\delta_{\mathpzc{y}_0}P^\infty_{\mathpzc{s}}$ and $\delta_{\mathpzc{y}_0'}P^\infty_{\mathpzc{s}}$
and significant mass on the diagonal of the product space $\mathbb{B}^d\times \mathbb{B}^d$:
\begin{equation}
\label{eq:coupling_inequality_marginal}
\Gamma_{\mathpzc{y}_0,\mathpzc{y}'_0}(\{(\mathpzc{y},\mathpzc{y}):\mathpzc{y}\in \mathbb{B}^d\})\geq 1-2   C_0 \lambda^{-\mathpzc{s}}.
\end{equation}
In fact, by~\cite[Lem.~1]{Kulik}, 
$\Gamma_{\mathpzc{y}_0,\mathpzc{y}'_0}(\mathrm{d}\mathpzc{y},\mathrm{d}\mathpzc{y}')$ is a probability kernel (see~\cite[p.~106]{Kallenberg} for definition) on the product 
$\mathbb{B}^d\times \mathbb{B}^d$
with the Borel $\sigma$-field (i.e., for any event $A\subset\mathbb{B}^d\times \mathbb{B}^d$, the map $(\mathpzc{y}_0,\mathpzc{y}'_0)\mapsto \Gamma_{\mathpzc{y}_0,\mathpzc{y}'_0}(A)$ is Borel-measurable). We may thus construct the probability measure $\Gamma$, supported on the product $\mathbb{B}^d\times \mathbb{B}^d\times\mathbb{B}^d\times \mathbb{B}^d$, given by the following formula 
\begin{equation}
\label{eq:def_Gamma_coupling}
\Gamma\coloneqq \int_{ \mathbb{B}^d\times\mathbb{B}^d } \delta_{\mathpzc{y}_0}\otimes \delta_{\mathpzc{y}'_0} \otimes 
	\Gamma_{\mathpzc{y}_0, \mathpzc{y}'_0}\  \nu(\mathrm{d} \mathpzc{y}_0) \mu(\mathrm{d} \mathpzc{y}'_0).
\end{equation}
Under the probability measure $\Gamma$, the projection $(\mathbb{B}^d)^4\to\mathbb{B}^d$ on the first (resp.~second, third, fourth) component follows the law $\nu$ (resp.~$\mu$, $\nu P^\infty_\mathpzc{s}$, $\mu P^\infty_\mathpzc{s}=\mu$); 
see Remark~\ref{rem:Notation} for the definition of the semigroup $P^\infty$.
Moreover, under $\Gamma$, the joint law of the first and third (resp.~second and fourth) components is equal to the law of $(\mathcal{Y}^{\infty,\nu }_0,\mathcal{Y}^{\infty,\nu }_{\mathpzc{s}} )$  (resp. $(\mathcal{Y}^{\infty,\mu }_0,\mathcal{Y}^{\infty,\mu }_{\mathpzc{s}} )$), where $\mathcal{Y}^{\infty,\nu}$ (resp.~$\mathcal{Y}^{\infty,\mu}$) is the solution
of SDER~\eqref{eq:limitEq}.
We shall see below that, under $\Gamma$, the first component is independent from the second and fourth one.

Let the standard Brownian motion $\cW$ be defined as the identity map on the canonical probability space $( \Omega,\mathcal{F}, \mathbb{P})$ with $\Omega= \mathcal{C}_0(\mathbb{R}_+, \mathbb{R}^{1+d})$. 
On the product space $\BB^d\times\Omega$ 
we construct the unique strong solution of SDER~\eqref{eq:limitEq}, started at  $\cY^{\infty,\nu,\cW}_0$
and driven by $\cW$, where the random element $(\cY^{\infty,\nu,\cW}_0,\cW)$ is given as the identity map on $\BB^d\times\Omega$ under the product measure $\nu\otimes\mathbb{P}$.
Then~\cite[Thm~6.3]{Kallenberg}, applied on the product space $\BB^d\times\Omega$ with measure $\nu\otimes\mathbb{P}$, yields a regular conditional probability 
$$\mathbb{P}_{\mathpzc{y}_0,\mathpzc{y}_{\mathpzc{s}}}\coloneqq\nu\otimes\mathbb{P}(\cW\in\ \cdot\ |\cY^{\infty,\nu,\cW}_0=\mathpzc{y}_0,  \ \cY^{\infty,\nu,\cW}_\mathpzc{s}=\mathpzc{y}_\mathpzc{s})\>\text{
for almost every $(\mathpzc{y}_0,\mathpzc{y}_{\mathpzc{s}})\in\BB^d\times\BB^d$.}$$ 
Note that ``almost every'' in the previous display is with respect to the distribution of the pair
$(\cY^{\infty,\nu,\cW}_0,\cY^{\infty,\nu,\cW}_\mathpzc{s})$, 
which is equivalent for example to the probability measure $\nu\otimes \mu$. This fact is not used in the proof of the proposition. In contrast, an important fact in what follows is that 
$\mathbb{P}_{\mathpzc{y}_0,\mathpzc{y}_{\mathpzc{s}}}$ 
is a conditional probability of the measure $\nu\otimes\mathbb{P}$, under which 
$\cY^{\infty,\nu,\cW}_0$ and $\cW$ are independent.


Let the standard Brownian motion 
$(\cW_s')_{s\in[0,\mathpzc{s}]}$ on the time interval $[0,\mathpzc{s}]$ be defined as the identity map on
the probability space $(\Omega',\mathcal{F}', \mathbb{P}')$, where $\Omega'= \mathcal{C}_0([0,\mathpzc{s}], \mathbb{R}^{1+d})$.
As in the previous paragraph, we construct the strong solution of SDER~\eqref{eq:limitEq} using the data given by the coordinates 
$(\cY^{\infty,\mu,\cW'}_0,(\cW'_t)_{t\in[0,\mathpzc{s}]})$ on the product space
$\mathbb{B}^d\times \Omega'$ under the measure $\mu\otimes\mathbb{P}'$.
By~\cite[Thm~6.3]{Kallenberg},  there exists a regular conditional probability on $\mathbb{B}^d\times \Omega'$, such that
$$ \mathbb{P}'_{\mathpzc{y}_0',\mathpzc{y}_{\mathpzc{s}}' }\coloneqq \mu\otimes\mathbb{P}'((\cW_s')_{s\in[0,\mathpzc{s}]}\in\ \cdot\  | \cY^{\infty,\mu,\cW'}_0=\mathpzc{y}_0',  \ \cY^{\infty,\mu,\cW'}_\mathpzc{s}=\mathpzc{y}_\mathpzc{s}')$$
 for almost every $(\mathpzc{y}_0',\mathpzc{y}_{\mathpzc{s}}')\in\BB^d\times\BB^d$.
As in the previous paragraph, 
we note again that $\mathbb{P}'_{\mathpzc{y}_0',\mathpzc{y}_{\mathpzc{s}}'}$ 
is a conditional probability of the measure
$\mu\otimes\mathbb{P}'$
under which $(\cW_s')_{s\in[0,\mathpzc{s}]}$
and 
$\cY^{\infty,\mu,\cW'}_0$
are independent.

On the probability space $ \Omega_0\coloneqq \mathbb{B}^d \times \mathbb{B}^d \times \Omega\times \Omega'$, define the probability measure 
\begin{equation}
\label{eq:definition_coupling_nearly_final}
\mathbb{P}_0\coloneqq \int_{\mathbb{B}^d\times\BB^d\times\BB^d\times\BB^d}   \delta_{\mathpzc{y}_0} 
\otimes \delta_{\mathpzc{y}'_0} \otimes 
\mathbb{P}_{\mathpzc{y}_0,\mathpzc{y}_{\mathpzc{s}} } \otimes 
\mathbb{P}'_{\mathpzc{y}'_0,\mathpzc{y}'_{\mathpzc{s}} }  \Gamma(\d \mathpzc{y}_0, \d \mathpzc{y}'_0, \d \mathpzc{y}_{\mathpzc{s}}, \d  \mathpzc{y}'_{\mathpzc{s}})
\end{equation}
and set the random elements $\tilde{Y}_0$, $\hat{Y}_0$, $\mathcal{W}$ and $(\cW_s')_{s\in[0,\mathpzc{s}]}$
 as the projections onto the first, second, third and fourth  component of 
$ \Omega_0=\mathbb{B}^d \times \mathbb{B}^d\times \Omega\times \Omega'$, respectively.
Then, under $\mathbb{P}_0$, 
 $\mathcal{W}$ is a Brownian motion, $(\cW_s')_{s\in[0,\mathpzc{s}]}$ is a Brownian motion up to time $\mathpzc{s}$, $\tilde{Y}_0$ has distribution $\nu$ and is independent from $\mathcal{W}$, and  
 $\hat{Y}_0$ has distribution $\mu$  and is independent from $(\cW_s')_{s\in[0,\mathpzc{s}]}$. 
 
 We extend $(\cW_s')_{s\in[0,\mathpzc{s}]}$ to a Brownian motion defined on $\mathbb{R}_+$ by setting $\mathcal{W}'_t\coloneqq
\mathcal{W}'_\mathpzc{s}+\mathcal{W}_t-\mathcal{W}_\mathpzc{s}$  for $t\in[\mathpzc{s},\infty)$. 
We now prove that this coupling satisfies  properties (1), (2) and (3) in the proposition. The property (2) is immediate by construction of $\mathcal{W}'$ after the time $\mathpzc{s}$.

The processes $(\cZ^{\infty, \tilde{Y}_0,\mathcal{W}}_{u+\mathpzc{s}})_{u\in \mathbb{R}_+}$  and $(\cZ^{\infty, \hat{Y}_0,\mathcal{W}'}_{u+\mathpzc{s}})_{u\in\mathbb{R}_+}$ 
    satisfy  SDER~\eqref{eq:limitEqJoint_Integral} on the infinite cylinder, driven by the same Brownian motion $ (\mathcal{W}_{u+\mathpzc{s}}-\mathcal{W}_{\mathpzc{s}})_{u\in \mathbb{R}_+}=(\mathcal{W}'_{u+\mathpzc{s}}-\mathcal{W}'_{\mathpzc{s}})_{u\in \mathbb{R}_+}$. By translation invariance of  SDER~\eqref{eq:limitEqJoint_Integral} in the $\mathpzc{x}$ coordinate (cf.~Remark~\ref{rem:beginning}\ref{rem:beginning(a)} above), we deduce that, on the event $\mathcal{Y}^{\infty,\tilde{Y}_0,\mathcal{W}}_{\mathpzc{s}}=\mathcal{Y}^{\infty,\hat{Y}_0,\mathcal{W}'}_{\mathpzc{s}}$, the difference $\cZ^{\infty, (0,\tilde{Y}_0),\mathcal{W}}_t -\cZ^{\infty, (0,\hat{Y}_0),\mathcal{W}'}_t=(\cX^{\infty, (0,\tilde{Y}_0) ,\mathcal{W}}_\mathpzc{s} - \cX^{\infty, (0, \hat{Y}_0),\mathcal{W}'}_\mathpzc{s},0_{\R^d})$ is constant for $t\in[\mathpzc{s},\infty)$ (since on this event we have $\mathcal{Y}^{\infty,\tilde{Y}_0,\mathcal{W}}_{t}=\mathcal{Y}^{\infty,\hat{Y}_0,\mathcal{W}'}_t$ for all $t\in(\mathpzc{s},\infty)$).
    The
coupling $\mathbb{P}_0$ thus satisfies property (1) in the proposition:
\begin{align*}
\mathbb{P}_0&\big(\exists x\in \R: \forall t\geq \mathpzc{s}, \cZ^{\infty,(0,\tilde{Y}_0),\mathcal{W}}_t-\cZ^{\infty,  (0,\hat{Y}_0),\mathcal{W}'}_t=  (x, 0_{\R^d}  ) \big)\\
& =\Gamma(\{(a,b,c,c): a,b,c\in \mathbb{B}^d \} )\\
&\geq \inf_{a,b\in\BB^d} \Gamma_{a,b}(\{(c,c): c\in \mathbb{B}^d \})
\geq 1-2 C_0 \lambda^\mathpzc{s},
\end{align*}
where the last and penultimate inequalities follow from~\eqref{eq:coupling_inequality_marginal} and~\eqref{eq:def_Gamma_coupling}, respectively.

We now prove property (3) in the proposition, i.e.~that, under the probability measure $\mathbb{P}_0$ in~\eqref{eq:definition_coupling_nearly_final}, the random vector 
$\tilde{Y}_0$  and the random element $(\hat{Y}_0,\mathcal{W}')$ are independent.
For measurable $A,B,C\subset\BB^d$, by~\eqref{eq:def_Gamma_coupling} we obtain
\begin{align*}
\Gamma(A\times B\times\mathbb{B}^d\times C )
&= 
 \int_{ \mathbb{B}^d\times\BB^d} \delta_{\mathpzc{y}_0}(A) \delta_{\mathpzc{y}'_0}(B)  
\Gamma_{\mathpzc{y}_0, \mathpzc{y}'_0}(\mathbb{B}^d\times C  )\  \nu(\d \mathpzc{y}_0) \mu(\d \mathpzc{y}'_0)\\
&= 
 \int_{ A\times B}  
(\delta_{\mathpzc{y}'_0} P^\infty_\mathpzc{s}) (C  )\  \nu(\d \mathpzc{y}_0) \mu(\d \mathpzc{y}'_0)\\
&= \nu(A)  \int_{ B}  
(\delta_{\mathpzc{y}'_0} P^\infty_\mathpzc{s}) (C  ) \mu(\d \mathpzc{y}'_0)\\
&=\nu(A)\Gamma(\mathbb{B}^d\times B\times\mathbb{B}^d \times C),
\end{align*}
implying
\begin{equation}
\label{eq:temp:indep}
\Gamma(A, \d \mathpzc{y}'_0 ,  \mathbb{B}^d, \d \mathpzc{y}'_{\mathpzc{s}})
= 
\nu(A) \Gamma( \mathbb{B}^d, \d \mathpzc{y}'_0 ,  \mathbb{B}^d, \d \mathpzc{y}'_{\mathpzc{s}}).
\end{equation}

Let now $A,B$ be measurable subsets of $\mathbb{B}^d$ and $D$ a measurable subset of $\Omega'$.  
Then
\begin{align*}
\mathbb{P}_0(A\times B\times\Omega\times D )
&= 
 \int_{ \mathbb{B}^d\times\BB^d\times\BB^d\times\BB^d} \delta_{\mathpzc{y}_0}(A) \delta_{\mathpzc{y}'_0}(B)
\mathbb{P}'_{\mathpzc{y}'_0,\mathpzc{y}'_{\mathpzc{s}} }(D)  \Gamma(\d \mathpzc{y}_0, \d \mathpzc{y}'_0, \d \mathpzc{y}_{\mathpzc{s}}, \d  \mathpzc{y}'_{\mathpzc{s}} )\\
&= 
 \int_{A \times B\times  \mathbb{B}^d\times\BB^d} 
\mathbb{P}'_{\mathpzc{y}'_0,\mathpzc{y}'_{\mathpzc{s}} }(D)  \Gamma(\d \mathpzc{y}_0, \d \mathpzc{y}'_0, \d \mathpzc{y}_{\mathpzc{s}}, \d  \mathpzc{y}'_{\mathpzc{s}} )\\
&= 
 \int_{ B\times  \mathbb{B}^d} 
\mathbb{P}'_{\mathpzc{y}'_0,\mathpzc{y}'_{\mathpzc{s}} }(D)  \Gamma(A, \d \mathpzc{y}'_0, \mathbb{B}^d, \d  \mathpzc{y}'_{\mathpzc{s}} )\\
&= 
\nu(A) \int_{ B\times  \mathbb{B}^d} 
\mathbb{P}'_{\mathpzc{y}'_0,\mathpzc{y}'_{\mathpzc{s}} }(D)  \Gamma(\BB^d, \d \mathpzc{y}'_0, \mathbb{B}^d, \d  \mathpzc{y}'_{\mathpzc{s}} ) \qquad \qquad \text{(by \eqref{eq:temp:indep}).}
\end{align*}
In particular, for   $A=\mathbb{B}^d$, 
we get 
$\mathbb{P}_0(\mathbb{B}^d\times B\times\Omega\times D )= \int_{ B\times  \mathbb{B}^d} 
\mathbb{P}'_{\mathpzc{y}'_0,\mathpzc{y}'_{\mathpzc{s}} }(D)  \Gamma(\BB^d, \d \mathpzc{y}'_0, \mathbb{B}^d, \d  \mathpzc{y}'_{\mathpzc{s}} )$, implying
\begin{equation}
\label{eq:independnece_in_coupling}
\mathbb{P}_0(A\times B\times\Omega\times D )=\nu(A) \mathbb{P}_0(\mathbb{B}^d\times B\times\Omega\times D ).
\end{equation}
Since $A,B\subset\BB^d$ and $D\subset\Omega'$ in~\eqref{eq:independnece_in_coupling} were arbitrary, 
$\tilde{Y}_0$  and  $(\hat{Y}_0,\mathcal{W}')$ are indeed independent, which concludes the proof. 
\end{proof}

\section{Local convergence of the rescaled process}
\label{sec:local}

In this section, we look at the reflected process $Z$ in the time window $[T,T+ C b(X_T)^2 ]$. Theorem~\ref{th:main2} asserts that this process, when recentred and appropriately rescaled, converges in distribution to the limiting process we studied in  Section~\ref{sec:limiting} above. The main goal of this section is to establish this weak convergence. We start by introducing definitions and notation needed to formulate our approach.

\subsection{Definitions and notation}
\label{subsec:def}
For $x_0\in(0,\infty)$, define  the affine function $\mathbf{a}_{x_0}: \R^{1+d} \to \R^{1+d}$, 
\[\mathbf{a}_{x_0}(x,y)\coloneqq \frac{1}{b(x_0)}(x-x_0,y).\]
The image under $\mathbf{a}_{x_0}$ of the domain $\cD$ defined at~\eqref{eq:D-def} is given by
\begin{equation}
\label{eq:D-transform}
\cD_{x_0}\coloneqq \mathbf{a}_{x_0}(\cD) =
\Bigl\{ (\mathpzc{x},\mathpzc{y})\in \R^{1+d}: 
\mathpzc{x} \geq - \frac{x_0}{b(x_0)}\mbox{ and }
|\mathpzc{y}|_d \leq \frac{b(b(x_0)\mathpzc{x}+x_0 )}{b(x_0)} \Bigr\}.\end{equation}
Recalling that $b(x_0) \sim a_\infty x_0^\beta$ as $x_0 \to \infty$, note that $\cD_{x_0}$ is locally cylinder-like in the following sense: as $x_0\to\infty$  we have
$b(b(x_0)x+x_0)/b(x_0)\sim (x_0^{\beta-1}x+1)^\beta$ and hence 
for fixed $B >0$ and large $x_0$ we have
$\{ ( \mathpzc{x},\mathpzc{y}) \in \cD_{x_0} : -B \leq \mathpzc{x} \leq B \} \approx [ -B, B] \times \BB^{d}$.
Define the vector field  $\phi_{x_0}:\partial \cD_{x_0}  \to  \R^{1+d}$ and the matrix-valued function  $\sigma_{x_0}:\cD_{x_0}\to  \cM^+_{1+d}$ 
as follows:
 \begin{equation}
 \label{def:phix}
\phi_{x_0}(\mathpzc{x},\mathpzc{y}) \coloneqq \phi\circ \mathbf{a}^{-1}_{x_0}(\mathpzc{x},\mathpzc{y})\qquad \mbox{and} \qquad \sigma_{x_0} (\mathpzc{x},\mathpzc{y})  \coloneqq  \sigma\circ \mathbf{a}^{-1}_{x_0}(\mathpzc{x},\mathpzc{y}).
\end{equation} 
Set  $\Sigma_{x_0}\coloneqq\sigma_{x_0}^2$.

The process ${\cZ}^T\eqqcolon ({\cX}^T,{\cY}^T)$, defined in~\eqref{eq:Z^T-def} for any $T>0$,  can now be expressed as 
\[ {\cZ}^T_t=
\mathbf{a}_{X_T}(Z_{T+b(X_T)^2t }),\quad t\geq 0,
\] 
where $(Z,L)$ is the strong solution of the initial SDER in~\eqref{eq:SDER}. 
This process, which we should think of as being a scaled and translated (both in time and space) version of $Z=(X,Y)$, starts from $\cZ_0^T=(\cX_0^T,\cY_0^T)=(0,Y_T/b(X_T) )$ and lies inside the domain $\cD_{X_T}$ given by~\eqref{eq:D-transform}. 
Furthermore, an elementary computation shows that
\[ \{ t\in\RP: {\cZ}^T_t\in \partial \cD_{X_T} \}=\{ t\in\RP: {\cZ}_{T+b(X_T)^{2} t}\in \partial \cD \}.\]
Thus, the finite variation process
$\mathcal{L}^T=(\mathcal{L}^T_t)_{t\in\RP}$, given by $\mathcal{L}^T_t\coloneqq L_{T+b(X_T)^{2}t }-L_T$,
increases only when ${\cZ}^T$ belongs to $\partial \cD_{X_T}$, via 
$\mathcal{L}^T_t=\int_0^t \mathbbm{1}_{{\cZ}^T_\mathpzc{s}\in \partial \cD_{X_T} } \d \mathcal{L}^T_\mathpzc{s}$.
The process $\mathcal{W}^T=(\mathcal{W}^T_t)_{t\in\RP}$, 
\begin{equation}
\label{eq:BM_W^T}
    \cW^T_t \coloneqq \mathbf{a}_{X_T}(W_{T+b(X_T)^2 t })- \mathbf{a}_{X_T}(W_T),  
\end{equation}is a standard Brownian motion in $\R^{1+d}$, independent from $Z_T$ ($W$ is the Brownian motion in SDER~\eqref{eq:SDER} satisfied by $(Z,L)$).
Using definitions in~\eqref{def:phix} and changing variables, we obtain
\begin{align*}
 {\cZ}^T_t
&=\Bigl(0,\frac{Y_T}{b(X_T)}\Bigr)+
\int_T^{T+b(X_T)^{2}t } \sigma(Z_s)\frac{ \d W_s }{b(X_T)}
 +\int_T^{T+b(X_T)^{2}t } \phi(Z_s) \d L_s\\
&=  {\cZ}^T_0+\int_0^t \sigma_{X_T}( \cZ^T_\mathpzc{s}  ) \d  \cW^T_\mathpzc{s}+\int_0^t \phi_{X_T}(  \cZ^T_\mathpzc{s}) \d \mathcal{L}^{T}_\mathpzc{s}.
\end{align*}

Given $(x,y)\in \cD\setminus\{0_{\R^{1+d}}\}$ and a standard Brownian motion $\cW$, we  define the process $( \cZ^{(x,y),\cW},\mathcal{L}^{(x,y),\cW})$ to be the unique strong solution
of the SDER in the domain $\cD_x$,
\begin{align}
\label{eq:SDER_Skorkhod_Prob}
\begin{split}
\cZ^{(x,y)}_t&=\Bigl(0,\frac{y}{b(x)}\Bigr)+\int_0^t \sigma_x( \cZ^{(x,y)}_\mathpzc{s}  ) \d  \cW_\mathpzc{s}+\int_0^t \phi_{x}(  \cZ^{(x,y)}_\mathpzc{s}) \d L^{(x,y)}_\mathpzc{s},\\
\mathcal{L}^{(x,y)}_t&=\int_0^t \mathbbm{1}_{\cZ^{(x,y)}_\mathpzc{s}\in \partial \cD_x } \d \mathcal{L}^{(x,y)}_\mathpzc{s},\quad t\in\RP,
\end{split}
\end{align}
driven by  $\cW$,  with boundary reflection field $\phi_{x}$ and diffusion coefficient $\Sigma_x$ (see~\eqref{def:phix} above), and started from $(0,y/b(x))\in \cD_x$. When it is not relevant to specify which Brownian motion $\cW$ is used in~\eqref{eq:SDER_Skorkhod_Prob}, we drop the corresponding superscript in  $(\cZ^{(x,y),\cW},\mathcal{L}^{(x,y),\cW})$.
In particular, the process ${\cZ}^T= ({\cX}^T,{\cY}^T)$ satisfies the almost sure equality
\begin{equation} 
\label{eq:ZT=ZZT}
( {\cZ}^T,\mathcal{L}^T)= ( {\cZ}^{Z_T,\cW^T},\mathcal{L}^{Z_T,\cW^T}).\end{equation}

Recall that the process $\cZ^{\infty,(\mathpzc{x},\mathpzc{y}),\cW}$ is the strong solution to SDER~\eqref{eq:limitEqJoint_Integral} in the infinite cylinder, started from $(0,\mathpzc{y})$ and driven by the Brownian motion $\cW$. For any $T>0$, we now define
\begin{equation}
\label{eq:def_mathcal_Z^T}
\cZ^{\infty,T}\coloneqq\cZ^{\infty,\mathcal{Z}^T_0,\cW^T}, \quad\text{where $\mathcal{Z}^T_0=(0,Y_T/b(X_T))\in\BB^d$.}
\end{equation}
Since the Brownian motion $\cW^T$
is independent of $Z_T=(X_T,Y_T)$, almost surely the  equalities
$\mathbb{P}(\cW^T\in\cdot)=\mathbb{P}(\cW^T\in\cdot|Z_T)=\mathbb{P}(\cW^T\in\cdot|X_T,\cZ_0^T)$ 
and
$$\mathbb{P}(\cW^T\in\cdot|\cZ_0^T)=\mathbb{E}(\mathbb{P}(\cW^T\in\cdot|X_T,\cZ_0^T)|\cZ_0^T)=\mathbb{P}(\cW^T\in\cdot)$$ 
hold, implying that $\cW^T$ is also independent of the starting point $\cZ_0^T$ and thus making $\cZ^{\infty,T}$
in~\eqref{eq:def_mathcal_Z^T} a  well defined solution of SDER~\eqref{eq:limitEqJoint_Integral}.
Note also that $\mathcal{Z}^T$ and $\mathcal{Z}^{\infty,T}$ start from the same point and are driven by the same Brownian motion, but (by~\eqref{eq:def_mathcal_Z^T}) $\mathcal{Z}^{\infty,T}$ satisfies SDER~\eqref{eq:limitEqJoint_Integral} in the infinite cylinder while (by~\eqref{eq:ZT=ZZT})
$\mathcal{Z}^T$ satisfies SDER~\eqref{eq:SDER_Skorkhod_Prob} in the rescaled domain $\mathcal{D}_{X_T}$ (which approximates an infinite cylinder on an appropriate scale when $T$ is large). 

Finally, for any function  $f:\RP\to\mathbb{R}^k$ and $s>0$, we denote   the uniform norm $$\| f \|_{[0,s]} \coloneqq \sup_{u \in [0,s]} | f(u) |_{k}$$ (if $s=\infty$, the interval $[0,s]$ is taken to equal $[0,\infty)$). 

\subsection{Convergence of the scaled process to the limiting process: the case of deterministic starting points}
\label{sec:cv}

The core of the proof of  Theorem~\ref{th:main2}  essentially amounts to showing that $\cZ^{(x,b(x) y),\cW}$ 
(which lives in $\cD_x$ and starts from $(0,y)$ with $y\in\BB^d$)
converges as $x \to \infty$ to $\cZ^{\infty,(0,y), \cW}$ (which satisfies SDER~\eqref{eq:limitEqJoint_Integral} on the cylinder $\R \times \BB^d$, is driven by the same Brownian motion $\cW$ as $\cZ^{(x,b(x) y),\cW}$, and starts from $(0,y)$). This looks quite challenging, since we would need some kind of continuity property of the SDER with respect to perturbation of the diffusion matrix, the reflection vector field, and  the domain itself. Fortunately, it is only a one-parameter family of data, indexed by $x$, that we need to consider for such a continuity property.  
To proceed, we embed
$\cZ^{(x,b(x) y),\cW}$ and $\cZ^{\infty,(0,y), \cW}$ into a richer space, based on a domain $\cDtot \subset \R^{1+1+d}$ in which the first coordinate 
$h\in(0,\infty)$ corresponds to $x^{-\frac{1}{\gamma}}$. 
Throughout $\gamma$ is fixed and should be thought of as being very large: namely, we assume $\gamma\geq \frac{2}{1-\beta}$ and $\gamma>\frac{1}{\epsilon}$ where $\epsilon$ is the one in Assumption~\eqref{hyp:S}, say for example 
$ \gamma\coloneqq \max(2/(1-\beta),2/\epsilon)$.
We then compactify smoothly near $h=0$ by defining
\begin{equation}
    \label{eq:Dtot-def}
\cDtot\coloneqq \Big( (-1,0]  \times \R\times \BB^d\Big) \cup \Big( \bigcup_{h\in(0,\infty)} \big( \{h\} \times \cD_{h^{-\gamma} } \big)\Big)\subset \mathbb{R}^{1+1+d}.
\end{equation}
 See Figure~\ref{fig:Dtot} below for a visualization of the boundary of $\cDtot\subset\R^3$, i.e.~the case $d=1$.
\begin{figure}[ht]
\includegraphics[width=0.6\textwidth]{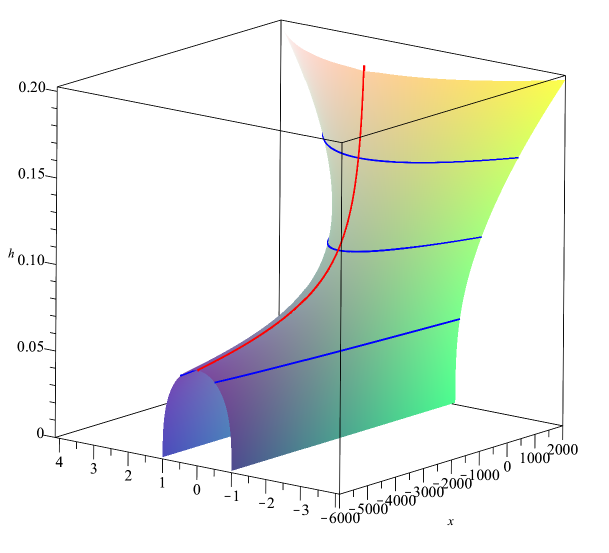}
\caption{
\label{fig:Dtot}
A visualization of the qualitative features of the ``boot'' domain~$\cDtot \subseteq \R^{1+1+1}$ as defined at~\eqref{eq:Dtot-def}; the coloured surface is the boundary of $\cDtot$.
The blue contours illustrate (the boundaries of)  
sections in the $(x,y)$ plane for fixed $h >0$, which are the domains $\cD_{h^{-\gamma}}$ defined at~\eqref{eq:D-transform}, i.e., transformed versions of $\cD$ from~\eqref{eq:D-def} that resemble cylinders in a neighbourhood of $x \approx 0$.
For $h \leq 0$ the corresponding sections are all the cylinder $\R \times \BB^d$.
The red curve is the ridge in the $y=0$ plane that is the locus of maximal~$h$. 
 The exponent $\gamma$, by being large enough, ensures a smooth transition at $h=0$ on the $(h,y)$ plane. Here $\beta>0$. When $\lim_{x\to\infty} b(x) =0$, the $h>0$ slices contract as $x \to \infty$, rather than expand.}
\end{figure}

In the richer  space  $\cDtot$, 
we consider for $x>0$ the augmented process $( x^{-\frac{1}{\gamma}},  \cZ^{(x,b(x)y),\cW}   )$, 
the first coordinate of which remains constant at $h = x^{-\frac{1}{\gamma}}$ at all times, indicating that
$\cZ^{(x,b(x)y),\cW}$ remains in 
$\cD_{h^{-\gamma}} = \cD_x$. 
The problem now reduces to showing  that the pair $( x^{-\frac{1}{\gamma}},  \cZ^{(x,b(x)y),\cW})$ converges as $x \to \infty$ to $(0,\cZ^{\infty,(0,y), \cW} )$.
This is achieved by establishing
continuity with respect to the initial condition of the solutions of  an SDER on $\cDtot$,
for which we  
 rely on well-known results in~\cite{Ishii}. Applying~\cite{Ishii} requires establishing first certain regularity properties of the space $\cDtot$ and the coefficients of the SDER
 satisfied by the process $( x^{-\frac{1}{\gamma}},  \cZ^{(x,b(x)y),\cW})$.  This is the purpose of Lemma~\ref{le:smooth} below.

We emphasize that the coefficients of the SDERs on $\cDtot$ are somewhat degenerate,
since the ``artificial'' first coordinate of the process remains constant. However this is not an obstacle for~\cite{Ishii}, and the main technical challenge in Lemma~\ref{le:smooth}  is in terms of the smoothness of the coefficients near $h =0$.

\begin{lemma}
\label{le:smooth}
The following statements hold.
\begin{enumerate}[label=(\roman*)]
\item
\label{le:smooth-i}
Under Assumption \eqref{hyp:D3}, the intersection of the boundary $\partial \cDtot$ with the open half-space $(-1,\infty)\times \R^{1+d}$ is $\cC^1$.
\item
\label{le:smooth-ii}
Let $\sigmatot:\cDtot \to \cM_{2+d}$ be the diffusion matrix given by
\[
    \sigmatot (h,x,y)\coloneqq
\begin{cases}    
      \begin{pmatrix}
        0 & 0\\
        0 & \sigma^\infty
      \end{pmatrix}
      & \text{if } h\leq 0, \\[1.4em]
      \begin{pmatrix}
        0 & 0\\
        0 & \sigma(   \mathbf{a}_{h^{-\gamma}}^{-1} (x,y) )
      \end{pmatrix}
& \text{if } h > 0.
    \end{cases}
\]
Under Assumptions \eqref{hyp:D3} and \eqref{hyp:S}, $\sigmatot$
is $\cC^1$.
\item
\label{le:smooth-iii}
Let $\phitot:\cDtot\to\R^{1+1+d}$ be the vector field defined on $\partial\cDtot $ by $\phitot(h,x, y)=(0,s_0, \phi^{(d)}_\infty(y) )$ for $h\leq 0$ and
$\phitot(h,x, y)= (0,\phi(   \mathbf{a}_{h^{-\gamma}}^{-1} (x,y)    ))$ for $h>0$.
Under Assumptions~\eqref{hyp:D3} and \eqref{hyp:S}, this vector field is~$\cC^1$.
\end{enumerate}
\end{lemma}

The assumptions in Section~\ref{sec:assumptions} above, specifically~\eqref{hyp:S}, have been tuned  for the
continuity properties in Lemma~\ref{le:smooth}  to hold. The proof of Lemma~\ref{le:smooth}, based on elementary calculus, is contained in Appendix~\ref{Appendix:deterministic}.

In the remainder of the paper,  Assumptions \eqref{hyp:D1}, \eqref{hyp:D3}, \eqref{hyp:C}, \eqref{hyp:V+}, and~\eqref{hyp:S} are assumed to hold without mentioning them explicitly. We also assume $\beta\in(-1,1)$ and  make it explicit when it is further assumed that $\beta>-1/3$.

We can now state the approximation result for deterministic starting points.

\begin{lemma}
  \label{le:cv1new}
  For $w\in \cDtot$, let $\cZtot^{w,\cW}$ be the solution to the SDER on the domain $\cDtot$
  (defined in~\eqref{eq:Dtot-def} above)
  with reflection vector field $\phitot$ and diffusion coefficient $\sigmatot$, driven by $(0,\cW)$ and started from $w$. This solution is unique in the strong sense. Furthermore, as $w\to w_0$ sequentially in $\cDtot$,
   $\cZtot^{w,\cW}$ converges locally uniformly in probability to $\cZtot^{w_0,\cW}$. That is, for all $s\in(0,\infty)$, $w_0\in \cDtot$ and $\epsilon>0$, there exists $\delta = \delta (s,\epsilon,w_0)>0$ such that for all $w\in \cDtot$ satisfying $|w-w_0|_{2+d}\leq \delta$, we have
   \begin{equation}
   \label{eq:cont}
    \mathbb{P} \Big(   \bigl\| \cZtot^{w,\cW}-\cZtot^{w_0,\cW} \bigr\|_{[0,s]} \geq \epsilon \Big)\leq \epsilon.
   \end{equation}
\end{lemma}
\begin{proof}
    The result essentially follows from Theorem~5.1 and Corollary~5.2 in \cite{Ishii}, together with a localization procedure, as the results in~\cite{Ishii} apply for SDERs on compact domains only. We emphasize again that the degenerate nature of neither $\sigmatot$ nor the driver $(0,\cW)$ are obstacles for the machinery in~\cite{Ishii}.
    
    We will take a smooth truncation of $\cDtot$ chosen large enough so that, with high probability, our process on $\cDtot$ remains in the truncated domain.     For $C\geq 1$, 
    there exists $A := A(C,\gamma) \in(0, \infty)$ such that $\cDtot \cap \bigl( \R \times [-C,C] \times \R^d \bigr) \subseteq \R \times [-C,C] \times A \BB^d$.
    Let $\cDtot^C$ be a compact set with $\cC^1$-smooth boundary which agrees with $\cDtot$ on
    the extended cylinder $\cC_C := (-\frac{1}{2}, C]\times [-C,C]\times A \BB^d$. Let also $\phitot^C$ be a $\cC^1$ vector field defined on $\partial \cDtot^C$ and  which agrees with $\phitot$ on $\partial \cDtot \cap \cC_C$. 
    
    With this regularity, and because $\cDtot^C$ is now compact, we are in \emph{Case 2} from \cite{Ishii}, with the condition (5.2) and the Lipschitz condition on the diffusion coefficient (both also from \cite{Ishii}) are met, according to Lemma~\ref{le:smooth}. We can thus apply both \cite[Thm 5.1]{Ishii} and \cite[Cor.~5.2]{Ishii}. Corollary~5.2 ensures the existence and uniqueness of a strong solution.
    Denote by 
    $\cZtot^{w,\cW,C}_{t}$ and $\cZtot^{w_0,\cW,C}_{t}$ 
    \emph{these strong solutions}, with initial conditions $w$ and $w_0$ respectively. These new processes are defined exactly as  $\cZtot^{w,\cW}_{t}$ and $\cZtot^{w_0,\cW}_{t}$, except that $\cDtot$ and $\phitot$ are replaced with  $\cDtot^C$ and $\phitot^C$. Then applying Theorem~5.1 of~\cite{Ishii} (together with Gr\"onwall's inequality) shows that the compactified version of~\eqref{eq:cont} holds, i.e., with $\cZtot^{w,\cW}_{t}$ and $\cZtot^{w_0,\cW}_{t}$  replaced, respectively, with $\cZtot^{w,\cW,C}_{t}$ and $\cZtot^{w_0,\cW,C}_{t}$. 
    (In fact,  \cite[Thm 5.1]{Ishii}  provides $L^2$ convergence, rather than just convergence in probability, but we retain only the latter  when taking the limit $C\to \infty$.)
    
    We denote by $\pi_1 : \R^{1+1+d} \to \R$ the projection onto the horizontal spatial coordinate, i.e.,
    $\pi_1 (h,x,y ) := x$. For $w \in \cDtot$ and $C, r \in \RP$, define the stopping times 
    \[ \tau_r^w := \inf \bigl\{ t \in \RP : | \pi_1 \cZtot^{w,\cW}_t | > r \bigr\}, \text{ and }
    \tau_r^{w,C} := \inf \bigl\{ t \in \RP : |\pi_1 \cZtot^{w,\cW,C}_t | > r \bigr\}.
    \]
    For every $r < C$, $\cZtot^{w,\cW}_{t \wedge \tau^w_r} \in \cC_C$ for all $t \in \RP$, and hence $(\cZtot^{w,\cW}_{t \wedge \tau^w_r}, \tau^w_r)$ solves the same SDER as does $(\cZtot^{w,\cW}_{t \wedge \tau^{w,C}_r}, \tau^{w,C}_r)$. Thus, by strong uniqueness,  for every $r <C$, a.s., 
    $(\cZtot^{w,\cW}_{t \wedge \tau^w_r}, \tau^w_r)$ coincides with
    $(\cZtot^{w,\cW}_{t \wedge \tau^{w,C}_r}, \tau^{w,C}_r)$. 
    Since $\cZtot^{w,\cW}$ is non-explosive by~\cite[Thm~2.2(ii)]{MMW},  we conclude that $\sup_{t \in [0,s]} | \pi_1 \cZtot^{w,\cW}_t |$ is a.s.~finite for every $s$.
    
    Fix 
    $\epsilon\in(0,1)$, $s \in \RP$, and $w_0 \in \cDtot$.
    Choose $C = C(s,\epsilon,w_0)$ such that
    $\mathbb{P} ( \tau^{w_0}_C \leq s ) \leq \epsilon/2$. 
    Then for this $C$ there exists $\delta = \delta (s,\epsilon,w_0)$
    such that, for all $w\in \cDtot$ with $|w-w_0|_{2+d}\leq \delta$, 
    \[    \mathbb{P} \Bigl(   \bigl\| \cZtot^{w,\cW,C+1}-\cZtot^{w_0,\cW,C+1} \bigr\|_{[0,s]} < \epsilon \Bigr)\leq \epsilon/2.    \]
    Hence, by choice of~$C$, it holds that, for all $w\in \cDtot$ with $|w-w_0|_{2+d}\leq \delta$, 
    \begin{align*}
    1- \epsilon & 
    \leq 
    \mathbb{P} 
    \Bigl(  \bigl\| \cZtot^{w,\cW,C+1}-\cZtot^{w_0,\cW,C+1} \bigr\|_{[0,s]} \geq \epsilon, \ \tau^{w_0}_C > s \Bigr) \\
    & 
    \leq 
    \mathbb{P} 
    \Bigl(  \sup_{t\in [0,s]} \bigl| \cZtot_{t}^{w,\cW,C+1}-\cZtot_{t\wedge \tau^{w_0,C+1}_C}^{w_0,\cW,C+1} \bigr|_{2+d} \geq \epsilon, \, \tau^{w_0,C+1}_C = \tau^{w_0}_C > s \Bigr) \\
    & 
    \leq 
    \mathbb{P} 
    \Bigl(  \sup_{t\in [0,s]} \bigl| \cZtot_{t \wedge \tau^{w,\cW,C+1}_{C+\epsilon}}^{w,\cW,C+1}-\cZtot_{t\wedge \tau^{w_0,C+1}_C}^{w_0,\cW,C+1} \bigr|_{2+d} \geq \epsilon, \, \tau^{w_0,C+1}_C  > s , \, \tau^{w,C+1}_{C+\epsilon} > s\Bigr), \end{align*}
    since if the process started from $w_0$ has not reached $C$, then the process started from $w$ has not reached $C+\epsilon < C+1$. Using strong uniqueness for the stopped processes, as described above, it follows that
    \begin{align*}
    1- \epsilon 
    & \leq 
    \mathbb{P} 
    \Bigl(  \sup_{t\in [0,s]} \bigl| \cZtot_{t \wedge \tau^{w,\cW}_{C+\epsilon}}^{w,\cW}-\cZtot_{t\wedge \tau^{w_0}_C}^{w_0,\cW} \bigr|_{2+d} \geq \epsilon, \, \tau^{w_0,C+1}_C > s, \, \tau^{w,C+1}_{C+\epsilon} > s \Bigr) \\
    & = 
    \mathbb{P} 
    \Bigl(  \bigl\| \cZtot^{w,\cW}-\cZtot^{w_0,\cW} \bigr\|_{[0,s]} < \epsilon, \, \tau^{w_0,C+1}_C > s, \, \tau^{w,C+1}_{C+\epsilon} > s \Bigr) \\
    & \leq
    \mathbb{P} 
    \Bigl(   \bigl\| \cZtot^{w,\cW}-\cZtot^{w_0,\cW} \bigr\|_{[0,s]} < \epsilon \Bigr) ,
    \end{align*}
    for all $w\in \cDtot$ with $|w-w_0|_{2+d}\leq \delta$, and this completes the proof of~\eqref{eq:cont}.
\end{proof}

\subsection{Convergence of the scaled process to the limiting process in stationarity}
\label{subsec:techPart}
The main result of this subsection,  Proposition~\ref{prop:seq} below, will readily imply Theorem~\ref{th:main2}, and plays a crucial role in
the proof of Theorem~\ref{th:main1} above, the main result of the paper. The proof of Proposition~\ref{prop:seq} relies on an approximation result in Proposition~\ref{prop:long} below, which builds on Lemma~\ref{le:cv1new} from the previous subsection, and which we establish first.

For an initial condition $z(T)=(x(T), y(T))\in\cD$ with asymptotic behaviour specified by $\lim_{T \to \infty} x(T) = \infty$ and $\lim_{T \to \infty} \frac{y(T)}{b(x(T))} =  \mathpzc{y}\in \BB^d$, choosing 
$w_0=(0,0,\mathpzc{y})$ in 
Lemma~\ref{le:cv1new} implies that
$\cZ^{z(T),\cW}$ (defined in~\eqref{eq:SDER_Skorkhod_Prob} above) converges in probability to $\cZ^{\infty, (0,\mathpzc{y}), \cW}$ (see~\eqref{eq:limitEqJoint_Integral} and Remark~\ref{rem:Notation} above for definition) on any compact time interval. Proposition~\ref{prop:long} below extends this in two ways.
First, via a compactness argument exploiting the local uniformity in~\eqref{eq:cont}, 
we extend the result to the case of a \emph{random} starting point in the $\mathpzc{y}$ coordinate. However, this is still some way short of a proof of  Proposition~\ref{prop:seq} (and Theorem~\ref{th:main2}) because it does not establish convergence of the random initial starting point $\cY^T_0=Y_T/b(X_T)$ in $\BB^d$ (recall the definition of $\cZ^T=(\cX^T,\cY^T)$ in~\eqref{eq:Z^T-def}) to the stationary measure $\mu$ of the process $\cZ^{\infty, (0,\mathpzc{y}), \cW}$.
The additional ingredient required for this convergence is the fast mixing of the process $\cY^T$. 

Exploiting the fast mixing property in turn requires the  strengthening of the convergence in Lemma~\ref{le:cv1new} to \emph{a growing time window} (as opposed to one of constant size). Specifically, Proposition~\ref{prop:long} shows that instead of considering the supremum over $[0,s]$ for a fixed $s$ as in Lemma~\ref{le:cv1new}, we permit $s = \mathpzc{s}$ to be a function of $T$ that slowly increases to~$\infty$. The growing-time-window property will be deduced from  a general principle  given in the following lemma.

\begin{lemma}
    \label{le:extend}
    Let $A^T:\mathbb{R}_+\times \Omega\to \mathbb{R}^k$ be a family of continuous stochastic processes indexed by $T\in (0,\infty)$. Assume that for any constant $s>0$, as $T\to\infty$, we have
    \[
    \|A^T \|_{[0,s]}\overset{(proba)}\longrightarrow 0.
    \]
    Then there exists a function $\mathpzc{s}_0:\mathbb{R}_+\to\mathbb{R}_+ $ such that, as $T\to\infty$,  we have $\mathpzc{s}_0(T)\to\infty$ and 
    \[
    \|A^T \|_{[0,\mathpzc{s}_0(T)]}\overset{(proba)}\longrightarrow 0.
    \]
\end{lemma}
\begin{proof}
    Define the following non-decreasing function (with convention $\inf\emptyset\coloneqq\infty$),
    \[
    \mathcal{T}:s\mapsto  \inf \Bigl\{ T_0: \forall T\geq T_0, \, \mathbb{P} \Bigl(  \| A^T\|_{[0,s]}\geq \frac{1}{s}   \Bigr)\leq \frac{1}{s} \Bigr\},
    \]
    and its (modified) inverse
    \[
    \mathpzc{s}:T\mapsto \sup \{ s: \mathcal{T}(s)\leq T-1\} \in \RP\cup\{\infty\}.
    \]
    
    In the case there exists $T_0<\infty$ such that $\mathpzc{s}(T_0)=\infty$, it is easily seen that  $A^T=0$ almost surely for all $T>T_0+1$, and then the conclusion of the lemma follows directly. Thus, we assume that $\mathpzc{s}(T)$ is finite for all $T$. 
    
    Let us assume that the function $\mathpzc{s}$ is bounded, and let $\mathpzc{s}_\infty=\|\mathpzc{s}\|_{[0,\infty)}<\infty$. Then it follows from the definition of $\mathpzc{s}(T)$ that $\mathcal{T}( \mathpzc{s}_\infty+1)>T-1$ for all $T$, implying 
    $\mathcal{T}( \mathpzc{s}_\infty+1)=\infty$. This implies that for $\mathpzc{s}=\mathpzc{s}_\infty+1$, the assumed convergence  
    $\|A^T \|_{[0,s]}\to 0$ in probability cannot hold. This contradicts 
    the assumption of the lemma, implying  the function $\mathpzc{s}$ must be unbounded.
    
    By the definition of $\mathpzc{s}$, for all $T$ such that $\mathpzc{s}(T)\geq 1$, we have $\mathcal{T}(\mathpzc{s}(T)-1)\leq T-1<T $. Set $\mathpzc{s}_0(T)\coloneqq \max(0,\mathpzc{s}(T)-1)$, and $T_0$ be the first time when $\mathpzc{s}_0(T)= \mathpzc{s}(T)-1$. Then, for all $T\geq T_0$, 
    \[
    \mathcal{T}(\mathpzc{s}_0(T))<T. 
    \]    
    Note that the $-1$ in the definition of $\mathpzc{s}$  ensures this inequality is strict, which is useful because the infimum defining $\mathcal{T}$ might not be attained.
    The $-1$ in the definition of the function $\mathpzc{s}_0$ above is there because the supremum defining $\mathpzc{s}$ might not be attained. 
    
    Let $\epsilon>0$. For $T_1>T_0$ sufficiently large, 
    $\frac{1}{\mathpzc{s}_0(T_1)}<\epsilon$. By definition of $\mathcal{T}$, for all $T> \mathcal{T}(s)$, 
    \[ 
    \mathbb{P} \Bigl(  \| A^T\|_{[0,s]}\geq \frac{1}{s}   \Bigr)\leq \frac{1}{s}.
    \] 
    In particular, for all $T>T_1$, we have 
    \[ 
    \mathbb{P} \Bigl(  \| A^T\|_{[0,\mathpzc{s}_0(T)]}\geq \epsilon   \Bigr)
    \leq 
    \mathbb{P} \Bigl(  \| A^T\|_{[0,\mathpzc{s}_0(T)]}\geq \frac{1}{\mathpzc{s}_0(T)}   \Bigr)
    \leq  \frac{1}{\mathpzc{s}_0(T)}\longrightarrow0\quad\text{as $T\to\infty$.}\qedhere 
    \]
\end{proof}

\begin{proposition}
    \label{prop:long}
    For the families of processes $\mathcal{Z}^T$ and $\mathcal{Z}^{\infty,T}$, defined respectively in~\eqref{eq:Z^T-def} and~\eqref{eq:def_mathcal_Z^T} for all  $T>0$,
    there exists a function $\mathpzc{s}_1:\RP\to \RP$ such that the following limits hold: $\lim_{T\to\infty}\mathpzc{s}_1(T)=\infty$
    and
    \begin{equation} 
    \label{eq:long}
    \| \cZ^{T}-\cZ^{\infty,T}\|_{ [0,\mathpzc{s}_1(T)] } \overset{(proba)}{\underset{T\to \infty}{\longrightarrow}} 0.
    \end{equation}
\end{proposition}
Recall from~\eqref{eq:ZT=ZZT} and~\eqref{eq:def_mathcal_Z^T} above that, for any given $T$, the processes $\mathcal{Z}^T$ and $\mathcal{Z}^{\infty,T}$ are driven by the same Brownian motion $\cW^T$ and are also starting from the same random point $\cZ^{T}_0=(0,\cY^T_0)=(0,Y_T/b(X_T)) $ in $\{0\}\times \mathbb{B}^d$, which is independent of $\cW^T$ (but $\mathcal{Z}^T$ and $\mathcal{Z}^{\infty,T}$ do not satisfy the same SDER nor do they have the same state space).  

\begin{proof}[Proof of Proposition~\ref{prop:long}]
By Lemma~\ref{le:extend} (with $A^T={\cZ}^{T}-\cZ^{\infty,T}$), it suffices to show that 
    \begin{equation}
    \label{eq:conv_ptob_need_to_prove}
    \text{for all $s\in(0,\infty)$ we have}\quad
         \| {\cZ}^{T}-\cZ^{\infty,T}\|_{[0,s]} \overset{(proba)}{\underset{T\to \infty}{\longrightarrow}} 0.
    \end{equation}

    Fix arbitrary $s>0$ and $\epsilon>0$. Pick arbitrary $\mathpzc{y}_0\in \BB^d$ and Brownian motion $\cW$. By Lemma~\ref{le:cv1new},  there exists $\delta( s,\epsilon, \mathpzc{y}_0 )>0$ such that for $w_0\coloneqq(0,0,\mathpzc{y}_0)$ and all
    $w =(h,0,\mathpzc{y})\in \widehat{\cD}$ (with $h>0$, $\mathpzc{y}\in \BB^d$ and the domain $\widehat{\cD}$ given in~\eqref{eq:Dtot-def}) satisfying $|w-w_0|_{2+d}\leq \delta(s,\epsilon, \mathpzc{y}_0)$, with probability at least $1-\epsilon$ we have
    \begin{align*} 
    \| {\cZ}^{(h^{-\gamma},b(h^{-\gamma})\mathpzc{y}),\cW}-\cZ^{\infty,(0,\mathpzc{y}_0),\cW}\|_{[0,s]}
    =
    \| 
    \cZtot^{w,\cW} -\cZtot^{w_0,\cW} \|_{[0,s]} \leq \epsilon.
    \end{align*}
   
    We fix such a positive function $(s,\epsilon,\mathpzc{y}_0)\mapsto \delta(s,\epsilon,\mathpzc{y}_0)$ with the property in the previous paragraph.
    We now show that there exists $\delta_*(s,\epsilon)>0$, satisfying $\delta_*(s,\epsilon)\leq\delta(s,\epsilon,\mathpzc{y}_0)$ for all  $\mathpzc{y}_0\in\BB^d$. Indeed, by the compactness of $\BB^d$, 
    there exists a finite set $(\mathpzc{y}_i)_{i\in I}\in (\BB^d)^I$ such that the balls of radius  $\delta(s,\epsilon/2, \mathpzc{y}_i )/2$ and centred at $\mathpzc{y}_i$ cover $\BB^d$. Define $\delta_*(s,\epsilon )\coloneqq \min_{i\in I} \delta(s,\epsilon/2, \mathpzc{y}_i )/2$. Then, 
    for every $w_0=(0,0,\mathpzc{y}_0)$  there is some $i\in I$ such that $|\mathpzc{y}_0-\mathpzc{y}_i|_d   \leq \delta_*(s,\epsilon ) $; choose one such $i$ for each $w_0$. 
    Thus, with probability at least $1-\epsilon/2$, 
    \begin{equation}
    \label{eq:prob1}
    \| \cZ^{\infty, (0,\mathpzc{y}_0),\cW}-\cZ^{\infty,(0,\mathpzc{y}_i),\cW}\|_{[0,s]}
    \leq \epsilon/2.
    \end{equation}
    Furthermore, 
    for all $w=(h,0,\mathpzc{y})$ satisfying $|w-w_0|_{2+d} \leq \delta_*(s,\epsilon )$,
    \[ | w- (0,0,\mathpzc{y}_i) |_{2+d} \leq |w-w_0|_{2+d} + |\mathpzc{y}_0-\mathpzc{y}_i|_d  \leq 
    \delta(s,\epsilon/2, \mathpzc{y}_i ).
    \]
    Thus, with probability at least $1-\epsilon/2$, 
    \begin{equation}
    \label{eq:prob2}
    \| \cZ^{(h^{-\gamma},b(h^{-\gamma})\mathpzc{y}),\cW}-\cZ^{\infty,(0,\mathpzc{y}_i),\cW}\|_{[0,s]}
    \leq \epsilon/2.
    \end{equation}
    Combining~\eqref{eq:prob1} and~\eqref{eq:prob2} with the triangle inequality implies
    that  for $0<h<\delta_*(s,\epsilon)$ and all $\mathpzc{y}_0\in\BB^d$, we have
    \begin{equation}
    \label{eq:uniform_prob_bound}
    \mathbb{P} \bigl( \| \cZ^{(h^{-\gamma},b(h^{-\gamma})\mathpzc{y}_0),\cW}-\cZ^{\infty,(0,\mathpzc{y}_0),\cW}\|_{[0,s]}
    \leq \epsilon \bigr) \geq 1 -\epsilon.
    \end{equation}
    
    By fixing random $h=X_T^{-1/\gamma}$ and  $\mathpzc{y}_0=\cY^T_0$ (recall by~\eqref{eq:Z^T-def} that $\cY^T_0=Y_T/b(X_T)$), setting $\cW$ equal to $\cW^T$ and applying the equality 
    in~\eqref{eq:ZT=ZZT}, we 
    get $\cZ^{T}=\cZ^{(X_T,Y_T),\cW_T}$ and 
    $\cZ^{\infty, T}=\cZ^{\infty,(0,\cY^T_0),\cW_T}$.
    By~\eqref{eq:uniform_prob_bound},
    $\cY^T_0$-almost surely it holds
    \begin{align*}
  &  \mathbb{P}( \| \cZ^{T}-\cZ^{\infty, T}\|_{[0,\mathpzc{s}]}\geq \epsilon \mid \cY^T_0)\\
&\leq        \mathbb{P}(  X_T^{-1/\gamma}\geq \delta_*(s,\epsilon) \mid \cY^T_0 )
+\mathbb{P}(  \| \cZ^{T}-\cZ^{\infty, T}\|_{[0,\mathpzc{s}]}\geq \epsilon\text{ and } X_T^{-1/\gamma}\leq \delta_*(s,\epsilon) \mid \cY^T_0)\\
& \leq 
       \mathbb{P}(  X_T\leq \delta_*(s,\epsilon)^{-\gamma} \mid \cY^T_0 )+\sup_{\mathpzc{y}_0\in \mathbb{B}^d, h \leq \delta_*(s,\epsilon) } \mathbb{P}( 
       \| \cZ^{(h^{-\gamma},b(h^{-\gamma})\mathpzc{y}_0),\cW}-\cZ^{\infty,(0,\mathpzc{y}_0),\cW}\|_{[0,s]}    \geq \epsilon)
       \\
&\leq \mathbb{P}(  X_T\leq \delta_*(s,\epsilon)^{-\gamma} \mid \cY^T_0 )+ \epsilon.
    \end{align*}
By taking expectations, for $T$ sufficiently large, we get
\[\mathbb{P}( \| \cZ^{T}-\cZ^{\infty, T}\|_{[0,\mathpzc{s}]} \geq \epsilon)
   \leq 
    \mathbb{P}(  X_T\leq \delta_*(s,\epsilon)^{-\gamma} )+\epsilon
    \leq 2\epsilon.\]
     Since $\epsilon>0$ was arbitrary, the  convergence in probability in~\eqref{eq:conv_ptob_need_to_prove} follows.
\end{proof}

Although Lemma~\ref{le:sequenceRand} below is in appearance probabilistic, it is in fact deterministic: it follows  directly by applying Lemma~\ref{le:sequenceDet} to the random function $\theta:T\mapsto 4 b(X_T)^2$. Almost surely, the function $\theta$  satisfies the assumption of Lemma~\ref{le:sequenceDet}, namely that $\theta(T)\sim C T^{\frac{2\beta}{1+\beta}}$ as $T\to\infty$ for some constant $C>0$, by the strong law in~\eqref{def:c1} and the assumption that $b(x)\sim a_\infty x^\beta$. As Lemma~\ref{le:sequenceDet} is deterministic, it is stated and proved in Appendix~\ref{sec:s-s-s} below. 

\begin{lemma}
    \label{le:sequenceRand}
    Assume that $\mathpzc{s}_1:(0,\infty) \to (0,\infty)$ is such that $\lim_{T \to \infty} \mathpzc{s}_1(T)=\infty$. Then, there exist deterministic functions $\mathpzc{s}_2:(0,\infty) \to (0,\infty)$ and $S:(0,\infty)\to (0,\infty)$, such that 
    $4 \mathpzc{s}_2\leq \mathpzc{s}_1$, 
    $|S(T)-T|=O( T^\frac{2\beta}{1+\beta}\log T ) $ and $\mathpzc{s}_2(T)\to \infty$ as $T \to \infty$, and  such that almost surely, there exists a finite (random) time $T_0$ such that for all $T>T_0$,
    \begin{equation}
    \label{eq:inclu}
    [T, T+ b(X_T)^2 {\mathpzc{s}_2}(T)]
    \subset [S(T)+2 b(X_{S(T)})^2 \mathpzc{s}_2(S(T)) , S(T)+ 4 b(X_{S(T)})^2 \mathpzc{s}_2(S(T))].
    \end{equation}
\end{lemma}
The interval inclusion~\eqref{eq:inclu} is illustrated in Figure~\ref{fig:sequence}. 

\begin{proposition}
    \label{prop:seq}
   There exist a function ${\mathpzc{s}}_3:(0,\infty)\to (0,\infty)$,  with $\lim_{T\to \infty} {\mathpzc{s}}_3(T)=\infty$, and a family of probability measures $(\mathbb{P}^T)_{T>0}$, each extending the probability space on which $(Z,L)$ and $W$ satisfy SDER~\eqref{eq:SDER}, such that the following holds: 
    under $\mathbb{P}^T$, there exists a random vector $\tilde{Y}^T$ in $\mathbb{B}^d$ with law $\mu$ and 
    a Brownian motion $\tilde \cW^T$, independent of $\tilde{Y}^T$, such that  the strong solution $\cZ^{\infty,(0,\tilde{Y}^T),\tilde{\mathcal{W}}^T}$ of SDER~\eqref{eq:limitEqJoint_Integral}, driven by $\tilde \cW^T$, satisfies 
    \[ \|\cZ^{T} -  \cZ^{\infty,(0,\tilde{Y}^T),\tilde{\mathcal{W}}^T}  \|_{[0,\mathpzc{s}_3(T)]} \overset{(proba)}{\underset{T\to \infty}\longrightarrow} 0.\]
\end{proposition}
To be specific, what is meant here by the convergence in probability is that for all $\epsilon>0$, for all $T$ large enough, 
\[ 
\mathbb{P}^T( \|\cZ^{T} -  \cZ^{\infty,(0,\tilde{Y}^T),\tilde{\mathcal{W}}^T}  \|_{[0,\mathpzc{s}_3(T)]} \geq \epsilon  )\leq \epsilon.
\]
\begin{remark}
\label{rem:proof_of_main2}
For every $T>0$, under $\mathbb{P}^T$, the process 
    $\cZ^{\infty,(0,\tilde{Y}^T),\tilde{\mathcal{W}}^T}$ (see Remark~\ref{rem:Notation} above for the definition of this process) is a copy of the stationary process
    $\cZ^{\infty,\mu}$.
    Since $\lim_{T\to \infty} {\mathpzc{s}}_3(T)=\infty$,
Proposition~\ref{prop:seq} readily implies Theorem~\ref{th:main2}.
\end{remark}

Since by definition~\eqref{eq:Z^T-def} we have 
$\cZ_0^T=(0,Y_T/b(X_T) )$,
taking $t=0$ in Proposition~\ref{prop:seq} yields 
weak convergence of $Y_T/b(X_T)$ to the distribution $\mu$.
Note that, by the strong law in~\eqref{def:c1} and  Assumption~\eqref{hyp:D3} on the boundary function $b$, we have $a_\infty c_1^\beta T^{\beta/(1+\beta)}/b(X_T)\to1$ almost surely as $T\to\infty$.  Thus the the following corollary holds.

\begin{corollary}
    \label{coro:Ycv} For $\beta\in(-1,1)$, as $T\to \infty$,
        the random vectors  $Y_T/(a_\infty c_1^\beta T^{\beta/(1+\beta)})$ and $Y_T/b(X_T)$ converge in distribution to the law
    $\mu$ supported on the ball $\mathbb{B}^d$
    (see Corollary~\ref{cor:mu-existence-convergence} above for the characterisation of $\mu$).
\end{corollary}

We prove Proposition~\ref{prop:seq} in two steps. In the first step, using Proposition~\ref{prop:long}, we find a positive function $\mathpzc{s}_1$ tending to infinity and a coupling of $\mathcal{Z}^S$ with $\mathcal{Z}^{\infty,S}$ (respectively given in~\eqref{eq:ZT=ZZT} and~\eqref{eq:def_mathcal_Z^T}) such that we can control the distance between the processes up to time $\mathpzc{s}_1(S)$. We then find, for all $T$, a corresponding $S=S(T)$, which is smaller than $T$, and $\mathpzc{s}_2(S)$ such that inclusion~\eqref{eq:inclu} holds. By Proposition~\ref{prop:coupling}, there exists a coupling of $\mathcal{Z}^{\infty,S}$ with  $\mathcal{Z}^{\infty,\mu, S}$, where $\mathcal{Z}^{\infty,\mu, S}$ is some $S$-dependent copy of $\mathcal{Z}^{\infty,\mu}$, such that we can control the distance between $\mathcal{Z}^{\infty,S}$ and $\mathcal{Z}^{\infty,\mu, S}$ up to time $\mathpzc{s}_2$. Thus we control the distance between $\mathcal{Z}^S$ and $\mathcal{Z}^{\infty,\mu,S}$ on the smaller of the two intervals in inclusion~\eqref{eq:inclu}. 
 In the second step of the proof, we show that the distance between  $\mathcal{Z}^T$ and a random time-shift  of  the process $\mathcal{Z}^{\infty,\mu, S}$, which we prove is still in stationarity (and  thus a copy of  $\mathcal{Z}^{\infty,\mu}$), is a small perturbation of the distance between $\mathcal{Z}^S$ and $\mathcal{Z}^{\infty,\mu, S}$. The latter distance is small by  the first step of the proof. 
  
\begin{proof}[Proof of Proposition~\ref{prop:seq}]
By Proposition~\ref{prop:long} there exists a function $\mathpzc{s}_1:(0,\infty)\to(0,\infty)$ satisfying $\mathpzc{s}_1(S)\underset{S\to \infty}\longrightarrow \infty$ and such that
    \begin{equation}
    \label{eq:appliaciton_of_Prop_4.4}
    \| \cZ^{S}-\cZ^{\infty,{S}}\|_{[0, \mathpzc{s}_1({S})]}\overset{(proba)}{\underset{S\to \infty}\longrightarrow}  0.
    \end{equation}
    Recall $\mathcal{Z}^S$ and $\mathcal{Z}^{\infty,S}$, given in~\eqref{eq:ZT=ZZT} and~\eqref{eq:def_mathcal_Z^T} respectively, are driven by the same Brownian motion $\mathcal{W}^S$, defined in~\eqref{eq:BM_W^T} in Subsection~\ref{subsec:def} above. 
    
    By Lemma~\ref{le:sequenceRand} applied to the function $\mathpzc{s}_1$, there exist deterministic functions $\mathpzc{s}_2,S:(0,\infty)\to(0,\infty)$, and  a random time $T_0$ satisfying 
    $4\mathpzc{s}_2\leq \mathpzc{s}_1$, $|S(T)-T|=O(T^{\frac{2\beta}{1+\beta}}\log T)$ and $\mathpzc{s}_2(T)\to\infty$ as $T\to\infty$ and, for all $T>T_0$, inclusion~\eqref{eq:inclu} holds.
    By definition~\eqref{eq:BM_W^T}, we have  
    \begin{equation}  
    \label{eq:def_of_W^S}
    \mathcal{W}^{S}_s-\mathcal{W}^{S}_{2 \mathpzc{s}_2(S)}=(W_{S+b(X_S)^2s}-W_{S+2b(X_S)^2\mathpzc{s}_2(S)})/b(X_S)\quad\text{for all $s\geq 2 \mathpzc{s}_2(S)$,}
    \end{equation}
    where $W$ is the Brownian motion driving the initial SDER~\eqref{eq:SDER} for $(Z,L)$.
    Note that, conditional on $Z_S=(X_S,Y_S)$, the law of $(\mathcal{W}^{S}_{s+ 2 \mathpzc{s}_2(S)}-\mathcal{W}^{S}_{2 \mathpzc{s}_2(S)})_{s\in\RP}$ 
    is (by e.g.~L\'evy's characterisation) that of a standard Brownian motion, making the process independent of $Z_S=(X_S,Y_S$).
    More generally, the Brownian motion  $(\mathcal{W}^{S}_{s+ 2 \mathpzc{s}_2(S)}-\mathcal{W}^{S}_{2 \mathpzc{s}_2(S)})_{s\in\RP}$ is independent of the pair of processes  $((Z_u)_{u\leq S},(\mathcal{W}^{S}_s)_{s\leq 2 \mathpzc{s}_2(S)})$.

By Proposition~\ref{prop:coupling} applied at time $\mathpzc{s}= 2\mathpzc{s}_2(S)$ to the process $\mathcal{Z}^{\infty,S}$ defined in~\eqref{eq:def_mathcal_Z^T} (started at $(0,Y_S/b(X_S))$ and driven by the Brownian motion $\cW^S$ defined in~\eqref{eq:BM_W^T}) and supported  on the probability space on which $Z$ and $W$ were initially defined, there exists 
an extended probability space with probability measure $\mathbb{P}_S\coloneqq\mathbb{P}^{2 \mathpzc{s}_2(S)}$, supporting  
\begin{equation}
\label{eq:procs_from_Cor}
\text{a Brownian motion $\mathcal{W}'^S$ and a copy $\cZ^{\infty,(0,\widehat Y^S),\mathcal{W}'^S}$ of the process $\cZ^{\infty,\mu}$}
\end{equation}
(with $\widehat Y^S$ following the invariant distribution $\mu$, cf.~Remark~\ref{rem:Notation} above). Moreover, $\mathbb{P}_S$ satisfies properties (1)--(4) in  Proposition~\ref{prop:coupling} at time $2\mathpzc{s}_2(S)$, where $\nu$ is the law of $Y_S/b(X_S)$ and $\mathcal{Z}^{\infty,\nu,\mathcal{W}^S}=\mathcal{Z}^{\infty,S}$.

Since $2\mathpzc{s}_2(S)\to\infty$ as $S\to\infty$, by property (1) in  Proposition~\ref{prop:coupling}, we get
\begin{align}
\nonumber
\mathbb{P}_S & \big(   \forall s\geq  2 \mathpzc{s}_2(S),  \ \cZ^{\infty,{S}}_s-\cZ^{\infty,(0,\widehat Y^S),\mathcal{W}'^S}_s = (\cX^{\infty,{S}}_{2\mathpzc{s}_2(S)}   -\cX^{\infty,(0,\widehat Y^S),\mathcal{W}'^S}_{2\mathpzc{s}_2(S)}, 0_{\R^d} )  \big)\\
&\geq 1-2C_0\lambda^{2 \mathpzc{s}_2(S)} \underset{S\to \infty}\longrightarrow 1. 
\label{eq:application_of_cor_at_infty}
\end{align}
   
    For all $s_0,s \in [2 \mathpzc{s}_2(S),4 \mathpzc{s}_2(S)]$, the triangle inequality implies
\begin{align*}
& |\cZ^{S}_s -  \cZ^{\infty,(0,\widehat Y^S),\mathcal{W}'^S}_s - ({\cX}^{ S}_{s_0 } -{\cX}^{\infty,(0,\widehat Y^S),\mathcal{W}'^S}_{s_0 }, 0_{\R^d}    )   |_{1+d} \\
&\qquad \leq 
|\cZ^{S}_s -  \cZ^{\infty,S}_s  |_{1+d}
+|\cX^{S}_{s_0} -  \cX^{\infty,S}_{s_0}  |\\
& \qquad  \qquad+
|\cZ^{\infty,S}_s - \cZ^{\infty,(0,\widehat Y^S),\mathcal{W}'^S}_s - ({\cX}^{\infty, S}_{s_0}-{\cX}^{\infty,(0,\widehat Y^S),\mathcal{W}'^S}_{s_0}, 0_{\R^d} )|_{1+d}\\
&\qquad \leq 2 \sup_{u \leq 4 \mathpzc{s}_2(S) }|\cZ^{S}_u -  \cZ^{\infty,S}_u  |_{1+d}\\
& \qquad \qquad +\sup_{u\geq 2 \mathpzc{s}_2(S) }|\cZ^{\infty,S}_u - \cZ^{\infty,(0,\widehat Y^S),\mathcal{W}'^S}_s - ({\cX}^{\infty, S}_{s_0}-{\cX}^{\infty,(0,\widehat Y^S),\mathcal{W}'^S}_{s_0}, 0_{\R^d} )|_{1+d}.
\end{align*}
For any $\epsilon>0$, we thus obtain
    \begin{align}
    &\mathbb{P}_S\big( \sup_{s_0,s\in [2\mathpzc{s}_2(S),4 \mathpzc{s}_2(S)]}
    |\cZ^{S}_s -  \cZ^{\infty,(0,\widehat Y^S),\mathcal{W}'^S}_s - ({\cX}^{ S}_{s_0 } -{\cX}^{\infty,(0,\widehat Y^S),\mathcal{W}'^S}_{s_0 }, 0_{\R^d}    )   |_{1+d}  \geq  \epsilon  \big)\nonumber\\
         &\leq \mathbb{P}_S\big( \exists s\leq 4 \mathpzc{s}_2(S):  |\cZ^{S}_s -  \cZ^{\infty,S}_s  |_{1+d}  \geq \epsilon /2 \big)\nonumber \\
     &+\mathbb{P}_S\big( \exists s\geq 2 \mathpzc{s}_2(S):   \cZ^{\infty,S}_s - \cZ^{\infty,(0,\widehat Y^S),\mathcal{W}'^S}_s \neq ({\cX}^{\infty, S}_{2 \mathpzc{s}_2(S)}-{\cX}^{\infty,(0,\widehat Y^S),\mathcal{W}'^S}_{2 \mathpzc{s}_2(S)}, 0_{\R^d}    )   \big) \underset{S\to \infty}\longrightarrow 0, \label{eq:cvGamma4}
    \end{align}
    where the first and second summands in~\eqref{eq:cvGamma4} tend to zero by~\eqref{eq:appliaciton_of_Prop_4.4} and~\eqref{eq:application_of_cor_at_infty}, respectively.
    Note that~\eqref{eq:cvGamma4} looks much like the conclusion of the proposition we are proving, with a major difference nonetheless: here the range in the supremum does not start at $s=0$, but rather at $2\mathpzc{s}_2(S)$, which tends to infinity as $S\to\infty$. As we shall see, this issue will be remedied by the function $S=S(T)$, obtained from  the application of Lemma~\ref{le:sequenceRand}.

Define a random time 
\begin{equation}
\label{eq:defs0}
s_0=s_0(T)\coloneqq (T-S(T))b(X_{S(T)})^{-2},
\end{equation} 
which (by the inclusion in Lemma~\ref{le:sequenceRand}) satisfies $s_0\in [2 \mathpzc{s}_2(S),4 \mathpzc{s}_2(S)]$
for all large $T$ almost surely.

\smallskip

\paragraph*{\textbf{Claim 1}} \label{para:Claim_1} For $s_0$ in~\eqref{eq:defs0}, the random vector  
\begin{equation}
\label{eq:stationary_procesess_cylinder}
    \tilde{Y}^T\coloneqq \mathcal{Y}^{\infty, \widehat Y^S , \mathcal{W}'^S}_{s_0}\quad\text{ follows the law $\mu$ supported on $\BB^d$.}
\end{equation}
The Brownian motion $\tilde{\mathcal{W}}^T=(\tilde{\mathcal{W}}^T_t)_{t\in\RP}$, given by $\tilde{\mathcal{W}}^T_t\coloneqq  {\mathcal{W}}'^S_{t+s_0} -  {\mathcal{W}}'^S_{s_0}$, is independent of $\tilde{Y}^T$, making 
the process $\mathcal{Y}^{\infty, \tilde{Y}^T ,\tilde{\mathcal{W}}^T }$  a stationary solution of SDER~\eqref{eq:limitEq} started at $\tilde{Y}^T$.

\smallskip

\noindent \underline{Proof of \nameref{para:Claim_1}}.
Since $s_0$ is a function of $X_S$ and thus random, the claim in~\eqref{eq:stationary_procesess_cylinder} does not follow directly from the stationarity of the process
$\mathcal{Y}^{\infty, \widehat Y^S , \mathcal{W}'^S}$ defined in~\eqref{eq:procs_from_Cor}.
We first establish the following fact:
\begin{equation}
\label{eq:X_S_independent_of_Y^S_hat}
\text{$X_S$ and  $(\widehat Y^S,\cW'^S)$ are independent.}
\end{equation}

Recall the chain rule~\cite[Prop.~6.8]{Kallenberg} for sub-$\sigma$-fields $\cH$, $\cG$, $\cF_1,\cF_2$ in a probability space:
\begin{equation}
    \label{eq:chain_rule}
\cH\indep_{\cG}\sigma(\cF_1,\cF_2) \iff \cH\indep_{\cG}\cF_1\quad \& \quad \cH\indep_{\sigma(\cG,\cF_1)}\cF_2,
\end{equation}
where $\sigma(\cF_1,\cF_2)$ is the $\sigma$-field generated by the subsets in $\cF_1\cup\cF_2$ and $\cH\indep_\cG\cF_1$ denotes the conditional independence of $\cH$ and $\cF_1$, given $\cG$ (see e.g.~\cite[p.~109]{Kallenberg} for definition). If $\cG$ is trivial, then~\eqref{eq:chain_rule} states
that $\cH\indep\sigma(\cF_1,\cF_2)$ if and only if $\cH\indep\cF_1$ and $\cH\indep_{\cF_1}\cF_2$.

Recall that $\tilde Y_0=Y_S/b(X_S)$ and the definition of $\cW^S$ in~\eqref{eq:BM_W^T} above. Since $\cW^S$ is independent of $Z_S=(X_S,Y_S)$, we have
$\cW^S\indep\sigma(\tilde Y_0, X_S)$. Thus $X_S\indep_{\tilde Y_0}\cW^S$ by~\eqref{eq:chain_rule} (with trivial $\cG$, $\cF_1=\sigma(\tilde Y_0)$, $\cF_2=\sigma(X_S)$ and $\cH=\sigma(\cW^S)$). Furthermore, by property~(4) in Proposition~\ref{prop:coupling} and definition~\eqref{eq:procs_from_Cor} of $(\widehat Y^S,\cW'^S)$, we have 
$X_S\indep_{\sigma(\tilde Y_0,\cW^S)}\sigma(\widehat Y^S,\cW'^S)$.
Combining the two conditional independence statements via~\eqref{eq:chain_rule} (with $\cG=\sigma(\tilde Y_0)$, $\cH=\sigma(X_S)$, $\cF_1=\sigma(\cW^S)$ and $\cF_2=\sigma(\widehat Y^S,\cW'^S)$), we obtain $X_S\indep_{\tilde Y_0}\sigma(\widehat Y^S,\cW'^S,\cW^S)$. In particular, the following holds:
\begin{equation}
\label{eq:cond_main_indep}
X_S\indep_{\tilde Y_0} \sigma(\widehat Y^S,\cW'^S).   
\end{equation}
By property~(3) in Proposition~\ref{prop:coupling} and definition~\eqref{eq:procs_from_Cor} above, we have 
$\tilde Y_0\indep\sigma(\widehat Y^S,\cW'^S)$. This independence, together with~\eqref{eq:cond_main_indep}, yields~\eqref{eq:X_S_independent_of_Y^S_hat} via the chain rule in~\eqref{eq:chain_rule} (with trivial $\cG$).

By definition~\eqref{eq:procs_from_Cor},
the process $\cY^{\infty,\widehat Y^S,\cW'^S}$
is in stationarity. 
Since the time $s_0=s_0(T)=(T-S)/b(X_S)^2$ is a deterministic function of
$X_S$,
by~\eqref{eq:X_S_independent_of_Y^S_hat}, the unique strong solution $\cY^{\infty,\widehat Y^S,\cW'^S}$ of 
SDER~\eqref{eq:limitEq}  
 is independent of 
the time $s_0$, making 
the random vector 
$\tilde Y^T=\cY^{\infty,\widehat Y^S,\cW'^S}_{s_0}$ follow the stationary law $\mu$.
L\'evy's characterisation implies that 
$\tilde{\mathcal{W}}^T$ is a Brownian motion.
The strong uniqueness of solutions of 
SDER~\eqref{eq:limitEq}, together with~\eqref{eq:stationary_procesess_cylinder}, implies the final statement in~\nameref{para:Claim_1}.
\hfill$\diamondsuit$

\smallskip

Define $s(t,T)\coloneqq s_0(T)+t \frac{b(X_{T})^2}{b(X_{S(T)})^2}$, where we recall
$s_0(T)=(T-S)/b(X_S)^2$. Unless otherwise stated, to simplify the notation, we suppress the dependence on $T$ in these functions (e.g., $S=S(T), s_0=s_0(T), s(t)=s(t,T)$).
For $T\geq T_0$ and $t\in [0,{\mathpzc{s}_2}(T)]$, by inclusion~\eqref{eq:inclu} in Lemma~\ref{le:sequenceRand} the following inequalities hold:  $2\mathpzc{s}_2(S)\leq s(t,T)\leq 4\mathpzc{s}_2(S)$, see Figure~\ref{fig:sequence}. 
    In particular, since by Lemma~\ref{le:sequenceRand}, $S(T)\to\infty$ and $\mathpzc{s}_2(S)\to\infty$ as $S\to\infty$, we have $s_0(T)\to\infty$ almost surely as $T\to\infty$.

 By definition, the times $s(t)$ and $s_0$ satisfy  $T+b(X_T)^2 t=S+b(X_S)^2 s(t)$ and $T=S+b(X_S)^2 s_0$. Moreover, by definition~\eqref{eq:Z^T-def}, we get
    \begin{equation}
    \label{eq:reparam_y_component}
  \mathcal{Y}^T_t=  b(X_T)^{-1} Y_{T+tb(X_T)^2}=\frac{b(X_S)}{b(X_T)} \mathcal{Y}^S_{s(t)}.
    \end{equation}
    Similarly, again by~\eqref{eq:Z^T-def}, we have 
$\mathcal{X}^T_t=\frac{1}{b(X_T)}(X_{T+b(X_T)^2 t} -X_T)=\frac{1}{b(X_T)}(X_{S+b(X_S)^2 s(t)} -X_T)$,
$\mathcal{X}^S_{s(t)}=\frac{1}{b(X_S)}(X_{S+b(X_S)^2 s(t)} -X_S)$ and
$\mathcal{X}^S_{s_0}=\frac{1}{b(X_S)}(X_{T} -X_S)$,
from which we deduce 
\begin{equation}
    \label{eq:reparam2}
\mathcal{X}^T_t= \frac{b(X_S)}{b(X_T)}(\mathcal{X}^S_{s(t)}-\mathcal{X}^S_{s_0}).
\end{equation}

\begin{figure}
    \centering
    \includegraphics[width=\textwidth]{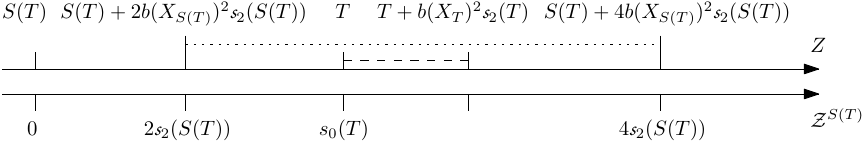}
\caption{Diagram of the interval inclusion in~\eqref{eq:inclu}, together with various times involved in the proof of Proposition~\ref{prop:seq}. On the top (resp.\ bottom), are marked times corresponding to the original process~$Z$ (resp.\ rescaled process~$\mathcal{Z}^{S}$). The diagram shows the ordering of times on top and bottom, not their scale (e.g.~$T$ is different from $s_0(T)= (T-S(T))b(X_{S(T)})^{-2}$ in~\eqref{eq:defs0}).
The smaller (resp.\ larger) interval  in the inclusion in~\eqref{eq:inclu} is indicated by the dashed (resp.\ dotted) line. 
Proposition~\ref{prop:long} controls the distance between $
\mathcal{Z}^{S(T)}$ and $
\mathcal{Z}^{\infty,S(T)}$ on the large interval $[0,4\mathpzc{s}_2(S(T))]$, which corresponds to the interval $[S(T), S(T)+ 4 b(X_{S(T)})^2 \mathpzc{s}_2(S(T))]$ for the original process $Z$.
In the proof of Proposition~\ref{prop:seq} we control the distance between $
\mathcal{Z}^{T}$ and $
\mathcal{Z}^{\infty, \mu}$ (the process started from the invariant measure $\mu$
of SDER~\eqref{eq:limitEq}, see Corollary~\ref{cor:mu-existence-convergence} above) on the smaller interval $[0,\mathpzc{s}_2(T)]$, which corresponds to the interval $[T, T+  b(X_{T})^2 \mathpzc{s}_2(T)]$ (resp.\  $[s_0(T),s_0(T)+\mathpzc{s}_2(T)b(X_T)^2/b(X_{S(T)})^2]$) for the original process $Z$ (resp. rescaled process $\cZ^S$).
}
\label{fig:sequence}
\end{figure}

Note that  SDER~\eqref{eq:limitEqJoint_Integral}  possesses strong uniqueness, implying
that $(\tilde Y^T,\tilde \cW^T)$ in~\nameref{para:Claim_1} satisfies
    \begin{equation*}
    \mathcal{Z}^{\infty, (0,\tilde{Y}^T) ,\tilde{\mathcal{W}}^T }_t= \mathcal{Z}^{\infty, (0,\widehat Y^S) , \mathcal{W}'^S}_{t+s_0}- (\mathcal{X}^{\infty, (0,\widehat Y^S) , \mathcal{W}'^S}_{s_0} ,0_{\R^d})\qquad\text{for all $t\in\RP$,}
    \end{equation*}
where $(\widehat Y^S,\cW'^S)$ are defined in~\eqref{eq:procs_from_Cor}.
Since $s(t)-s_0=\frac{b(X_T)^2}{b(X_S)^2}t$, we deduce 
\begin{equation} 
\label{eq:reparam}
\mathcal{Y}^{\infty, \widehat Y^S,\mathcal{W}'^S}_{s(t)}={\mathcal{Y}}^{\infty, \tilde{Y}^T ,\tilde{\mathcal{W}}^T }_{ \frac{b(X_T)^2}{b(X_S)^2} t }, \qquad 
\mathcal{X}^{\infty, (0,\widehat Y^S),\mathcal{W}'^S}_{s(t)}-\mathcal{X}^{\infty, (0,\widehat Y^S),\mathcal{W}'^S}_{s_0}={\mathcal{X}}^{\infty, (0,\tilde{Y}^T) ,\tilde{\mathcal{W}}^T }_{ \frac{b(X_T)^2}{b(X_S)^2} t }.
\end{equation}

We now prove that $\|\cZ^{T} -  \cZ^{\infty,(0,\tilde{Y}^T),\tilde{\mathcal{W}}^T}  \|_{[0,\mathpzc{s}_3(T)]} \overset{(proba)}{\underset{T\to \infty}\longrightarrow} 0$
for some function $\mathpzc{s}_3:(0,\infty)\to(0,\infty)$ satisfying $\mathpzc{s}_3(T)\to\infty$ as $T\to\infty$.
By Lemma~\ref{le:extend}, it suffices to prove 
$$\|\cZ^{T} -  \cZ^{\infty,(0,\tilde{Y}^T),\tilde{\mathcal{W}}^T}  \|_{[0,s]} \overset{(proba)}{\underset{T\to \infty}\longrightarrow} 0$$
for every fixed $s\in(0,\infty)$.
We first consider the $\mathpzc{y}$ coordinate of the processes in the last display.
By~\eqref{eq:reparam_y_component} and~\eqref{eq:reparam} we obtain
    \begin{align*}
    \| \mathcal{Y}^T -  \mathcal{Y}^{\infty, \tilde{Y}^T ,\tilde{\mathcal{W}}^T }  \|_{[ 0, s]} 
    & \leq 
    \sup_{t\in[ 0, s]} \bigl( | \mathcal{Y}^T_t
    - \mathcal{Y}^{\infty, \widehat Y^S,\mathcal{W}'^S}_{s(t)} |_d+ |
    \mathcal{Y}^{\infty, \tilde{Y}^T ,\tilde{\mathcal{W}}^T } _t- \mathcal{Y}^{\infty, \widehat Y^S,\mathcal{W}'^S}_{s(t)} |_d \bigr) \\
    & \leq
    \sup_{t\in[ 0, s]}  \Bigl| \frac{b(X_S)}{b(X_T)} \mathcal{Y}^S_{s(t)}
    - \mathcal{Y}^{\infty, \widehat Y^S,\mathcal{W}'^S}_{s(t)}  \Bigr|_d \\
  &\qquad \qquad + 
   \sup_{t\in[ 0, s]} \Bigl|
    \mathcal{Y}^{\infty, \tilde{Y}^T ,\tilde{\mathcal{W}}^T } _t- \mathcal{Y}^{\infty, \tilde{Y}^T ,\tilde{\mathcal{W}}^T } _{ \frac{b(X_T)^2}{b(X_S)^2} t } \Bigr|_d \\
    & \leq  A_1(T)+A_2(T)+A_3(T),
    \end{align*}
    where the first summands equals
    $A_1(T)\coloneqq |\frac{b(X_S)}{b(X_T)}-1|  \sup_{t\in[ 0, s]}| \mathcal{Y}^{\infty, \widehat Y^S,\mathcal{W}'^S}_{s(t)}|_d$, the second summand is given by
    $A_2(T)\coloneqq \frac{b(X_S)}{b(X_T)}  \sup_{t\in[ 0, s]} | \mathcal{Y}^S_{s(t)} - \mathcal{Y}^{\infty,  \widehat Y^S,\mathcal{W}'^S}_{s(t)}  |_d$
    and the third is
    \begin{equation}
    \label{eq:A_3(T)}
    A_3(T)\coloneqq  \sup_{t\in[ 0, s]} \Bigl|
    \mathcal{Y}^{\infty, \tilde{Y}^T ,\tilde{\mathcal{W}}^T } _t- \mathcal{Y}^{\infty, \tilde{Y}^T ,\tilde{\mathcal{W}}^T }_{ \frac{b(X_T)^2}{b(X_S)^2} t } \Bigr|_d.
    \end{equation}
    Since $\| \mathcal{Y}^{\infty, \widehat Y^S,\mathcal{W}'^S}\|_{[0,\infty)}\leq 1$, we have $A_1(T)\to0$ almost surely as $T\to  \infty$ by the strong law in~\eqref{def:c1}, Lemma~\ref{le:sequenceRand} (which implies  $S(T)/T\sim1$ as $T\to\infty$) and the asymptotic behaviour of $b$ in Assumption~\eqref{hyp:D3}.
    In addition, by~\eqref{eq:cvGamma4},  for $T$ sufficiently large so that $\mathpzc{s}_2(T)\geq s$, we have
    \begin{align} 
    0\leq A_2(T)& \leq 2 \sup_{t\in [ 0, \mathpzc{s}_2(T)]} | \mathcal{Y}^S_{s(t)} - \mathcal{Y}^{\infty,\widehat Y^S,\mathcal{W}'^S}_{s(t)}  |_d +\infty \mathbbm{1}_{\frac{ b(X_S)}{b(X_T)}\geq 2  } \label{eq:temp:bratio}\\ 
    &\leq 
    \| \mathcal{Y}^S - \mathcal{Y}^{\infty,\widehat Y^S,\mathcal{W}'^S}  \|_{[ 2\mathpzc{s}_2(T), 4\mathpzc{s}_2(T) ]}+\infty \mathbbm{1}_{\frac{ b(X_S)}{b(X_T)}\geq 2  } \underset{T\to \infty}{\overset{(proba)}{\longrightarrow}} 0,\nonumber
    \end{align}
    since  $b(X_S)/b(X_T)\to1$ almost surely as $T\to\infty$.

\paragraph*{\textbf{Claim 2}} \label{para:Claim_2} $A_3(T)\underset{T\to \infty}{\overset{(proba)}{\longrightarrow}} 0$, where $A_3(T)$ is defined in~\eqref{eq:A_3(T)} above.

\smallskip

\noindent \underline{Proof of \nameref{para:Claim_2}}. Recall that, for any $\delta>0$ and $m>0$,  the modulus of continuity of a function $f:[0,\infty)\to \mathbb{R}^k$ is given by
    \[ \omega_{\delta,m}(f)\coloneqq  \sup \{ |f(u)-f(v)|_k: u,v \in [0,m], |u-v|\leq \delta\}.
    \]
     Since $\mathcal{Z}^{\infty, \mu}$ is almost surely continuous, for any given $s>0$ we have  $\omega_{\delta,s}(\mathcal{Z}^{\infty, \mu})\underset{\delta\to 0} \longrightarrow 0$ almost surely, hence also in probability. By~\eqref{eq:stationary_procesess_cylinder} in \nameref{para:Claim_1}, for every $T>0$, the law of $\mathcal{Z}^{\infty, (0,\tilde{Y}^T) ,\tilde{\mathcal{W}}^T }$
     equals that of $\mathcal{Z}^{\infty, \mu}$ (cf.~Remark~\ref{rem:Notation} above).
     Thus, for every $\epsilon>0$ and $s>0$, there exists $\delta>0$ such that
    \[ 
   \mathbb{P}_S(\omega_{\delta,s}(\mathcal{Z}^{\infty, (0,\tilde{Y}^T) ,\tilde{\mathcal{W}}^T }) \geq \epsilon )=
    \mathbb{P}(\omega_{\delta,s}(\mathcal{Z}^{\infty, \mu}) \geq \epsilon )\leq \epsilon\quad\text{for all $T>0$.}
    \]

    For any constant $s>0$ and $\delta\in(0,1)$, on the event that $|\frac{b(X_T)^2}{b(X_S)^2}-1| s\leq \delta$, it holds
    \begin{align*} \sup_{t\in[ 0, s]} \Bigl|
    \mathcal{Z}^{\infty, (0,\tilde{Y}^T) ,\tilde{\mathcal{W}}^T }_t- \mathcal{Z}^{\infty, (0,\tilde{Y}^T) ,\tilde{\mathcal{W}}^T }_{ \frac{b(X_T)^2}{b(X_S)^2} t } \Bigr|_{1+d} & \leq \omega_{  \bigl|\frac{b(X_T)^2}{b(X_S)^2}-1\bigr| s,s+\delta } (\mathcal{Z}^{\infty, \tilde{Y}^T ,\tilde{\mathcal{W}}^T } )\\
    & \leq   \omega_{\delta, s+1 } (\mathcal{Z}^{\infty, (0,\tilde{Y}^T) ,\tilde{\mathcal{W}}^T } ). 
    \end{align*}
    Since $S=S(T)\sim T$ as $T\to\infty$ by Lemma~\ref{le:sequenceRand}, by
   decomposing the probability space depending on whether $|\frac{b(X_T)^2}{b(X_S)^2}-1| s$ is larger or smaller than $ \delta$, we obtain
    \begin{align*}    
    &\limsup_{T\to \infty}\mathbb{P} \Bigl(  \sup_{t\in[ 0, s]} \bigl|
    \mathcal{Z}^{\infty, (0,\tilde{Y}^T) ,\tilde{\mathcal{W}}^T }_t- \mathcal{Z}^{\infty, (0,\tilde{Y}^T) ,\tilde{\mathcal{W}}^T }_{ \frac{b(X_T)^2}{b(X_S)^2} t } \bigr|_{1+d}\geq \epsilon  \Bigr) \\
    & \qquad \leq 
    \limsup_{T\to \infty} \mathbb{P}\Bigl( \Bigl| \frac{b(X_T)^2}{b(X_S)^2}-1\Bigr| s \geq \delta\Bigr)
     + \mathbb{P} \bigl(\omega_{\delta, s+1} \bigl(\mathcal{Z}^{\infty, (0,\tilde{Y}^T) ,\tilde{\mathcal{W}}^T } \bigr)\geq \epsilon\bigr) \\ 
     & \qquad =  \mathbb{P}\bigl(\omega_{\delta,s}(\mathcal{Z}^{\infty, \mu})\geq \epsilon\bigr)\leq \epsilon.
    \end{align*}
    Hence for any $s>0$ we get
    \begin{equation} 
    \label{eq:Zcontinuitymod}
    \sup_{t\in[ 0, s]} |
   \mathcal{Z}^{\infty, (0,\tilde{Y}^T) ,\tilde{\mathcal{W}}^T }_t- \mathcal{Z}^{\infty, (0,\tilde{Y}^T) ,\tilde{\mathcal{W}}^T }_{ \frac{b(X_T)^2}{b(X_S)^2} t }|_{1+d}\overset{(proba)}{\underset{T\to \infty}\longrightarrow} 0,
   \end{equation}
   implying that $A_3(T)$, defined in~\eqref{eq:A_3(T)}, converges to zero in probability. \hfill$\diamondsuit$

   \smallskip

    As for the $\mathpzc{x}$ coordinate, we proceed similarly. Recall $(\tilde Y^T,\tilde \cW^T)$ introduced in~\nameref{para:Claim_1} above.
Using the triangle inequality twice and then~\eqref{eq:reparam2} and~\eqref{eq:reparam},  we deduce 
\begin{align}
\bigl| \mathcal{X}^T_t - \mathcal{X}^{\infty,(0,\tilde{Y}^T),\tilde{\mathcal{W}}^T }_t \bigr|
 & \leq 
\Bigl| \mathcal{X}^T_t- \mathcal{X}^{\infty,(0,\tilde{Y}^T),\tilde{\mathcal{W}}^T }_{\frac{b(X_T)^2}{b(X_S)^2}t} \Bigr| + \Bigl| \mathcal{X}^{\infty,(0,\tilde{Y}^T),\tilde{\mathcal{W}}^T }_{\frac{b(X_T)^2}{b(X_S)^2}t}- \mathcal{X}^{\infty,(0,\tilde{Y}^T),\tilde{\mathcal{W}}^T }_t \Bigr| \nonumber\\
 & \leq 
 \Bigl|\frac{b(X_S)}{b(X_T)}-1 \Bigr|\cdot \Bigl| 
\mathcal{X}^{\infty,(0,\tilde{Y}^T),\tilde{\mathcal{W}}^T }_{\frac{b(X_T)^2}{b(X_S)^2}t} \Bigr| +
\Bigl|\mathcal{X}^T_t- \frac{b(X_S)}{b(X_T)} \mathcal{X}^{\infty,(0,\tilde{Y}^T),\tilde{\mathcal{W}}^T }_{\frac{b(X_T)^2}{b(X_S)^2}t}\Bigr| \nonumber \\
& {} \qquad + \Bigl| \mathcal{X}^{\infty,(0,\tilde{Y}^T),\tilde{\mathcal{W}}^T }_{\frac{b(X_T)^2}{b(X_S)^2}t}- \mathcal{X}^{\infty,(0,\tilde{Y}^T),\tilde{\mathcal{W}}^T }_t \Bigr| \nonumber\\
 &=
 \Bigl| \frac{b(X_S)}{b(X_T)}-1 \Bigr| \cdot \Bigl| \mathcal{X}^{\infty,(0,\tilde{Y}^T),\tilde{\mathcal{W}}^T }_{\frac{b(X_T)^2}{b(X_S)^2}t} \Bigr|+ \Bigl| \mathcal{X}^{\infty,(0,\tilde{Y}^T),\tilde{\mathcal{W}}^T }_{\frac{b(X_T)^2}{b(X_S)^2}t}- \mathcal{X}^{\infty,(0,\tilde{Y}^T),\tilde{\mathcal{W}}^T }_t \Bigr| \nonumber \\
 &\qquad + \frac{b(X_S)}{b(X_T)}
\big|(\mathcal{X}^S_{s(t)}-\mathcal{X}^S_{s_0})
-(\mathcal{X}^{\infty, (0,\widehat Y^S),\mathcal{W}'^S}_{s(t)}-\mathcal{X}^{\infty, (0,\widehat Y^S),\mathcal{W}'^S}_{s_0})\big| .
  \label{eq:temp:b123}
\end{align}
Taking the supremum over $t\in [0,s]$ in~\eqref{eq:temp:b123}, we deduce
\[
    \| \mathcal{X}^T - \mathcal{X}^{\infty,(0,\tilde{Y}^T),\tilde{\mathcal{W}}^T } \|_{[ 0, s]}
\leq B_1(T)+B_2(T) +B_3(T),\]
where $B_1(T),B_2(T),B_3(T)$ are the suprema over $t\in [0,s]$ of the three summands in~\eqref{eq:temp:b123}. 

    The term $B_3(T)$ tends to $0$ in probability as $T \to \infty$ by~\eqref{eq:Zcontinuitymod}. 
    The term $B_2(T)$ also tends to $0$ in probability  by~\eqref{eq:cvGamma4} since, for all $T$ such that $\mathpzc{s}_2(T)>s$, $t\in [ 0, s]$ implies
    $s(t)\in [2\mathpzc{s}_2(S(T)),4\mathpzc{s}_2(S(T)) ]$ and the following limits hold by Lemma~\ref{le:sequenceRand}:  as $T\to \infty$ we have
    $S=S(T)\to \infty$ and $\mathpzc{s}_2(T)\to\infty$.

    The convergence $B_1(T)\to 0$ is less immediate than that of $A_1(T)\to0$ above: unlike $\mathcal{Y}^{\infty,(0,\tilde{Y}^T),\tilde{\mathcal{W}}^T}$, the process $\mathcal{X}^{\infty,(0,\tilde{Y}^T),\tilde{\mathcal{W}}^T}$ is not bounded. We conclude the proof of the proposition by establishing
    \begin{equation}
    \label{eq:conv_B_1(T)}
    B_1(T)= \Bigl|\frac{b(X_S)}{b(X_T)}-1 \Bigr| \sup_{t\in[ 0, s]} \Bigl| \mathcal{X}^{\infty,(0,\tilde{Y}^T),\tilde{\mathcal{W}}^T }_{\frac{b(X_T)^2}{b(X_S)^2}t} \Bigr|\underset{T\to \infty}{\overset{(proba)}{\longrightarrow}} 0.
    \end{equation}

    By~\nameref{para:Claim_1}, the distribution of the process $\mathcal{X}^{\infty, (0,\tilde{Y}^T),\tilde{\mathcal{W}}^T }$ does not depend on $T$. In particular, 
    $\|\mathcal{X}^{\infty, (0,\tilde{Y}^T),\tilde{\mathcal{W}}^T }\|_{[0,2s]}$ is an almost-surely finite random variable whose distribution does not depend on $T$.
    Hence, for every $\epsilon>0$, there exists $C_\epsilon\in(0,\infty)$ such that the event $E_{\epsilon,T}\coloneqq \{\|\mathcal{X}^{\infty, (0,\tilde{Y}^T),\tilde{\mathcal{W}}^T }\|_{[0,2s]}>C_\epsilon\}$ has small probability uniformly in the parameter $T$: $\mathbb{P}_S(E_{\epsilon,T})\leq \epsilon/2$ for all $T>0$.
    By decomposing the probability space according to $E_{\epsilon,T}$,
    we obtain
    \begin{align*}
    \mathbb{P}_S \Bigl( \Bigl|\frac{b(X_S)}{b(X_T)}-1 \Bigr|\cdot \| \mathcal{X}^{\infty, (0,\tilde{Y}^T),\tilde{\mathcal{W}}^T }\|_{[ 0, 2 s]}&\geq\epsilon\Bigr)\leq \mathbb{P}_S(E_{\epsilon,T})+
    \mathbb{P}_S\Bigl( \Bigl|\frac{b(X_S)}{b(X_T)}-1\Bigr|\geq\epsilon/C_\epsilon\Bigr)\\
    & \leq \epsilon/2 +   \mathbb{P}_S \Bigl(\Bigl|\frac{b(X_S)}{b(X_T)}-1\Bigr|\geq\epsilon/C_\epsilon\Bigr).
    \end{align*}
Since furthermore $|\frac{b(X_S)}{b(X_T)}-1|$ converges in probability to $0$ as $T\to\infty$, it follows that
    \[
    \Bigl|\frac{b(X_S)}{b(X_T)}-1 \Bigr|\cdot \| \mathcal{X}^{\infty, (0,\tilde{Y}^T),\tilde{\mathcal{W}}^T }\|_{[ 0, 2 s]} \overset{(proba)}{\underset{T\to \infty}\longrightarrow} 0.
    \]
    
    As in \eqref{eq:temp:bratio} above, since 
    $\frac{b(X_S)^2}{b(X_T)^2}\leq 2$ with probability tending to $1$ as $T\to \infty$, we deduce that  the inequality $$\sup_{t\in[ 0, s]}\Bigl| \mathcal{X}^{\infty,(0,\tilde{Y}^T),\tilde{\mathcal{W}}^T }_{\frac{b(X_T)^2}{b(X_S)^2}t}\Bigr|\leq \| \mathcal{X}^{\infty, (0,\tilde{Y}^T),\tilde{\mathcal{W}}^T }\|_{[ 0, 2 s]}$$
   also holds with  probability tending to $1$ as $T\to\infty$, implying~\eqref{eq:conv_B_1(T)}.
\end{proof}

\subsection{Asymptotic independence}
\label{subsec:Asymptotic_indep}
As we have established the convergence of the rescaled random vector $Y_T/b(X_T)$ in Corollary~\ref{coro:Ycv}, we now consider the \emph{joint} convergence (after appropriate centring and scaling) of the vector $(X_T,Y_T)$, even though we have not yet established the convergence of the first component. (The latter we do in Section~\ref{sec:central}.) The reason for establishing Proposition~\ref{prop:joint}  in the present subsection, rather than later, is that its proof relies on arguments and results developed earlier in Section~\ref{sec:local} above.

\begin{proposition}
    \label{prop:joint} 
    Assume that 
    there exists a probability distribution $\rho$ on the real line such that, as $T\to \infty$, the quotient 
    $\frac{X_T- c_1 T^{\frac{1}{1+\beta}}}{\sqrt{T}}$ converges weakly to $\rho$.
    Then, as $T\to \infty$, the couple $$\Big(\frac{X_T- c_1 T^{\frac{1}{1+\beta}}}{\sqrt{T}},  \frac{Y_T}{a_\infty c_1^\beta T^{\beta/(1+\beta)}}\Big)$$ converges jointly in distribution to the product $\rho\otimes \mu$ (where $\mu$ is the stationary measure of SDER~\eqref{eq:limitEq}, see Corollary~\ref{cor:mu-existence-convergence} above).
\end{proposition}

 We will  prove in Subsection~\ref{subsec:Conclusion} below that the weak convergence of $\tilde X$, assumed in  Proposition~\ref{prop:joint}, holds when $\beta>-\frac{1}{3}$ with $\rho$ the centred Gaussian distribution with variance in~\eqref{eq:N-var} above. This final step will conclude the proof of Theorem~\ref{th:main1}.

By the strong law in~\eqref{def:c1} and  Assumption~\eqref{hyp:D3} on the boundary function $b$, we have $a_\infty c_1^\beta T^{\beta/(1+\beta)}/b(X_T)\to1$ almost surely as $T\to\infty$.
It thus suffices to prove that the random vector 
\begin{equation}
\label{eq:X_tilde_Y_tilde}
(\tilde{X}_T,\tilde{Y}_T)\coloneqq \big(\frac{X_T- c_1 T^{\frac{1}{1+\beta}}}{\sqrt{T}},  \frac{Y_T}{b(X_T)}\big)
\end{equation}
has, as $T\to\infty$, the limit distribution given in Proposition~\ref{prop:joint}.

Our proof of Proposition~\ref{prop:joint}, which essentially states that $\tilde{X}_T$ and $\tilde{Y}_T$ are asymptotically independent, relies on the fact that the mixing time of $\tilde{X}$ is much larger than the mixing time of~$\tilde{Y}$. In the time windows from $S$ to $T$, where (by Lemma~\ref{le:sequenceRand} above) $S=S(T)$ 
satisfies $|S(T)-T|=O( T^\frac{2\beta}{1+\beta}\log T )$,
we will show that $X$ hardly fluctuates (see Lemma~\ref{le:claimtech} below). On the contrary, since the mixing time of $\tilde{Y}$ is of order $T^{\frac{2\beta}{1+\beta}}$, which is much smaller than $T-S$ (recall from Lemma~\ref{le:sequenceRand} that $T-S\geq 2\mathpzc{s}_2(S)b(X_S)^2$ is lower bounded by a multiple of $\mathpzc{s}_2(S)T^{\frac{2\beta}{1+\beta}}$ as $T\to\infty$ and $\mathpzc{s}_2(S)\to\infty$), the value of $\tilde{Y}_T$ is almost independent from $Z_S=(X_S,Y_S)$ (cf.\  Lemma~\ref{le:claimtech2} below). In particular, since $\tilde{X}_T$ is approximately equal to $\tilde{X}_S$,  
the asymptotic independence between $\tilde X_T$ and $\tilde Y_T$ follows.

Throughout the present subsection, we fix $\mathpzc{s}_1$ the function given by Proposition~\ref{prop:long}, and the functions $\mathpzc{s}_2$ and $S$  produced by Lemma~\ref{le:sequenceRand}, when applied to that function $\mathpzc{s}_1$. When a sequence $(T_n)_{n\in \mathbb{N}}$ is given, we define $S_n\coloneqq S(T_n)$. 

\begin{lemma}
\label{le:claimtech}
Let $(T_n)_{n\in \mathbb{N}}$ be a deterministic sequence such that $T_n\underset{n\to \infty}\longrightarrow \infty$. Assume that 
$\tilde{X}_{S_n}$, defined in~\eqref{eq:X_tilde_Y_tilde}, converges in distribution, as $n\to \infty$. Then, 
in probability,    \[
    \tilde{X}_{T_n}-\tilde{X}_{S_n}\underset{n\to \infty}\longrightarrow 0.
    \]
\end{lemma}
\begin{proof}
    Let $T>0$ and $S=S(T)$.
    Recall that $\mathcal{Z}^{\infty,S}$,
defined in~\eqref{eq:def_mathcal_Z^T}, solves SDER~\eqref{eq:limitEqJoint_Integral}
    with initial condition $(0,Y_S/b(X_S))$, and its first component is denoted by~$\mathcal{X}^{\infty,S}$.
    Recall also that the $x$-component of $\mathcal{Z}^S_t$, defined  in~\eqref{eq:Z^T-def},   equals $\mathcal{X}^S_t=(X_{S+ t b(X_S)^2 }-X_S)/b(X_S)$.
    Rearranging the terms in the difference $\tilde{X}_T-\tilde{X}_S$, we obtain 
    \begin{align*}
    \tilde{X}_T-\tilde{X}_S
    =&\Big( \frac{\sqrt{S}}{\sqrt{T}}-1\Big) \tilde{X}_S +\frac{c_1}{\sqrt{T}}(S^{\frac{1}{1+\beta}} -T^{\frac{1}{1+\beta}} )\\
    & +
     \frac{b(X_S)}{\sqrt{T}} (\cX^S_{\frac{T-S}{b(X_S)^2}}- \cX^{\infty,S}_{\frac{T-S}{b(X_S)^2}} )+   \frac{b(X_S)}{\sqrt{T}} \cX^{\infty,S}_{\frac{T-S}{b(X_S)^2}}.
    \end{align*}
    
    Recall that, as $T\to \infty$, we have $S/T\to1$, 
    $T-S=O(T^\frac{2\beta}{1+\beta} \log(T))$ (this follows from the properties of $S$ given by Lemma~\ref{le:sequenceRand}), that 
    $X_T\sim X_S~\sim c_1 T^{\frac{1}{1+\beta}}$  almost surely by the strong law in~\eqref{def:c1}. Therefore $b(X_S)T^{-\frac{\beta}{1+\beta}}\to a_\infty c_1^\beta$ as $T\to \infty$. We now show that the  summands in  the last display tend to $0$ in probability, along the sequence $T=T_n$, as $n\to \infty$. 
    \begin{enumerate}
    \item[(i)] Since we know $\tilde{X}_{S_n}$ converges in distribution, we have
    \[ \Big( \frac{\sqrt{S_n}}{\sqrt{T_n}}-1\Big) \tilde{X}_{S_n}\underset{n\to \infty}{\overset{(proba)}\longrightarrow} 0. \]
    
    \item[(ii)] Since $T-S=O(T^\frac{2\beta}{1+\beta} \log(T))$ as $T\to\infty$ and $\beta\in(-1,1)$,
    we deduce 
    \[
    T^{\frac{1}{1+\beta}}
    - S^{\frac{1}{1+\beta} }= \frac{1}{1+\beta}T^{\frac{1}{1+\beta}-1}(T-S+o(T-S))  =  O(T^{\frac{\beta}{1+\beta}}  \log(T) )\quad\text{as $T\to\infty$,} 
    \]
    hence, as $\beta/(1+\beta)<1/2$, we obtain
    \begin{equation}
    \label{eq:temp:estim}
    \frac{c_1}{\sqrt{T_n}}\bigl(S_n^{\frac{1}{1+\beta}} -T_n^{\frac{1}{1+\beta}} \bigr) \underset{n\to \infty}{\longrightarrow} 0.
    \end{equation}
    
    \item[(iii)] Since, by~\eqref{eq:inclu}, we have $T\leq T+b(X_T)^2  \mathpzc{s}_2(T)\leq S+4b(X_S)^2 \mathpzc{s}_2(S)\leq S+b(X_S)^2 \mathpzc{s}_1(S)$ and hence
    \[\frac{T-S}{b(X_S)^2 }\in [0,\mathpzc{s}_1(S)].\]
  It follows that (recall $\mathpzc{s}_1$ was chosen so that~\eqref{eq:long} in Proposition~\ref{prop:long} holds)
    \[
    \Bigl| \cX^S_{\frac{T-S}{b(X_S)^2}}- \cX^{\infty,S}_{\frac{T-S}{b(X_S)^2}} \Bigr|
    \leq 
    \|\cX^S_{\frac{T-S}{b(X_S)^2}}- \cX^{\infty,S}_{\frac{T-S}{b(X_S)^2}} \|_{[0, \mathpzc{s}_1(S)]}\underset{T\to \infty}{\overset{(proba)}\longrightarrow}0.
    \]
    Since $b(X_{S_n})/\sqrt{T_n}\to0$ (hence is bounded), we deduce 
    \[\frac{b(X_{S_n})}{\sqrt{T_n}}
    \Bigl| \cX^{S_n}_{\frac{T_n-S_n}{b(X_{S_n})^2}}- \cX^{\infty,S}_{\frac{T_n-S_n}{b(X_{S_n})^2}} \Bigr|
    \underset{n\to \infty}{\overset{(proba)}\longrightarrow}0.\]
    \item[(iv)] 
    Write \[ \frac{b(X_S)}{\sqrt{T}} \cX^{\infty,S}_{\frac{T-S}{b(X_S)^2}}
    =
     \frac{T-S}{\sqrt{T} b(X_S)}\Big(\frac{1}{s} \cX^{\infty,S}_{s} \Big) \qquad \text{with} \qquad  s\coloneqq \frac{T-S}{b(X_S)^2}.\]
    Since  almost surely $b(X_S)^{-1}= O( T^{-\frac{\beta}{1+\beta}})$ as $T\to\infty$, we get
    \begin{equation}
     \label{eq:temp:estim2}
    \frac{T-S}{b(X_{S}) \sqrt{T}}= b(X_{S})^{-1} O\bigl(T^{\frac{2\beta}{1+\beta}-\frac{1}{2} }  \log T \bigr)
    =
    O \bigl( T^{\frac{\beta}{1+\beta}-\frac{1}{2}} \log T \bigr)\underset{T\to \infty}{\overset{(proba)}\longrightarrow} 0.
    \end{equation} 
    By the inclusion in~\eqref{eq:inclu}, for $T$ sufficiently large, it holds $T-S\geq 2b(X_S)^2\mathpzc{s}_2(S)$. Thus, $s\geq 2 \mathpzc{s}_2(S)\to\infty$ as $T\to \infty$.
    We now claim that $\frac{1}{s}\cX^{\infty,  S}_s$ converges in probability to a deterministic limit as $s\to \infty$. 
    This follows from \cite[Thm~2.2]{MMW} and  a localisation argument detailed in the final paragraph of this proof, below. Assuming this convergence, and using~\eqref{eq:temp:estim2}, we deduce 
    \[\frac{b(X_{S_n})}{\sqrt{T_n}} \cX^{\infty,S}_{\frac{T_n-S_n}{b(X_{S_n})^2}}\underset{n\to \infty}{\overset{(proba)}\longrightarrow} 0.\]
  \end{enumerate}  
  Points (i)--(iv), together with the first display of the proof, imply that $\tilde{X}_{T_n}-\tilde{X}_{S_n}$ converges in probability to $0$ as $n\to\infty$.
  
    It only remains to deduce the claimed convergence in probability claimed in point~(iv). Let $\mathcal{D}_0$ be a domain with smooth boundary lying inside $\mathbb{R}\times \mathbb{B}^d$ and satisfying $\mathcal{D}_0\cap(\mathbb{R}_+\times \mathbb{R}^d)=\mathbb{R}_+\times \mathbb{B}^d$ with intersection  $\mathcal{D}_0\cap(\mathbb{R}_-\times \mathbb{R}^d)$ being compact. We consider on $\mathcal{D}_0$ the SDER with smooth coefficients which extend those of $\mathcal{Z}^{\infty,S}$ on  $\mathbb{R}_+\times \mathbb{B}^d$. For $c>0$, let $\mathcal{Z}^{c}=(\mathcal{X}^{c},\mathcal{Y}^{c})$ be a strong solution, driven by the same Brownian motion as the one driving $\mathcal{Z}^{\infty,S}$, and with starting point $\mathcal{Z}^{c}_0=(c,Y_S/b(X_S))$. Let $\tau_c\in \mathbb{R}_+\cup\{\infty\}$ be the first time when $\mathcal{X}^{c}$ hits $0$. Then, for $t<\tau_c$, $\mathcal{Z}^{\infty,S}_t= \mathcal{Z}^c_t -c $. Thanks to the localisation, the process $\mathcal{Z}^c$
    is in the family of processes studied in \cite{MMW} (with the exponent $\beta=0$, as we are working here in an asymptotic cylinder). By \cite[Thm~2.2]{MMW}, there exists $\ell\in(0,\infty)$ which does not depend on $c$ and such that  almost surely, $\frac{1}{s} \mathcal{X}^{c}_s\longrightarrow \ell $.  
    Thus, almost surely on the event $\{\tau_c=\infty\}$, 
    $\frac{1}{s} \mathcal{X}^{\infty,\mu}_s$ converges to $\ell$ as $s\to \infty$. 
    By \cite[Prop.~5.3]{MMW}, the probability of $\{\tau_c=\infty\}$ goes to $1$ as $c\to \infty$. Hence, almost surely, $\frac{1}{s} \mathcal{X}^{\infty,S}_s$ converges to~$\ell$ as $s\to \infty$, as claimed.
\end{proof}

\begin{lemma}
    \label{le:claimtech2}
    Let $(T_n)_{n\in \mathbb{N}}$ be a deterministic sequence such that $T_n\underset{n\to \infty}\longrightarrow \infty$. Assume that 
    $ \tilde{X}_{S_n}$, defined in~\eqref{eq:X_tilde_Y_tilde} above, converges in law to a distribution $\rho$, as $n\to \infty$. Then 
    the couple 
    $(\tilde{X}_{S_n}, \mathcal{Y}^{\infty,S_n}_{(T_n-S_n)/b(X_{S_n})^2} )$ 
    converges in law to the product distribution $\rho\otimes \mu$.
\end{lemma}

\begin{proof}
Recall that 
$\mathcal{Y}^{\infty,S_n}$ is the $y$-component of 
    $\mathcal{Z}^{\infty,S_n}$
defined in~\eqref{eq:def_mathcal_Z^T}.
    Let $\rho_n$ and $\nu_n$ be the distributions of $\tilde{X}_{S_n}$ and $Z_{S_n}$, respectively. Let
    $\Gamma_n$ be the joint distribution of the random vector
    $(\tilde{X}_{S_n}, \mathcal{Y}^{\infty,S_n}_{(T_n-S_n)/b(X_{S_n})^2}  )$. Define the following deterministic functions: $p_n((x,y ))\coloneqq  (x-c_1 S_n^{\frac{1}{1+\beta}})/\sqrt{S_n}$,
    $s_n((x,y))\coloneqq (T_n-S_n)/b( x )^2  $
    and $y_n(x,y)\coloneqq y/b(x)$.
    Then we have: 
    $\tilde{X}_{S_n}=p_n(Z_{S_n})$, 
    the process $\mathcal{Y}^{\infty,S_n}$ starts at $y_n(Z_{S_n})$ and $(T_n-S_n)/b(X_{S_n})^2 =s_n(Z_{S_n})$.
    Thus,
    \[
    \rho_n=\int  \delta_{p_n(z)} \d \nu_n (z)
    \quad  \text{and} \quad
    \Gamma_n=\int ( \delta_{p_n(z)}\otimes (\delta_{y_n(z)} P^\infty_{s_n(z)}) ) \d \nu_n (z).\]
    
    The convergence we seek to establish is equivalent to  the convergence of $d_{\prok}(\Gamma_n,\rho\otimes \mu)\to0$ as $n\to\infty$, where  
    $d_{\prok}$ is the Prokhorov distance, metrising weak convergence. Recall the following elementary facts about the Prokhorov distance: for any measures $\mu_1,\mu_2,\mu_3$, it holds $d_\prok(\mu_1,\mu_2)\leq d_\TV(\mu_1,\mu_2)$
    and    $d_\prok(\mu_1\otimes \mu_3,\mu_2\otimes \mu_3)\leq d_\prok(\mu_1,\mu_2)  $.
    In particular, 
    \begin{align*}
    d_{\prok}(    \Gamma_n, \rho\otimes \mu)
   & \leq 
    d_{\prok}(    \Gamma_n, \rho_n\otimes \mu) 
    +
    d_{\prok}(    \rho_n\otimes \mu, \rho\otimes \mu) \\
   & \leq 
    d_{\TV}(    \Gamma_n, \rho_n\otimes \mu) 
    +
    d_{\prok}(    \rho_n, \rho).
    \end{align*}
    By assumption, $d_{\prok}(    \rho_n, \rho)\underset{n \to \infty}\longrightarrow 0$, so it suffices to show that 
    \begin{equation*} 
    d_{\TV}(    \Gamma_n, \rho_n\otimes \mu) \longrightarrow 0\quad\text{as $n\to\infty$.}
    \end{equation*}
    Let $\lambda<1$ and $C_0<\infty$ be given by Corollary~\ref{cor:mu-existence-convergence}. Fix $\epsilon>0$, and let $s$ be such that 
    $C_0 \lambda^s\leq \epsilon/2$. 
    For $n$ large enough, with large probability, $(T_n-S_n)/b(X_{S_n})^2 \geq  2 \mathpzc{s}_2(S_n)$, and $2 \mathpzc{s}_2(S_n)\to\infty$ as $n\to \infty$, so there exists 
    $n_0\in\N$  such that  $\mathbb{P}(  (T_n-S_n)/b(X_{S_n})^2\leq s )\leq \epsilon/2$ for all $n\geq n_0$. Then, for all $n\geq n_0$, it holds 
    \begin{align*}
    d_{\TV}(   & \Gamma_n,  \ \rho_n\otimes \mu)
    =
    d_{\TV} \Bigl(    \int (\delta_{p_n(z)}\otimes (\delta_{y_n(z)} P^\infty_{s_n(z)}) ) \d \nu_n (z),
    \int ( \delta_{p_n(z)}\otimes \mu )\d \nu_n (z)
    \Bigr)\\    
    &\leq
    \int d_{\TV}( \delta_{p_n(z)}\otimes (\delta_{y_n(z)} P^\infty_{s_n(z)})  ,   \delta_{p_n(z)}\otimes \mu )\d \nu_n (z)\\
    &=
    \int d_{\TV}(\delta_{y_n(z)} P^\infty_{s_n(z)},\mu )  \d \nu_n (z)\\
    &= \int \mathbbm{1}_{s_n(z) \leq s} d_{\TV}(\delta_{y_n(z)} P^\infty_{s_n(z)},\mu )  \d \nu_n (z) 
    + \int \mathbbm{1}_{s_n(z) > s} d_{\TV}(\delta_{y_n(z)} P^\infty_{s_n(z)},\mu )  \d \nu_n (z)\\
    &\leq \int \mathbbm{1}_{s_n(z) \leq s} \d \nu_n (z) 
    + \sup_{z:s_n(z) > s} d_{\TV}(\delta_{y_n(z)} P^\infty_{s_n(z)},\mu )\\
    &\leq 
    \mathbb{P}(  (T_n-S_n)/b(X_{S_n})^2\leq s )+   C_0 \lambda^s  \leq \epsilon/2+\epsilon/2= \epsilon,
    \end{align*}
    which concludes the proof of the lemma.
\end{proof}

\begin{proof}[Proof of Proposition~\ref{prop:joint}] 
    It is sufficient to show that the weak convergence of the pair $(\tilde X_T,\tilde Y_T)$, defined in~\eqref{eq:X_tilde_Y_tilde} above, holds along any arbitrary sequence $(T_n)_{n\in\N}$, satisfying $T_n\to \infty$ as $n\to\infty$. We fix such a sequence $(T_n)_{n\in\N}$. 
    
    
    By the assumption of convergence in Proposition~\ref{prop:joint}, it holds that $\tilde{X}_{S_n}$ converges in distribution to $\rho$ as $n\to \infty$ (recall that $S_n=S(T_n)\to\infty$). By Lemma~\ref{le:claimtech},  $\tilde{X}_{S_n}-\tilde{X}_{T_n}$ tends to $0$ in probability as $n\to \infty$. Thus, by Slutsky's theorem, 
    \[
    (\tilde{X}_{T_n},\tilde{Y}_{T_n} )\underset{n\to \infty}{\overset{(d)}\longrightarrow} \rho\otimes \mu
    \iff 
    (\tilde{X}_{S_n},\tilde{Y}_{T_n} )\underset{n\to \infty}{\overset{(d)}\longrightarrow} \rho\otimes \mu.
    \]
    We now prove that the right-hand side holds, concluding the proof of the proposition. 

    Let $\Gamma_n$ be the distribution of $ (\tilde{X}_{S_n},\tilde{Y}_{T_n} )$.
    By Corollary~\ref{coro:Ycv}, $\tilde{Y}_{T_n}=Y_{T_n}/b(X_{T_n})$ converges in distribution to $\mu$ as $n\to \infty$. As both marginal distributions of $\Gamma_n$ converge weakly, it follows that  the family $(\Gamma_n)_{n\in\N}$ is tight. 
    By Prokhorov's theorem, it suffices to show that for any convergent subsequence $(\Gamma_{n_k})_{k\in\N}$, 
    the limit law $\tilde{\Gamma}$, supported on $\R\times\BB^d$, equals  $\rho\otimes \mu$. We fix such a subsequence, which for simplicity we call $n$. To summarise, we are given a sequence $(T_n)_{n\in\N}$, such that $T_n\to\infty$ and  
    \[ 
    (\tilde{X}_{S_n},\tilde{Y}_{T_n} )\underset{n\to \infty}\longrightarrow \tilde{\Gamma} \qquad \text{and} \qquad 
    \tilde{X}_{S_n} \underset{n\to \infty}\longrightarrow \rho,
    \]
    in distribution.
    It then suffices to prove that $\tilde{\Gamma}=\rho\otimes \mu$. 
    
    By the inclusion~\eqref{eq:inclu}, for all $n$ sufficiently large, we have 
    \[ 
    T_n\leq S_n+4 b(X_{S_n})^2 \mathpzc{s}_2(S_n)\leq S_n+ b(X_{S_n})^2 \mathpzc{s}_1(S_n), 
    \quad \text{and thus} 
    \quad \frac{T_n-S_n}{b(X_{S_n})^2} \in[0,\mathpzc{s}_1(S_n)].\] 
    It follows from Proposition~\ref{prop:long} that
    \[
    \bigl| \mathcal{Y}^{\infty,S_n}_{(T_n-S_n)/b(X_{S_n})^2}- \mathcal{Y}^{S_n}_{(T_n-S_n)/b(X_{S_n})^2} \bigr|\overset{(proba)}{\underset{n\to \infty}\longrightarrow} 0.
    \] 
    Furthermore, by the definition in~\eqref{eq:Z^T-def} of $\mathcal{Y}^{S_n}$, we obtain
    \[
    \bigl| \mathcal{Y}^{S_n}_{(T_n-S_n)/b(X_{S_n})^2}- \tilde{Y}_{T_n} \bigr| = \Big| \frac{Y_{T_n}}{b(X_{S_n})} -\frac{Y_{T_n}}{b(X_{T_n})} \Big|\leq \Big| \frac{b(X_{T_n})}{b(X_{S_n})} -1 \Big| \overset{(proba)}{\underset{n\to \infty}\longrightarrow} 0.
    \]
    By triangle inequality,  we deduce $|\mathcal{Y}^{\infty,S_n}_{(T_n-S_n)/b(X_{S_n})^2}- \tilde{Y}_{T_n} |$ also converges to $0$ in probability.
    Using Slutsky's theorem again, we deduce that the triple
    $(\tilde{X}_{S_n},\tilde{Y}_{T_n}, \mathcal{Y}^{\infty,S_n}_{(T_n-S_n)/b(X_{S_n})^2}  )$ converges in distribution to $(N,Y^1,Y^2)$, with $Y^1=Y^2$ and $(N,Y^1)$ distributed as $\tilde{\Gamma}$. 
    By Lemma~\ref{le:claimtech2}, the distribution of $(N,Y^2)$ 
is $\rho\otimes \mu$, implying $\tilde{\Gamma}=\rho \otimes \mu$.
\end{proof}

\section{Ergodicity}
\label{sec:ergodicity}

Our goal in this section is to establish  asymptotic behaviour of certain additive functionals of the reflected process $Z=(X,Y)$ following SDER~\eqref{eq:SDER} in terms of the invariant measure $\mu$, supported on $\BB^d$ and characterised in Corollary~\ref{cor:mu-existence-convergence} above. In the proof of the  central limit theorem for $X$ in Section~\ref{sec:central} below, integrals such as the one in the following proposition appear naturally through It\^o's formula. 
\begin{proposition}
\label{prop:int}
  Let $P=P_0+P_1+P_2$ be a polynomial in $d$ variables of degree $2$, where $P_p$ is homogeneous of degree $p\in\{0,1,2\}$. Define $\tilde{P}\coloneqq P_0+a_\infty P_1+a_\infty^2 P_2$ (see Assumption~\eqref{hyp:D3} for the coefficient $a_\infty>0$), assume  $\beta>-\frac{1}{3}$ 
and recall the constant $c_1$ from~\eqref{def:c1}.  Then, 
  \[ T^{-1-\frac{2\beta}{1+\beta}} \int_0^T X_t^{2\beta} P(\frac{Y_t}{X_t^\beta})\d t \underset{T\to \infty}{\overset{(proba)}\longrightarrow} \frac{1+\beta}{1+3\beta} c_1^{2\beta} \int_{\BB^d} \tilde{P} \d \mu,  \]
  where $\mu$ is the unique invariant measure of SDER~\eqref{eq:limitEq} (cf.~Corollary~\ref{cor:mu-existence-convergence} above).
\end{proposition}

The proof of Proposition~\ref{prop:int} subdivides the time interval $[0,T]$  into pieces of adequate length, so that the process $X$ has little variation over any of these subintervals. To this end, we use a function $\mathpzc{s}_3:\RP\to\RP$  for which Proposition~\ref{prop:seq} holds, and which we can assume is bounded below by $1$. We then take a function $\mathpzc{s}_4\leq \mathpzc{s}_3$, also bounded below by $1$, satisfying $\mathpzc{s}_4(T)\to \infty$ and $\mathpzc{s}_4(T)=O(\log(T))$, as $T\to \infty$, and, in particular, $T\sim_{T\to \infty} T+T^\frac{2\beta}{1+\beta} \mathpzc{s}_4(T)$. The key step in the proof, given in Lemma~\ref{le:integral}, consists of estimating the integral in Proposition~\ref{prop:int} when the interval of integration $[0,T]$ is replaced by
$[T, T+C^2T^\frac{2\beta}{1+\beta} \mathpzc{s}_4(T)]$ for some $C>0$.

\begin{lemma}
\label{le:integral}
  Let $q\geq 0$ be a non-negative real number and let $P$ be a homogeneous polynomial in $d$ variables of degree $p$, i.e.~$P(\lambda y)=\lambda^p P(y)$ for all $\lambda\geq0$ and $y\in \mathbb{R}^d$.
  Let $C\coloneqq a_\infty c_1^{\beta}$, where $a_\infty$ and $c_1$ are defined in Assumption~\eqref{hyp:D3} and~\eqref{def:c1}, respectively. 
  As $T\to \infty$, in probability,
  \begin{equation}
  \label{eq:le:conv}
  \frac{1}{C^2 T^{\frac{2\beta+p\beta+q}{1+\beta}}  \mathpzc{s}_4(T)} \int_T^{T+C^2 T^{\frac{2\beta}{1+\beta}} \mathpzc{s}_4(T)} X_s^q P(Y_s) \d s\longrightarrow  a_\infty^p c_1^{q+p\beta} \int_{\BB^d} P \d \mu =:c_4.
  \end{equation}
\end{lemma}  

\begin{proof} 
 As $T\to\infty$, we have
     $X_T\sim c_1 T^{\frac{1}{1+\beta}}$ (by~\eqref{def:c1}) and  $T+C^2 T^{\frac{2\beta}{1+\beta}} \mathpzc{s}_4(T)\sim T$.
     Thus,
    \begin{equation}
    \label{eq:temp:boundX}
    \sup_{s\in [T, T+ C^2  T^{\frac{2\beta}{1+\beta}} \mathpzc{s}_4(T) ]} T^{-\frac{q }{1+\beta}}  |X_s^q-X_T^q|  \underset{T\to\infty}{\longrightarrow} 0.
    \end{equation}
    Furthermore, by~\eqref{hyp:D3},~\eqref{def:c1} and~\eqref{eq:temp:boundX}, we have
    \begin{equation}\label{eq:temp:boundbX}
    \sup_{s\in [T, T+C^2 T^{\frac{2\beta}{1+\beta}}  \mathpzc{s}_4(T) ]} |T^{-\frac{\beta }{1+\beta}} b(X_s)-C|
    \underset{T\to\infty}{\longrightarrow} 0.\end{equation}
    Limits~\eqref{eq:temp:boundX} and~\eqref{eq:temp:boundbX} hold if we replace
    $[T, T+C^2 T^{\frac{2\beta}{1+\beta}}  \mathpzc{s}_4(T) ]$ with  $[T,T+ b(X_T)^2 \mathpzc{s}_4(T) ]$.

    To prove the convergence in~\eqref{eq:le:conv}, we would like to be able to replace the $X_s$ with $X_T$, and we first show that the error so committed goes to $0$ as $T\to \infty$. 
    Since $P$ is homogeneous of degree $p$ and $|Y_s|_d/b(X_s)$ is bounded by $1$, we have
    \begin{align*}
    &\frac{1}{C^2T^{\frac{2\beta + p\beta+q}{1+\beta} }   \mathpzc{s}_4(T)   } \int_T^{T  + C^2 T^{\frac{2\beta }{1+\beta} }   \mathpzc{s}_4(T)  } |X_s^q -X_T^q||P(Y_s)|\d s\\
    &\leq \frac{ 1}{T^{\frac{ p\beta+q}{1+\beta} }     } \sup_{t\in [T,T  +  C^2T^{\frac{2\beta}{1+\beta} } \mathpzc{s}_4(T)  ]}   |X_t^q -X_T^q| P^* b(X_t)^p \\
    &=  P^* \sup_{t\in [T,T  + C^2 T^{\frac{2\beta}{1+\beta} } \mathpzc{s}_4(T)  ]} T^{- \frac{q}{1+\beta}  } |X_t^q -X_T^q|   \sup_{t\in [T,T  +  C^2T^{\frac{2\beta}{1+\beta} } \mathpzc{s}_4(T)  ]} T^{\frac{-p\beta}{1+\beta}} b(X_t)^p\\
    &\underset{T\to \infty}\longrightarrow 0,
    \end{align*}
    where  $P^*=\sup_{y\in \BB^d}|P(y)|$ and the last convergence follows from \eqref{eq:temp:boundX} and \eqref{eq:temp:boundbX}. It thus suffices to prove \eqref{eq:le:conv} with the factor $X_s$ replaced with $X_T$. Since $X_T\sim c_1 T^{\frac{\beta}{1+\beta}}$, it is then easily seen that the case $q=0$ implies the general case. In the remainder of the proof we assume $q=0$.
    
    Now, we wish to replace the integral bound $T+C^2 T^{\frac{2\beta}{1+\beta}}\mathpzc{s}_4(T)$ with $T+b(X_T)^2  \mathpzc{s}_4(T)$ in the left-hand side of \eqref{eq:le:conv}. Let us estimate the error:
    \begin{align*}
    &\frac{1}{C^2 T^{\frac{(2+p)\beta}{1+\beta} }   \mathpzc{s}_4(T)   } \Big|\int_{T+b(X_T)^2 \mathpzc{s}_4(T)}^{T  +C^2  T^{\frac{2\beta }{1+\beta} }   \mathpzc{s}_4(T)  } P(Y_s) \d s\Big|\\
    &\leq \frac{P^*}{C^2 T^{\frac{(2+p)\beta}{1+\beta} }      }
    \Big|b(X_T)^2 - C^2T^{\frac{2\beta }{1+\beta} }\Big|
    \sup_{t\in\Big[T, T+ \mathpzc{s}_4(T)   \max\big( C^2T^{\frac{2\beta}{1+\beta} },  b(X_T)^2 \big)\Big]} b(X_t)^p  \\
    &\underset{T\to \infty}\sim P^* C^{p-2} \Big| b(X_T)^2T^{-\frac{2\beta}{1+\beta}}-C^2 \Big|
    \underset{T\to \infty}\longrightarrow 0.
    \end{align*}
    Thus, it suffices to show that in probability,
    \[
    \frac{1}{C^2 T^{\frac{2\beta+p\beta}{1+\beta}}  \mathpzc{s}_4(T)} \int_T^{T+b(X_T)^2   \mathpzc{s}_4(T)} P(Y_s) \d s \underset{T\to \infty}\longrightarrow  C^p \int_{\BB^d} P \d \mu.\]
Recall that by definition~\eqref{eq:Z^T-def} of $\cZ^T=(\cX^T,\cY^T)$, we have  $Y_{T+b(X_T)^2 t}=b(X_T) \cY^{T}_t$.  Since $P$ is homogeneous of degree $p$, we obtain
    \begin{align*}\frac{1}{C^2T^{\frac{(2+p)\beta}{1+\beta} }   \mathpzc{s}_4(T)   }\int_T^{T+ b(X_T)^2\mathpzc{s}_4(T)}  P(Y_s) \d s
    &=\frac{b(X_T)^{2} }{C^2T^{\frac{(2+p)\beta}{1+\beta} }   \mathpzc{s}_4(T)   }\int_0^{  \mathpzc{s}_4(T)}     P(Y_{T+ b(X_T)^2 \mathpzc{s} }) \d \mathpzc{s}\\
    &=\frac{b(X_T)^{p+2} }{C^2 T^{\frac{(2+p)\beta}{1+\beta} }   \mathpzc{s}_4(T)   }\int_0^{  \mathpzc{s}_4(T)}      P(\cY^{T}_\mathpzc{s} ) \d \mathpzc{s}\\
    &\underset{T\to \infty}{\sim} \frac{C^{p}}{\mathpzc{s}_4(T)} \int_0^{ \mathpzc{s}_4(T)}      P(\cY^{T}_\mathpzc{s} ) \d \mathpzc{s}.
    \end{align*}
    
    To conclude, it thus suffices to prove
    \begin{equation}
    \label{eq:conv_in_prob_nearly_final}
    \frac{1}{\mathpzc{s}_4(T)} \int_0^{ \mathpzc{s}_4(T)}      P(\cY^{T}_\mathpzc{s} ) \d \mathpzc{s} \underset{T\to \infty}{\overset{(proba)}\longrightarrow} \int_{\BB^d} P \d \mu.
    \end{equation}
    Since $\mathpzc{s}_4(T)\to\infty$ as $T\to\infty$, by Birkhoff ergodic theorem (see e.g.~\cite[Thm~2.8]{Eberle}) 
    applied to the stationary solution $\cY^{\infty,\mu}$ of SDER~\eqref{eq:limitEq}, started from its unique invariant measure $\mu$ (cf.~Corollary~\ref{cor:mu-existence-convergence}), we obtain 
    \begin{equation}
        \label{eq:Birkhoff_consequence}
    \frac{1}{\mathpzc{s}_4(T) }   \int_{0}^{\mathpzc{s}_4(T) } P(\cY^{\infty, \mu}_\mathpzc{s}) \d \mathpzc{s} \underset{T\to \infty}\longrightarrow \int_{\BB^d} P \d \mu\qquad\text{almost surely,}
    \end{equation}
    and hence in probability. Thus,~\eqref{eq:conv_in_prob_nearly_final} follows if there exist couplings $(\mathbb{P}^T)_{T> 0}$ of $\cY^{T}$ and $\cY^{\infty, \mu}$ such that for any $\epsilon>0$ we have 
    \[ \mathbb{P}^T\Big( \frac{1}{\mathpzc{s}_4(T)} \int_0^{ \mathpzc{s}_4(T)}     |  P(\cY^{T}_\mathpzc{s} )-P(\cY^{\infty, \mu}_\mathpzc{s})|  \d \mathpzc{s}  \geq \epsilon \Big)\leq \epsilon\quad\text{for all sufficiently large $T$.} \] 
    The couplings $\mathbb{P}^T$ of the  pairs of processes $(\cZ^T,\cZ^{\infty,(0,\tilde Y^T,\tilde \cW^T})$, $T>0$, in Proposition~\ref{prop:seq} satisfy this requirement: recall that $\cZ^{\infty,(0,\tilde Y^T),\tilde \cW^T}$ has the same law as $\cZ^{\infty,\mu}$
    and that for any $\epsilon>0$, there exists  a deterministic time $T_0>0$ such that for all $T\geq T_0$ we have
    \[ \mathbb{P}^T( \| \cZ^T- \cZ^{\infty,\mu} \|_{[0,\mathpzc{s}_4(T)]} \geq \epsilon )\leq \epsilon.\]
    Here and in the remainder of the proof we identify $\cZ^{\infty,(0,\tilde Y^T),\tilde \cW^T}$ and $\cZ^{\infty,\mu}$ and recall that its $y$-component $\cY^{\infty,\mu}$ is stationary,  following SDER~\eqref{eq:limitEq} on $\BB^d$ with invariant measure $\mu$.
On the event $\| \cZ^T- \cZ^{\infty,\mu} \|_{[0,\mathpzc{s}_4(T)]}\leq \epsilon$, 
we have $|\cY^T_t|_d\leq |\cY^T_t- \cY^{\infty,\mu}_t|_d +1\leq\epsilon+1$
for all $t\in[0,\mathpzc{s}_4(T)]$.
Thus, on this event, it also holds $|P(\cY^T_t)- P(\cY^{\infty,\mu}_t)|\leq \epsilon K_\epsilon$  for all $t\in [0, s_4(T)]$, where  $K_\epsilon>0$ is a Lipschitz constant of the polynomial $P$ restricted to the ball around the origin in $\R^d$ with radius $(1+\epsilon)$.
Thus, for an arbitrary $\epsilon\in(0,1)$, there exists $T_0$ such that for all $T\geq T_0$ with probability greater than $1-\epsilon$,
    \[  \frac{1}{\mathpzc{s}_4(T) }   \int_{0}^{\mathpzc{s}_4(T) } |P(\cY^{\infty, \mu}_t) -  P(\cY^T_t) |\d t
    \leq \epsilon K_1.
    \]
    Thus, in probability,
    \begin{equation*}
    \frac{1}{\mathpzc{s}_4(T) }   \int_{0}^{\mathpzc{s}_4(T) } |P(\cY^{\infty, \mu}_t) -  P(\cY^T_t) | \d t \underset{T\to \infty}\longrightarrow 0,
    \end{equation*}
    which, together with~\eqref{eq:Birkhoff_consequence}, implies~\eqref{eq:conv_in_prob_nearly_final} concluding the proof of the lemma.
\end{proof}

\begin{remark}
    There is an alternative approach to Lemma~\ref{le:integral}: instead of  coupling $\mathcal{Y}^T$ with the stationary process and $\mathcal{Y}^{\infty,\mu}$ using Proposition~\ref{prop:seq}, we could  couple it with  the process $\mathcal{Y}^{\infty,T}$ using Proposition~\ref{prop:long}. This is appealing as the proof of Proposition~\ref{prop:seq} is much more involved than that of Proposition~\ref{prop:long}. However, the latter approach would require a stronger form of the ergodic theorem, namely one that is uniform in the starting distribution: 
    \[ 
    \forall \epsilon>0, \ \sup_{\nu} \mathbb{P} \Big( \Big|   \frac{1}{\mathpzc{s}}\int_0^\mathpzc{s} P(\mathcal{Y}^{\infty,\nu}_u)\d u -\int_{\mathbb{B}^d} P \d \mu \Big|\geq \epsilon \Big) \underset{s\to \infty}\longrightarrow0.
    \]
    This could be established via the usual Birkhoff ergodic theorem (i.e.~using \eqref{eq:Birkhoff_consequence}) and  the coupling
 $(\mathcal{Y}^{\infty,T},\mathcal{Y}^{\infty,\mu})$ provided by Proposition~\ref{prop:coupling} (with $\nu$ the law of $Y_T/b(X_T)$), so that with large probability $\mathcal{Y}^{\infty,T}_u=\mathcal{Y}^{\infty,\mu}_u$ for all $u\geq t$, where $t=t(\mathpzc{s})$ is chosen such that  $1\ll t \ll \mathpzc{s}$ (i.e.~$t(\mathpzc{s})\underset{\mathpzc{s}\to \infty}\longrightarrow\infty$ and $t(\mathpzc{s})/\mathpzc{s}\underset{\mathpzc{s}\to \infty}\longrightarrow 0$). 
 We stress that Proposition~\ref{prop:seq} is nevertheless essential in the proof of our main result, as it is used to imply the convergence of the second component in the limit in Theorem~\ref{th:main1} (see also Corollary~\ref{coro:Ycv}).
\end{remark}

Recall $C= a_\infty c_1^{\beta}$ from Lemma~\ref{le:integral} and define recursively an increasing sequence $(T_n)_{n\in\N}$, 
\begin{equation}
\label{eq:T_n}
T_1\coloneqq 1\quad\text{and}\quad T_{n+1}\coloneqq T_n+C^2T_n^\frac{2\beta}{1+\beta} \mathpzc{s}_4(T_n)\quad \text{for $n\in\N$.}
\end{equation} 
Note that for $T=T_n$ in Lemma~\ref{le:integral}, the upper bound in the integral is exactly $T_{n+1}$. As $\mathpzc{s}_4$ is bounded below by $1$, $T_n\to \infty$ as $n\to \infty$ for all $\beta\in (-1,1)$ (or even $\beta\in\R$). For $\beta>0$ this is clear. For $\beta<0$,  if $T_n\leq D$ for all $n\in\N$ and some constant $D\in(0,\infty)$,  then the increments satisfy 
$T_{n+1}-T_n\geq C^2D^{\frac{2\beta}{1+\beta}}$ for all $n$, making $T_n$ in fact unbounded. 

The  elementary result in Lemma~\ref{le:estimTechnic} is established in Appendix~\ref{Appendix:52} below. The proof relies on the fact that $\mathpzc{s}_4(T)$ grows more slowly than any positive power of $T$, as $T\to \infty$.

\begin{lemma}
\label{le:estimTechnic}
Assume $\beta>-\frac{1}{3}$ and let $\alpha\coloneqq \frac{2\beta}{1+\beta}$ (then $1+\alpha>0$). For $(T_n)_{n\in\N}$ in~\eqref{eq:T_n}, as $n\to \infty$, we have
\[ 
\sum_{k=1}^n (T_{k+1}-T_k) |T^\alpha_k-T^\alpha_{k+1}|=o(T_n^{1+\alpha})\quad
\text{and}\quad 
\sum_{k=1}^n (T_{k+1}-T_k) T^\alpha_k =\frac{T_n^{1+\alpha}}{1+\alpha} (1+o(1)).
\]
\end{lemma}

\begin{proof}[Proof of Proposition~\ref{prop:int}]
By linearity we may assume $P=P_p$ for $p\in \{0,1,2\}$. We let again $\alpha=\frac{2\beta}{1+\beta}$. By homogeneity, the integral we are trying to estimate can be expressed as
\[ 
T^{-1-\alpha}\int_0^T X_t^{2\beta}P(\frac{Y_t}{X_t^\beta})\d t =
T^{-1-\alpha}\int_0^T X_t^{(2-p)\beta}P(Y_t)\d t.
\]
    For $R>0$, let $E_R$ be the event
    $\sup_{T\geq 1} T^{-\frac{1}{1+\beta}} X_T \leq R$.
Recall $(T_n)_{n\in\N}$ in~\eqref{eq:T_n} and let 
    \[I_n\coloneqq \frac{1}{T_{n+1}-T_n} T_n^{-\alpha } \int_{T_n}^{T_{n+1}} X_t^{(2-p)\beta} P(Y_t) \d t
    = \frac{T_n^{-2\alpha}}{C^2  \mathpzc{s}_4(T_n)} \int_{T_n}^{T_{n+1}} X_t^{(2-p)\beta} P(Y_t) \d t
    .\]
    
    On the event $E_R$, the sequence $ ( T_n^{-\alpha} \sup_{t\in [T_n,T_{n+1}] } X_t^{(2-p)\beta} P(Y_t) )_{n\in\N}$ is bounded, say by $C'$. It follows directly that the sequence $(I_n)_{n\in\N}$ is also bounded by $C'$ on $E_R$. Furthermore, as noted above, the sequence $(T_n)_{n\in\N}$ is deterministic and diverges to $\infty$, which allows us to apply the limit in~\eqref{eq:le:conv} of Lemma~\ref{le:integral}  (with $q={(2-p)\beta}$), ensuring that
    $I_n$ converges in probability to $c_4$. Thus, $I_n\mathbbm{1}_{E_R}$ converges in $L^2$ to $c_4 \mathbbm{1}_{E_R}$.
    Let $\epsilon>0$ and $n_0\in\N$ such that for all $n\geq n_0$, $\|(I_n-c_4)\mathbbm{1}_{E_R}\|_{L^2}\leq \epsilon$.
    Using the identities
    \[
    \int_1^{T_n} X_t^{(2-p)\beta} P(Y_t) \d t=\sum_{k=1}^{n-1}
    (T_{k+1}-T_k) T_k^\alpha I_k,\quad 
    \frac{1}{1+\alpha} T_n^{1+\alpha}=
    \frac{1}{1+\alpha} + \sum_{k=1}^{n-1}\int_{T_k}^{T_{k+1} } t^\alpha \d t ,
    \]
    we get
    \begin{align*}
    &\Big\| \mathbbm{1}_{E_R}\Big(
    \int_1^{T_n} X_t^{(2-p)\beta} P(Y_t) \d t - c_4 \frac{1}{1+\alpha} T_n^{1+\alpha} \Big) \Big\|_{L^2}\\
    &\leq \frac{|c_4|}{1+\alpha} +
    \sum_{k=1}^{n-1}  \big\| \mathbbm{1}_{E_R} \big( (T_{k+1}-T_k) T_k^\alpha I_k  - c_4 \int_{T_k}^{T_{k+1} } t^\alpha \d t \big) \big\|_{L^2} \\
    &\leq \frac{|c_4|(1+\beta)}{1+q+(1+p)\beta} +
    \sum_{k=1}^{n_0-1}  \big\| \mathbbm{1}_{E_R} \big( (T_{k+1}-T_k) T_k^\alpha I_k  - c_4 \int_{T_k}^{T_{k+1} } t^\alpha \d t \big) \big\|_{L^2} \\
    &+
    \sum_{k=n_0}^{n-1} \Big( \big\| \mathbbm{1}_{E_R}  (T_{k+1}-T_k) T_k^\alpha (I_k  - c_4)\big\|_{L^2}+
    |c_4| | 
    (T_{k+1}-T_k) T_k^\alpha-
    \int_{T_k}^{T_{k+1} } t^\alpha \d t \big) |\Big) \\
    &\leq
    O(1) + \sum_{k=n_0}^{n-1} \Big( (T_{k+1}-T_k) T_k^\alpha \epsilon
    +
    |c_4|
    |(T_{k+1}-T_k) T_k^\alpha - \int_{T_k}^{T_{k+1} } t^\alpha\d t| \Big)\\
    &\leq O(1)+\epsilon \sum_{k=n_0}^{n-1} (T_{k+1}-T_k) T_k^\alpha +
        |c_4| \sum_{k=n_0}^{n-1}   (T_{k+1}-T_k) |T_{k+1}^\alpha - T_k^\alpha |\\
        &\leq \frac{\epsilon}{1+\alpha}T_n^{1+\alpha}+o(T_n^{1+\alpha}),
    \end{align*}
    where the last inequality follows from Lemma~\ref{le:estimTechnic}.
    Here the $O(1)$ depends on $\epsilon$ but not on $n$. 
We deduce
\[ 
\Big\| \mathbbm{1}_{E_R}\Big(
    \int_1^{T_n} X_t^{(2-p)\beta} P(Y_t) \d t - c_4 \frac{1}{1+\alpha} T_n^{1+\alpha} \Big) \Big\|_{L^2}\leq\frac{\epsilon}{1+\alpha} T_n^{1+\alpha}+ o(T_n^{1+\alpha}).
\]
    Since $\epsilon$ is arbitrary, we deduce
    \[
    \mathbbm{1}_{E_R}  \Big(T_n^{-1-\alpha}
    \int_1^{T_n} X_t^{(2-p)\beta} P(Y_t) \d t -  \frac{c_4}{1+\alpha}   \Big) \overset{L^2}{\underset{n\to \infty}\longrightarrow} 0.
    \]
    
    Since $T_n\sim T_{n+1}$ and  $\sup \{ t^{-\alpha} X_t^{(2-p)\beta} P(Y_t)  : t\geq 1\}<\infty$ on the event $E_R$, we easily deduce that the convergence extends to $T
    \in\R_+$. Moreover, we can also make the integral start from $0$ rather than $1$: 
    \[\mathbbm{1}_{E_R} T^{-1-\alpha}
    \int_0^{T} X_t^{(2-p)\beta} P(Y_t) \d t \overset{L^2}{\underset{T\to \infty}\longrightarrow} \mathbbm{1}_{E_R} c_4 \frac{1}{1+\alpha} .\]
    Since $R$ is arbitrary and by the strong law in~\eqref{def:c1} we have $\mathbb{P}(\bigcup_{R>0} E_R)=1$, the claimed convergence in probability holds (note $\frac{1}{1+\alpha}=\frac{1+\beta}{1+3 \beta}$).
\end{proof}

\section{Central limit theorem for the horizontal process}
\label{sec:central}
We assume that $\beta>-\frac{1}{3}$ in this entire section. We emphasize again that \emph{this is not for technical reasons}. 
We now introduce the functions which will help us study the asymptotic behaviour of $X$. 
First, we define $g:\cD\to \RP$ and $B:\RP\to \RP$ by
\[g(x,y)\coloneqq x+  \frac{s_0}{2c_0}\frac{|y|_d^2}{b(x)}, \qquad B(x)\coloneqq \int_0^x b_0(x')\d x',\]
where $b_0$ is equal to $b$ on $[1,\infty)$, supported on $[\frac{1}{2},\infty)$, and smooth on $[0,\infty)$.

In order to study the second order behaviour of $X$, we analyse the process
$\xi=B(g(Z))$.
It\^o's formula implies the following semimartingale decomposition for  $\xi$: 
\begin{equation}
\label{eq:itoXi}
\xi_t=\xi_0+\int_0^t \mu(Z_s)\d s+ \int_0^t \Lambda (Z_s) \d L_s+M_t,
\end{equation}
where the quadratic variation of the local martingale $M$ is given by $[M]_t= \int_0^t f(Z_s) \d s$. 
The exact expressions for the functions $\mu, f:\mathcal{D}\to \mathbb{R}$ and $\Lambda: \partial\mathcal{D}\to \mathbb{R}$ are given in Appendix \ref{Appendix:62} below, where we also give the proof of the following lemma, which contains all the information we require about these functions.

\begin{lemma}
    \label{le:deterministicEstimations}
    The functions $g, \mu, f:\mathcal{D}\to \mathbb{R}$ and $\Lambda: \partial\mathcal{D}\to \mathbb{R}$ are continuous and the following asymptotic equalities hold:
    \begin{align}
    g(x,y)&\underset{\substack{x\to \infty\\ (x,y)\in \cD} }=
    x +O(x^{\beta}), \label{eq:g} \\
    \mu(x,y)&\underset{\substack{x\to \infty\\ (x,y)\in \cD} }=
    \frac{s_0 \os}{2 c_0} +o(x^{\frac{\beta-1}{2}}), \label{eq:mu} \\
    \Lambda(x,y)&\underset{\substack{x\to \infty\\ (x,y)\in \partial \cD} }=o( x^\frac{3\beta-1}{2}),
    \label{eq:lambda}\\
        f(x,y)&\underset{\substack{x\to \infty\\ (x,y)\in \cD} }= a_\infty^2 x^{2 \beta } Q(\frac{y}{x^\beta} )
    +o(x^{2\beta}), \label{eq:covariance}
    \end{align}
where
$Q: \mathbb{R}^d \to \mathbb{R}$  is the quadratic polynomial given by
$$Q(\mathpzc{y})\coloneqq
 \Sigma^{\infty}_{0,0}+ \frac{2 s_0}{c_0 a_\infty} \sum_{i=1}^d \Sigma^{\infty}_{0,i} \mathpzc{y}_i + \frac{s_0^2}{c_0^2 a^2_\infty}  \sum_{i,j=1}^d \Sigma^{\infty}_{i,j} \mathpzc{y}_i \mathpzc{y}_j.$$
\end{lemma}
Recall from the discussion preceding Assumption~\eqref{hyp:C} that the asymptotic statements in Lemma~\ref{le:deterministicEstimations} mean 
\[ 
\limsup_{x\to \infty} \sup_{y: (x,y)\in \mathcal{D}} \frac{|g(x,y)-x|}{x^\beta} <\infty, 
\quad
\limsup_{x\to \infty} \sup_{y: (x,y)\in \mathcal{D}} \frac{|\mu(x,y)- \frac{s_0 \bar{\sigma}^2}{2 c_0} |}{x^{\frac{\beta-1}{2} }} =0, \text{ \emph{etc.}}
\]

A crucial point of our strategy is that the function $g$ has been chosen so that its gradient is almost orthogonal to the reflection vector field, making the inner product $\langle \phi(z), \nabla g( z)\rangle$ nearly zero (when the process is far away from the origin), and consequently the function $\Lambda$ very small. The result of this tuning via the function $g$ is that the contribution to $\xi$ from the local time term in $\eqref{eq:itoXi}$ is itself exceptionally small, when $t$ is large, which allows to bound this term in a rather imprecise way and still obtain a good estimate for $\xi_t$. 

We will now estimate separately the martingale $M$, the smooth drift $\int \mu(Z_s)\d s$, and the local time contribution $\int \Lambda(Z_s) \d L_s$, relying on our estimates of the functions $f$, $\mu$ and $\Lambda$ in Lemma~\ref{le:deterministicEstimations} and the following elementary result, whose proof is omitted for brevity.

\begin{lemma}
    \label{le:deterministicEstimationsIntegral}
    Let $\beta\in(-\frac{1}{3},1)$. Then, $\frac{\beta-1}{2(1+\beta)}>-1$ and $\frac{2\beta}{1+\beta}>-1$. 
    For any $t_0\geq 0$ and any continuous function $h:[t_0,\infty) \to \R$, the following implications hold:
    \begin{align}
   h(t) \underset{t\to \infty}= o( t^{\frac{\beta-1}{2(1+\beta)} }) &\implies \int_{t_0}^t h(s)\d s \underset{t\to \infty}=o( t^{\frac{1}{2}+\frac{\beta}{1+\beta}}),
    \label{eq:deterministicEstimationsIntegral1}\\
  h(t) \underset{t\to \infty}= o( t^{\frac{2\beta}{1+\beta} }) &\implies \int_{t_0}^t h(s)\d s \underset{t\to \infty}=o( t^{1+\frac{2\beta}{1+\beta}}).
    \label{eq:deterministicEstimationsIntegral2}
    \end{align}
\end{lemma}

\subsection{Estimation of the martingale part}
Recall that $[M]_t= \int_0^t f(Z_s) \d s$.
\begin{lemma}
\label{le:variance1}
    Let $\beta>-1/3$. As $T\to \infty$, 
    $T^{-1-\frac{2\beta}{1+\beta} } [M]_T$ converges in probability to the constant
    \[
    S^2\coloneqq  \frac{1+\beta}{1+3\beta}a_\infty^2 c_1^{2\beta}\Big( \Sigma^\infty_{0,0}+ \frac{2 s_0}{c_0} \sum_{j=1}^d  \Sigma_{0,j}^{\infty} \int_{\BB^d} y_j
    \d \mu_y+\frac{s_0^2}{c_0^2} \sum_{j,k=1}^d \Sigma^\infty_{j,k} \int_{\BB^d} y_j y_k
    \d \mu_y\Big).
    \]
\end{lemma}
\begin{proof}
Since almost surely, $X_s\underset{s\to \infty}\longrightarrow\infty$ and $s^{-\frac{1}{1+\beta}} X_s$ is bounded away from $0$, \eqref{eq:covariance} implies that, almost surely, 
\[ f(Z_s)\underset{s\to \infty}{=}a_\infty^2 X_s^{2\beta}Q(\frac{Y_s}{X_s^\beta} ) +o( s^\frac{2\beta}{1+\beta}). \]
By applying \eqref{eq:deterministicEstimationsIntegral2} to the continuous function $s\mapsto  f(Z_s)-a_\infty^2 X_s^{2\beta}Q(\frac{Y_s}{X_s^\beta} )  $, we deduce that almost surely,
\begin{equation}
\label{eq:quad_var_conv}
T^{-1-\frac{2\beta}{1+\beta}} \big(
[M]_T - a_\infty^2 \int_0^T X_s^{2\beta} Q(\frac{Y_s}{X_s^\beta})   \d s\big) \underset{T\to \infty}\longrightarrow 0.
\end{equation}
By applying Proposition~\ref{prop:int} with the quadratic polynomial $Q$, we deduce that for $\beta>-\frac{1}{3}$, in probability, $T^{-1-\frac{2\beta}{1+\beta} } \int_0^T X_t^{2\beta} Q(\frac{Y_t}{X_t} )\d t\to S^2$ as $T\to\infty$.
By~\eqref{eq:quad_var_conv}, it follows that $ T^{-1-\frac{2\beta}{1+\beta} } [M]_T$ converges in probability to $S^2$, provided $\beta>-\frac{1}{3}$.
\end{proof}
\begin{lemma}
\label{le:variance2}
    Let $N$ be a continuous local martingale, and assume that for some positive constants $\gamma$ and $V$, as $t\to\infty$, $t^{-2\gamma} [N]_t$ converges in probability to $V$. Then, $t^{-\gamma} N_t$ converges in distribution to a Gaussian distribution with variance $V$ and mean zero.
\end{lemma}
\begin{proof}
    Let $\epsilon>0$. Let $\delta>0$ be such that, for $W$ a standard one-dimensional Brownian motion, 
    \[
    \mathbb{P}( \sup_{t\in[ V-\delta, V+\delta] } |W_t-W_V|\geq \epsilon )\leq \frac{\epsilon}{2}.
    \]
    Such a $\delta$ exists by continuity of the Brownian motion. 
    Let $T_0$ be such that for all $T\geq T_0$, 
    \[
    \mathbb{P}(    |T^{-2\gamma} [N]_T-V|\geq \delta )\leq \frac{\epsilon}{2}.
    \]
    
    For a given $T\geq T_0$, the process $t\mapsto T^{-\gamma} N_t$ is a continuous local martingale with quadratic variation equal to $V_t\coloneqq T^{-2\gamma} [N]_t$. By Dambis--Dubins-–Schwarz theorem applied to the local martingale $T^{-\gamma} N$  there exists a Brownian motion $W$ (on a possibly extended probability space) such that for all $t$, $W_{V_t}=  T^{-\gamma} N_t$. In particular, $T^{-\gamma} N_T=W_{V_T}=W_{T^{-2\gamma} [N]_T}$. 
    Thus, 
    \begin{align}
    \nonumber
    \mathbb{P}(    |T^{-\gamma} N_T-W_{V}|\geq \epsilon)
    & =\mathbb{P}(    |W_{T^{-2\gamma} [N]_T}-W_{V}|\geq \epsilon)\\
    &\leq \mathbb{P}(    |T^{-2\gamma} [N]_T-V|\geq \delta )
    +\mathbb{P}( \sup_{t\in[ V-\delta, V+\delta] } |W_t-W_V|\geq \epsilon )\leq \epsilon.
    \label{eq:prohorov_distance}
    \end{align}
   
Pick arbitrary $a\in\R$ and $\epsilon'>0$.     
    Since $W_V$ is Gaussian,
    there exists $\epsilon\in(0,\epsilon'/2)$ such that $\mathbb{P}(a\leq W_V\leq a+\epsilon)\leq \epsilon'/2$. By~\eqref{eq:prohorov_distance} there exits $T_0>0$ such that for all $T\geq T_0$ we have
\begin{align*}
    \mathbb{P}(T^{-\gamma} N_T\leq a)&\leq \mathbb{P}(W_V\leq a+\epsilon)+ \mathbb{P}(    |T^{-\gamma} N_T-W_{V}|\geq \epsilon)\\
    & \leq \mathbb{P}(W_V\leq a+\epsilon)+
    \epsilon\leq \mathbb{P}(W_V\leq a)+\epsilon'/2+\epsilon\leq \mathbb{P}(W_V\leq a)+\epsilon'.
\end{align*}    
    This inequality (for $-W_V, -T^{-\gamma} N_T$ and $-a$) also yields $\mathbb{P}(W_V\leq a)-\epsilon'\leq\mathbb{P}(T^{-\gamma} N_T\leq a)$, implying 
    $|\mathbb{P}(T^{-\gamma} N_T\leq a)-\mathbb{P}(W_V\leq a)|\leq \epsilon'$ for all $T\geq T_0$ as claimed in the lemma.
\end{proof}

\begin{corollary}
\label{coro:variance}
    As $T\to \infty$, $T^{-\frac{1}{2}-\frac{\beta}{1+\beta}} M_T$ converges in distribution to a Gaussian distribution with variance $S^2$. 
\end{corollary}
\begin{proof}
This follows from Lemma~\ref{le:variance1} and Lemma~\ref{le:variance2} (with $\gamma=\frac{1}{2}+\frac{\beta}{1+\beta}$).
\end{proof}

\subsection{The local time part}
\label{subsec:loc_time}

By~\cite[Thm~2.2(ii), Eq.~(2.11)]{MMW}, there exists a constant $c_2>0$ such that $ L_t\sim c_2 t^{\frac{1}{1+\beta}} $ almost surely. Let $C>0$ be such that for all $t\geq 1$, $L_t \leq C t^{\frac{1}{1+\beta}}$ almost surely.
Recall \eqref{eq:lambda}: $\Lambda(x,y)= o(x^{\frac{3\beta-1}{2}})$ for $(x,y)\in\partial \cD$.
Since, by~\eqref{def:c1}, $X_t\sim c_1t^{\frac{1}{1+\beta}}$ as $t\to\infty$ almost surely, 
for an arbitrary $\epsilon>0$, there exists  $T_1$ such that for all $T\geq T_1$, $\mathbbm{1}_{Z_T\in \partial \cD}|\Lambda(Z_T)|\leq \epsilon  T^\frac{3\beta-1}{2(1+\beta) } $.
As for any $\mathcal{C}^1$ function $h:[T_1,T]\to\R$, Fubini's theorem implies the integration-by-parts equality
$\int_{T_1}^Th(t)\d L_t=h(T)(L_T-L_{T_1})+L_{T_1} (h(T)-h(T_1))-\int_{T_1}^Th'(t)L_t\d t$, setting $h(t)\coloneqq \epsilon  t^\frac{-1+3\beta}{2(1+\beta) }$ we get almost surely
\begin{align*}
\int_0^T \Lambda(Z_t) \d L_t
&\leq O(1)+ \int_{T_1}^T \epsilon  t^\frac{-1+3\beta}{2(1+\beta) } \d L_t \\
&= O(1) +\epsilon   T^\frac{-1+3\beta}{2(1+\beta) } L_T  + \epsilon \frac{1-3\beta}{2(1+\beta)} \int_{T_1}^T L_t  t^{\frac{3\beta-1}{2(1+\beta) }-1   }  \d t \\
&= O(1) +\epsilon   T^\frac{-1+3\beta}{2(1+\beta) } L_T  + \epsilon \frac{1-3\beta}{2(1+\beta)}C \int_{T_1}^T t^{\frac{3\beta-1}{2(1+\beta) }-1 +\frac{1}{1+\beta}  } \d t\\
&= \epsilon \big(c_2+C\frac{1-3\beta}{1+3\beta}\big) T^{ \frac{1+3\beta}{2(1+\beta) }  }+o(T^{ \frac{1+3\beta}{2(1+\beta) }  })\quad\text{as $T\to\infty$.}
\end{align*}
Since $\epsilon$ is arbitrary, for all $\beta\in (-\frac{1}{3}, 1)$, almost surely, 
\begin{equation}
\label{eq:localTime}
\int_0^T \Lambda(Z_t) \d L_t= o (T^{ \frac{1+3\beta}{2(1+\beta) }  } ) \quad\text{as $T\to\infty$.}
\end{equation}

\subsection{The drift part}
Since almost surely $X_t\underset{t\to \infty}\longrightarrow \infty$ and $X_t=O( t^{\frac{1}{1+\beta}})$, by \eqref{eq:mu},  almost surely \[\mu(Z_t)=\frac{s_0 \os}{2 c_0}+o(t^{\frac{\beta-1}{2(1+\beta)}}).\] 
By applying \eqref{eq:deterministicEstimationsIntegral1} to the continuous function $t\mapsto 
\mu(Z_t)-\frac{s_0 \os}{2 c_0}$, we deduce that 
almost surely,
\begin{equation}
\label{eq:drift}
\int_0^T \mu(Z_t)\d t=\frac{s_0 \os}{2c_0}  T+o( T^{\frac{1+3\beta}{2(1+\beta)}} ).
\end{equation}

\subsection{Conclusion}
\label{subsec:Conclusion}

\begin{proof}[Proof of Theorem~\ref{th:main1}] By substituting~\eqref{eq:localTime} and~\eqref{eq:drift} into the semimartingale decomposition~\eqref{eq:itoXi} of $\xi$, we obtain, almost surely,
\begin{equation}
\label{eq:xi1}
\xi_t=B(g(Z_t))= \frac{s_0\os}{2c_0}  t+t^{\frac{1}{2}+\frac{\beta}{1+\beta}}  \tilde{M}_t+o(t^\frac{1+3\beta}{2(1+\beta)} ),
\end{equation}
where, by Corollary~\ref{coro:variance},  $\tilde{M}_t \coloneqq t^{-\frac{1}{2}-\frac{\beta}{1+\beta}} M_t $ converges in distribution to  a centred Gaussian random variable $N$ with variance $S^2$ given in Lemma~\ref{le:variance1}. Note also that $\frac{1+3\beta}{2(1+\beta)}=\frac{1}{2}+ \frac{\beta}{1+\beta}$.

Since almost surely, $X_t\sim c_1 t^{\frac{1}{1+\beta}}$ and, by~\eqref{eq:g}, $g(x,y)=x(1+O(x^{\beta-1}))$ as $x\to\infty$, we get 
$g(Z_t)=X_t(1+O(t^\frac{\beta-1}{1+\beta}))$ almost surely. By substituting this into $B(x)\underset{x\to \infty}= \frac{a_\infty}{1+\beta}x^{1+\beta}+o(x^{\frac{3\beta+1}{2} }) $, we obtain almost surely
\begin{equation}
\label{eq:xi2}
\xi_t
=\frac{a_\infty}{1+\beta} X_t^{1+\beta} +o( t^\frac{1+3\beta}{2(1+\beta)})\qquad\text{as $t\to\infty$.}
\end{equation}
By equating \eqref{eq:xi1} with \eqref{eq:xi2}, we get, almost surely,
\[
\frac{a_\infty}{1+\beta} X_t^{\beta+1}\underset{t\to \infty}=  \frac{s_0\os}{2c_0}   t +  t^{\frac{1}{2}+\frac{\beta}{1+\beta}}  \tilde{M}_t+o(t^{\frac{1}{2}+ \frac{\beta}{1+\beta} }).
\]
Recall from~\eqref{def:c1} that 
$c_1=( \frac{(1+\beta) s_0\os}{2 a_\infty c_0})^{\frac{1}{1+\beta}}$ and denote $\upsilon\coloneqq \frac{1}{2}-\frac{\beta}{1+\beta}>0$. Hence for a process $K$ satisfying $K_t\to0$ almost surely as $t\to0$, we have
\begin{equation}
\label{eq:idk2}
X_t=  c_1 t^{\frac{1}{1+\beta}} \Big(1+ \frac{1+\beta}{c_1^{1+\beta}a_\infty} t^{-\upsilon } (\tilde{M}_t + K_t)  \Big)^{\frac{1}{1+\beta}}.
\end{equation}
Note that the  factor in brackets on the right-hand side of~\eqref{eq:idk2} is very close to $1$ as $t\to\infty$. However care must be taken as $\tilde{M}_t$ converges only in distribution and not almost surely.
Since $\frac{1}{1+\beta}-\upsilon=\frac{1}{2}$, the Taylor approximation implies
\begin{align}
\nonumber
 X_t &=c_1
t^{\frac{1}{1+\beta}} \Big(1 +  \frac{1}{c_1^{1+\beta}a_\infty} t^{-\upsilon }(\tilde{M}_t+K_t) +C' t^{-2\upsilon }(\tilde{M}_t+K_t)^2(1+\zeta_t)^{\frac{1}{1+\beta}-2}  \Big) 
\\
&=c_1
t^{\frac{1}{1+\beta}} +  \frac{c_1^{-\beta}}{a_\infty} t^{\frac{1}{2} }(\tilde{M}_t+K_t) +C' t^{\frac{1}{2}-\upsilon }(\tilde{M}_t+K_t)^2(1+\zeta_t)^{\frac{1}{1+\beta}-2}, 
\label{eq:expansion_withou_E_t}
\end{align}
where  $\zeta_t$ is in the interval between $0$ and $\frac{1+\beta}{c_1^{1+\beta} a_\infty} t^{-\upsilon }(\tilde{M}_t+K_t)$ and $C'$ is some constant.

Denote for some $\delta\in (0,\upsilon)$ the event $E_t\coloneqq \{|\tilde{M}_t| \leq t^\delta\}$.
Since $\tilde{M}_t$ converges in distribution, we have $\mathbb{P}(E_t) \underset{t\to \infty}\longrightarrow 1$ (i.e.~$\mathbbm{1}_{E_t}$ converges in probability to~$1$). 
Since $K_t\to0$ as $t\to\infty$ almost surely and $\mathbbm{1}_{E_t}|\tilde M_t|\leq t^{\delta-\upsilon}$, the process
$\tilde K_t\coloneqq \frac{c_1^{-\beta}}{a_\infty} K_t+C' t^{-\upsilon }(\tilde{M}_t+K_t)^2(1+\zeta_t)^{\frac{1}{1+\beta}-2}$
tends to zero: $\mathbbm{1}_{E_t}\tilde K_t\to0$ as $t\to\infty$ almost surely. By~\eqref{eq:expansion_withou_E_t} we obtain
\[ 
\mathbbm{1}_{E_t} \frac{ X_t -c_1 t^{\frac{1}{1+\beta}} }{\sqrt{t}}= 
\mathbbm{1}_{E_t}\frac{c_1^{-\beta}}{a_\infty}\tilde{M}_t+ \mathbbm{1}_{E_t}\tilde K(t).
\]
Since, as $t\to\infty$, $\tilde M_t$ converges weakly to a centred Gaussian variable $N$ with variance $S^2$ given in Lemma~\ref{le:variance1} above, 
$\mathbbm{1}_{E_t}\overset{(proba)}\longrightarrow1$ and $\mathbbm{1}_{E_t}\tilde K_t\to0$  almost surely,
we get
\begin{equation*}
\frac{X_t -c_1 t^\frac{1}{1+\beta} }{\sqrt{t}} \underset{t\to \infty}\longrightarrow  \frac{c_1^{-\beta}}{a_\infty} N.
\end{equation*}
Having established the assumption in Proposition~\ref{prop:joint}, Theorem~\ref{th:main1} follows.
\end{proof}

To conclude the paper, let us show that the condition $\beta>-1/3$ in Theorem~\ref{th:main1} is optimal in the sense that it cannot be relaxed to $\beta>\beta_c$ for any $\beta_c\in(-1,-1/3)$.
\begin{proposition}
    \label{prop:nogo}
    If Assumptions~\eqref{hyp:D1}, \eqref{hyp:C}, \eqref{hyp:V+} hold and Assumption \eqref{hyp:D3} holds with $\beta\in(-1,-\frac{1}{3})$, then 
    $t^{-\frac{1}{2}} (X_t-c_1 t^{\frac{1}{1+\beta}})$ does \textbf{not} converge in distribution. 
\end{proposition}
\begin{proof} 
We are going to prove that, for all $\beta\in(-1,1)$, on an event of positive probability,   $t^{\frac{\beta}{1+\beta}} (X_t-c_1 t^{\frac{1}{1+\beta}})$ does not converge to $0$. If (and only if)
$\beta<-\frac{1}{3}$, we have $\frac{\beta}{1+\beta}=1-\frac{1}{1+\beta}<-\frac{1}{2} $, implying $t^{-\frac{1}{2}} (X_t-c_1 t^{\frac{1}{1+\beta}})$ does not converge in distribution as $t\to\infty$.

Let $Z$, $Z'$ be two solutions of the SDER \eqref{eq:SDER}, coupled in such a way that, in an event $E$ of positive probability, for all $t\geq 1$, $Z_t=Z'_{t+1}$. Such a coupling exists since the total variation distance between the distributions of $Z_1$ and $Z'_{2}$ is strictly smaller than $1$.
Since $X_t\sim c_1 t^{\frac{1}{1+\beta}}$ as $t\to\infty$, almost surely on the event $E$ we have
\begin{align*}
& t^{\frac{\beta}{1+\beta}} (X_t-c_1 t^{\frac{1}{1+\beta}})
- (t+1)^{\frac{\beta}{1+\beta}} (X'_{t+1}-c_1 (t+1)^{\frac{1}{1+\beta}})\\
=& t^{\frac{\beta}{1+\beta}}X_t( 1-(1+t^{-1})^{\frac{\beta}{1+\beta}} ) 
+c_1   
\underset{t\to\infty}\longrightarrow 
c_1  (-\frac{\beta}{1+\beta} +1 )\neq 0.
\end{align*}
Thus, it cannot hold that both $t^{\frac{\beta}{1+\beta}} (X_t-c_1 t^{\frac{1}{1+\beta}})$ and 
    $(t+1)^{\frac{\beta}{1+\beta}} (X'_{t+1}-c_1 (t+1)^{\frac{1}{1+\beta}})$ converge in probability to $0$. 
\end{proof}

\section*{Funding} AM is supported in part by
EPSRC grants  EP/V009478/1 and  EP/W006227/1. IS is funded by  EPSRC grant  EP/V009478/1. 
 AW is supported in part by EPSRC grant EP/W00657X/1. A part of this work was carried
 out during the programme ``Stochastic systems for anomalous diffusion'' (July--December 2024) hosted by the  Isaac Newton Institute, under EPSRC grant EP/Z000580/1.

\appendix

\section{Deterministic computations}
\label{Appendix:deterministic}

\subsection{Proof of Lemma~\ref{le:smooth}}
  For the fist point, notice first that the considered set is clearly $\cC^1$ on any open subset which contains no point $(0,x,y)$. 
  
  For $h>0$, let $x_h$ be the unique value such that   $\cDtot\cap (\{(h,x_h)\}\times \R^d)$ consists of a single point $(h,x_h,0)$, i.e.~$x_h\coloneqq -\frac{h^{-\gamma}}{b(h^{-\gamma})}$, and set $x_h\coloneqq-\infty$ for $h\leq 0$. For $h>0$ and $x\geq x_h$, set $f(x,h)= \frac{b(b(h^{-\gamma})x + h^{-\gamma}   )}{b(h^{-\gamma})}$. For $h<0$, set $f(h,x)=1$ for all $x\in \R$. For $y' \in \R^{d-1}$ with $|y'|\leq f(h,x)$, let $g(h,x,y')= \sqrt{f(h,x)^2-|y'|^2}$.
  Then $\partial \cDtot$ equals
\begin{align*}
    \partial \cDtot & = \{(h,x,y): x \geq x_h, 
|y|^2= f(h,x)\}\\
& =\{(h,x,y_1,y'): x \geq x_h, |y'|\leq f(h,x), y_1=\pm g(h,x,y')   \}.
\end{align*}
From rotational invariance, it suffices to show that for all $x>x_h$, $g$ is $\cC^1$ in a small neighbourhood of $(0,x,0)$,  which in turn follows from $f$ being $\cC^1$ in a small neighbourhood of $(0,x)$. Let us prove this. 
First,
  \begin{align*}
  \partial_h f(x,h)
  =&b(h^{-\gamma})^{-2}
\big( 
  -b(b(h^{-\gamma})x+h^{-\gamma}   )(-\gamma)h^{-\gamma-1} b'(h^{-\gamma})\\
  &\hspace{1.5cm} 
  + b'(b(h^{-\gamma})x+h^{-\gamma}   )b(h^{-\gamma})(-\gamma)h^{-\gamma-1}(1+b'(h^{-\gamma}))
  \big)\\
  =& \frac{\gamma h^{2\beta \gamma-\gamma-1}(1+o(1)) }{a_\infty^2}\Big(b'(h^{-\gamma})b(b(h^{-\gamma})x+h^{-\gamma}   )  -b'(h^{-\gamma})b(h^{-\gamma})\\
  &\hspace{1.5cm}   
  +b'(h^{-\gamma})b(h^{-\gamma})-b'(b(h^{-\gamma})x+h^{-\gamma}   )b(h^{-\gamma})\\
  &\hspace{1.5cm} + b(h^{-\gamma})b'(h^{-\gamma}) b'(b(h^{-\gamma})x+h^{-\gamma}   )   )\Big).
  \end{align*}
  Using
  \[ b(h^{-\gamma})=a_\infty h^{-\beta \gamma}+o(h^{-\gamma\frac{3\beta-1}{2} } ), \qquad
   b'(h^{-\gamma})=\beta a_\infty h^{\gamma-\beta \gamma}+o(h^{-\gamma\frac{3\beta-3}{2} } ),\]
  as $h\to 0$, we deduce
  \[ b(b(h^{-\gamma})x+h^{-\gamma}   )=a_\infty h^{-\beta \gamma}+o(h^{-\gamma\frac{3\beta-1}{2} } )+O( h^{\gamma-2\gamma \beta}), \]
  and
  \[ b'(b(h^{-\gamma})x+h^{-\gamma}   )=\beta a_\infty h^{\gamma-\beta \gamma}+o(h^{-\gamma\frac{3\beta-3}{2} } )+O( h^{2 \gamma-2\gamma \beta}).\]
  Using these four estimations and the fact that $\gamma\geq \frac{2}{1-\gamma}>\frac{1}{1-\gamma}$, we deduce that
  \[ h^{2\beta \gamma-\gamma-1} (b'(h^{-\gamma})b(b(h^{-\gamma})x+h^{-\gamma}   )  -b'(h^{-\gamma})b(h^{-\gamma}) )
  = o(1), \]
  \[ h^{2\beta \gamma-\gamma-1} (b'(h^{-\gamma})b(h^{-\gamma})-b'(b(h^{-\gamma})x+h^{-\gamma}   )b(h^{-\gamma}) )
  = o(1) ,\]
  and
  \[ h^{2\beta \gamma-\gamma-1} b(h^{-\gamma})b'(h^{-\gamma}) b'(b(h^{-\gamma})x+h^{-\gamma}   )
  = o(1) ,\]
  from which it follows that $\partial_h f(x,h)\underset{h\to 0, h>0}\longrightarrow 0=\lim_{h\to 0, h<0 }\partial_h f(x,h)  $ for all $x>x_h$.
  Beside,
  $\partial_x f(x,h) = b'(b(h^{-\gamma} x +h^{-\gamma}))\sim_{h\to 0} a_\infty \beta h^{-\gamma(\beta-1)}\underset{h\to 0}\longrightarrow 0$.
  It follows that $f$ is $\cC^1$ on $\{ (h,x): h\leq 0 \mbox{ or } x>x_h \}  $, from which we conclude that $\partial \cDtot$ is $\cC^1$.
    
  As for $\sigmatot$, we have
  \[\partial_x \sigmatot(h,x,y)=\Diag(0_\R,b(h^{-\gamma} )  (\partial_x \sigma)(b(h^{-\gamma} )x +h^{-\gamma}) ,   b(h^{-\gamma} ) y)     \]
  Using first Assumption \eqref{hyp:D3} and then Assumption \eqref{hyp:S}, as $x\to\infty$ we have
\begin{align*}
  &\| b(h^{-\gamma} )  (\partial_x \sigma)(b(h^{-\gamma} )x +h^{-\gamma}) ,   b(h^{-\gamma} ) y  ) \|\\ 
 &\sim
   \| a_\infty (b(h^{-\gamma} )x +h^{-\gamma})^\beta   (\partial_x \sigma)(b(h^{-\gamma} )x +h^{-\gamma})  ,   b(h^{-\gamma} ) y )\|\\
  & \leq (1+o(1)) a_\infty \lim_{r\to \infty} \sup_y r^\beta \| (\partial_x \sigma)(r,y)\|=0.
\end{align*}
  We deduce
  \[ \lim_{h\to 0} \partial_x \sigmatot(h,x,y) =0.\]
  We similarly deduce that \[ \lim_{h\to 0} \partial_y \sigmatot(h,x,y) =0,\]
  whilst \[ \lim_{h\to 0} \partial_h \sigmatot(h,x,y) =0\]
  is proved by also using the fact that $\gamma>\frac{1}{\epsilon}$.
  It follows that $\sigmatot$ is $\cC^1$.

  The computations for the reflection vector field $\phitot$ are analogous as the ones above and are omitted for brevity. This concludd the proof of Lemma~\ref{le:smooth}.

\subsection{Construction of the coupling windows}
\label{sec:s-s-s}

The aim of this section is to prove Lemma~\ref{le:sequenceDet}. 
The proof proceeds by taking a function $\mathpzc{s}_1$
and showing that one can construct a function $\mathpzc{s}_2 \leq \mathpzc{s}_1$ which possesses specified growth and smoothness properties, captured by the  following condition on a (deterministic) function $\mathpzc{s}:(0,\infty) \to (0,\infty)$.
\begin{description}
\item
[\namedlabel{hyp:*}{$*$}] The function $\mathpzc{s}$ is non-decreasing, with $\lim_{T \to \infty} \mathpzc{s} (T) = \infty$, and continuously differentiable, with derivative $\mathpzc{s}'$. Furthermore, as $T\to \infty$, $\mathpzc{s}(T) = O( \log T )$ and $\mathpzc{s}'(T) = o(T^{-\frac{2\beta}{1+\beta}})$. 
\end{description}
The next result shows that functions which satisfy~\eqref{hyp:*} are readily available.

\begin{lemma}
    \label{le:sequenceDetPrelim}
    Assume that $\mathpzc{s}_1:(0,\infty) \to (0,\infty)$ is such that $\lim_{T \to \infty} \mathpzc{s}_1(T)=\infty$. Then, there exists $\mathpzc{s}_2:(0,\infty) \to (0,\infty)$ satisfying~\eqref{hyp:*} and such that $\mathpzc{s}_2\leq \mathpzc{s}_1$.
\end{lemma}
\begin{proof}
    One can take for example $\mathpzc{s}_2= I ( I (i(\mathpzc{s}_1)))$, where \[i(f): t \mapsto \min \Bigl[ \log (1+t),  \inf_{u \geq t} f(u) \Bigr], \qquad I(g): t  \to \frac{1}{t}\int_0^t g(u)\d u.\]
    The desired properties are elementary to check.
\end{proof}
\begin{lemma}
    \label{le:sequenceDet}
    Let $\mathpzc{s}_1:(0,\infty) \to (0,\infty)$ is such that $\lim_{T \to \infty} \mathpzc{s}_1(T)=\infty$, and $\theta:(0,\infty)\to (0,\infty)$ such that $\theta(T)\sim_{T\to \infty} C T^{\frac{ 2 \beta}{1+\beta}}$ for some $C>0$ and $\beta\in (-1,1)$. Then, there exists $T_0>0$, 
    $\mathpzc{s}_2:(0,\infty) \to (0,\infty)$ and $S:(0,\infty)\to (0,\infty)$, such that 
    $4 \mathpzc{s}_2\leq \mathpzc{s}_1$, 
    $|S(T)-T|=O( T^\frac{2\beta}{1+\beta}\log T ) $ and $\mathpzc{s}_2(T)\to\infty$ as $T \to \infty$  
    and, for all $T>T_0$, we have
    \begin{equation}
    \label{eq:incluDet}
    [T, T+ \frac{1}{4}q(T)]
    \subset [S(T)+\frac{1}{2}  q(S(T)) , S(T)+  q(S(T))], \quad \text{where} \quad q \coloneqq \theta \cdot \mathpzc{s}_2.
    \end{equation}
    The functions $\mathpzc{s}_2$ and $S$ (but not $T_0$) can be chosen so that they depend on $\mathpzc{s}_1$ but not on $\theta$.   
\end{lemma}    
\begin{remark}
By the last sentence, we mean that the correct order between the quantifiers is in fact \[\forall \mathpzc{s}_1, \exists (\mathpzc{s}_2, S): \forall \theta, \exists T_0: \forall T\geq T_0, \dots.\]
Lemma~\ref{le:sequenceRand} is obtained by applying Lemma~\ref{le:sequenceDet} to a deterministic function $\mathpzc{s}_1$ and the random function $\theta(T)=4 b(X_T)^2$. Thus, we are provided with functions $\mathpzc{s}_2, S$ which are deterministic but a time $T_0$ which is random.  
\end{remark}

\begin{proof}[Proof of Lemma~\ref{le:sequenceDet}]
    First let $\mathpzc{s}_2$ be any function obtained by applying Lemma~\ref{le:sequenceDetPrelim} to the function $\mathpzc{s}_1$.
    Set then $f: u\mapsto u+ \frac{2}{3} C u^\frac{2\beta}{1+\beta} \mathpzc{s}_2(u)$. Then $f'(t)\underset{t\to \infty}\longrightarrow 1$, so that there exists $S_1$ such that $f$ is increasing on $[S_1,\infty)$. Let $T_1=f(S_1)$, and define $S$ to be the inverse of $f:[S_1,\infty)\to [T_1,\infty)$.
    Since $f(t)\underset{t\to \infty}\sim t$,  it follows $S(t)\underset{t\to \infty}\sim t$.

    Let $S_2$ be such that for all $S\geq S_2$, \[\theta(S)S^{-\frac{2\beta}{1+\beta}} \in [\frac{8}{9} C , \frac{4}{3}   C ], \qquad \text{i.e.} \qquad 
\frac{3}{4} \theta(S) \leq   C S^{\frac{2\beta}{1+\beta}}\leq \frac{9}{8} \theta(S).
    \] 
    Set $T_2 \coloneqq f(S_2)$, 
    $T_0\coloneqq\max(T_1,T_2)$. 
    
    For $T\geq  T_0$,
    \begin{align*}
    T=f(S(T))
    &= S(T)+\frac{2}{3} C S(T)^\frac{2\beta}{1+\beta} \mathpzc{s}_2( S(T) )\\
    &\geq S(T)+\frac{2}{3}\cdot\frac{3}{4} \theta(S(T)) \mathpzc{s}_2( S(T) )\\
&= S(T)+\frac{1}{2} \theta(S(T)) \mathpzc{s}_2( S(T) ),
    \end{align*}
and
   \begin{align*}
    T+ \frac{1}{4}\theta(T)\mathpzc{s}_2(T)
    &=S(T)+ \frac{2}{3} C S(T)^\frac{2\beta}{1+\beta} \mathpzc{s}_2(S(T))
    + \frac{1}{4} \theta(T)\mathpzc{s}_2(T)\\
    &\leq     S(T)+ \frac{2}{3}\cdot \frac{9}{8}  \theta(S(T)) \mathpzc{s}_2(S(T))
    + \frac{1}{4} \theta(S(T))  \mathpzc{s}_2(T) \\
    &=     S(T)+   \theta(S(T)) \mathpzc{s}_2(S(T)),
    \end{align*}
    which concludes the proof.
\end{proof}

\subsection{Proof of Lemma~\ref{le:estimTechnic} }
\label{Appendix:52}

Recall $T_1\coloneqq 1$, $T_{n+1}\coloneqq T_n+C^2T_n^\frac{2\beta}{1+\beta} \mathpzc{s}_4(T_n)$, and recall that $\mathpzc{s}_4(T)\geq 1$ diverges to $\infty$ as $T\to \infty$ more slowly than any positive power of $T$. We have set  $\alpha=\frac{2\beta}{1+\beta}>-1$, and the goal is to prove, as $n\to \infty$,
\begin{equation}
\label{eq:temp:app}
\sum_{k=1}^n (T_{k+1}-T_k) |T^\alpha_k-T^\alpha_{k+1}|=o(T_n^{1+\alpha})\quad \& \quad \sum_{k=1}^n (T_{k+1}-T_k) T^\alpha_k =\frac{T_n^{1+\alpha}}{1+\alpha} (1+o(1)).
\end{equation}

The first asymptotic equality implies the second, by comparison to the integral: if the first asymptotic equality holds, 
for $\alpha>0$ we get 
\begin{align*}
\sum_{k=1}^n (T_{k+1}-T_k) T^\alpha_k 
& \leq 
\sum_{k=1}^n \int_{T_k}^{T_{k+1}}  t^\alpha  \d t 
 \leq \sum_{k=1}^n (T_{k+1}-T_k) T^\alpha_{k+1}\\
&= \sum_{k=1}^n (T_{k+1}-T_k) T^\alpha_k +o (T_n^{1+\alpha}),\quad\text{as $n\to\infty$,}
\end{align*}
where the last equality follows from the first asymptotic equality in \eqref{eq:temp:app}. 
Thus, as $n\to\infty$,
\[ 
\sum_{k=1}^n (T_{k+1}-T_k) T^\alpha_k = \int_{T_1}^{T_{n+1}}  t^\alpha  \d t +o (T_n^{1+\alpha})=\frac{T_{n+1}^{1+\alpha}}{1+\alpha} (1+o(1))=\frac{T_n^{1+\alpha}}{1+\alpha} (1+o(1)).
\]

For $\alpha<0$, the same computation but with inequalities all reversed also allows to conclude. For $\alpha=0$, both bounds are trivial.

We now prove the first bound in~\eqref{eq:temp:app}.
Note that, as $k\to \infty$, \begin{equation} \label{eq:temp:Talpha} T^\alpha_{k+1}-T^\alpha_{k}=T^\alpha_{k}(  \big(\frac{T_{k+1}}{T_k}\big)^\alpha-1)\sim \alpha  T^\alpha_{k}(\frac{T_{k+1}}{T_k}-1)=\alpha  T_{k}^{\alpha-1} (T_{k+1}-T_k)=o(T_{k}^{\alpha})
.\end{equation}
Recall for $f_k,g_k>0$, such that  $f_k=o(g_k)$, we have $\sum_{k=1}^n f_k=O(1)$ or $\sum_{k=1}^n f_k=o(\sum_{k=1}^n g_k)$. In the case $\alpha\geq 0$, i.e.~$\beta>0$, 
since 
\[ \sum_{k=1}^n  T_{k}^{\alpha}  (T_{k+1}-T_k)\leq \sum_{k=1}^n  T_{n}^{\alpha}  (T_{k+1}-T_k)=T_n^\alpha (T_{n+1}-T_1)\sim T_{n}^{1+\alpha}\]
and $( T^\alpha_{k+1}-T^\alpha_{k})(T_{k+1}-T_k)=o( T_{k}^{\alpha}  (T_{k+1}-T_k)) $,
we deduce indeed 
\[ 
\sum_{k=1}^n( T^\alpha_{k+1}-T^\alpha_{k})(T_{k+1}-T_k) =o(T_{n}^{1+\alpha}).
\]

The case $\alpha<0$ requires extra work. 
First, we claim that 
there exists a constant $C_2$ such that for all $k$, $T_k\geq C_2 k^\frac{1+\beta}{1-\beta}$. Indeed, set $a_k= T_k^\frac{1-\beta}{1+\beta}$. Then, the recurrence relation defining $T_{k+1}$ gives 
\[
a_{k+1}-a_k= a_k\Big( \big(1+C^2 a_k^{-1} \mathpzc{s}_4(T_k) \big)^\frac{1-\beta}{1+\beta} -1 \Big)\sim C^2  \frac{1-\beta}{1+\beta}  \mathpzc{s}_4(T_k) \geq  C^2  \frac{1-\beta}{1+\beta} , 
\]
where the asymptotic equivalence comes from the fact that $a_k^{-1} \mathpzc{s}_4(T_k) \to 0$ as $k\to \infty$. 
It follows that $a_k\geq a_1+C^2  \frac{1-\beta}{1+\beta} (k-1)$ for all $k$. Hence there exists a constant $C_3>0$ such that for all positive integers $k$ we have $a_k\geq C_3 k$, implying $T_k\geq C_2 k^\frac{1+\beta}{1-\beta}$ for $C_2\coloneqq  C_3^\frac{1+\beta}{1-\beta}$.

By~\eqref{eq:temp:Talpha}, there exists a constant $C_4>0$ such that for all $k$, 
$T^\alpha_{k}-T^\alpha_{k+1} \leq C_4    T_{k}^{\alpha-1} (T_{k+1}-T_k)$. Since $\beta<0$, it holds 
$ \frac{1+\beta}{1-\beta}(3\alpha-1)= \frac{5\beta-1}{1-\beta}<-1$. Let $\epsilon>0$ be such that $ \theta\coloneqq \frac{1+\beta}{1-\beta}(3\alpha-1+2\epsilon)$ remains smaller than $-1$. Let then $C_5$ be such that for all $k$, $C^2 \mathpzc{s}_4(T_k)\leq C_5 T_k^\epsilon$, so that $T_{k+1}-T_k\leq C_5 T_k^{\alpha+\epsilon}$. Then, since $3\alpha+2\epsilon-1<0$, we get
\begin{align*}(T_{k+1}-T_k) ( T_{k}^\alpha-T_{k+1}^\alpha)
& \leq C_4 (T_{k+1}-T_k)^2 T_{k}^{\alpha-1}\\
&\leq C_4C_5^2 T_{k}^{3\alpha+2\epsilon-1}
\leq C_4C_5^2 C_2^{3\alpha+2\epsilon-1} k^{\frac{1+\beta}{1-\beta}(3\alpha+2\epsilon-1)} = C_6 k^{\theta}.
\end{align*}
Since $\theta<-1$,
\[ 
\sum_{k=1}^n (T_{k+1}-T_k) ( T_{k}^\alpha-T_{k+1}^\alpha)
\leq C_6 \sum_{k=1}^n k^{\theta}=O(1)=o(T_n^{1+\alpha})\quad\text{as $n\to\infty$.}
\]

\subsection{Proof of the deterministic limits in Lemma~\ref{le:deterministicEstimations} }

From the It\^o formula applied to $\xi=B(g(Z))$, and recalling that $B'=b_0$, and that the quadratic variation of $Z$ is given by $\d \langle Z_i, Z_j\rangle_t = \Sigma_{i,j}(Z_t) \d t $, we obtain the following formula for the functions $\mu, \Lambda, f$:
\begin{align}
f&\coloneqq (b_0\circ g)^2\| \Sigma^{\frac{1}{2}}\nabla g\|_{d+1}^2,\qquad
\Lambda  \coloneqq (b_0\circ g)  \langle \phi, \nabla g\rangle,\nonumber \\
\mu &\coloneqq\frac{1}{2}( b_0\circ g\Delta_{\Sigma} g+b_0'\circ g ) \|\Sigma^{\frac{1}{2} } \nabla g\|_{d+1}^2), \quad
\text{where} \quad \Delta_{\Sigma} g \coloneqq
\sum_{i,j=0}^d \Sigma_{i,j}\partial_i\partial_j g. \label{def:DeltaSigma}
\end{align}
Note that in the definition of $\Delta_{\Sigma} g$ all coordinates of the vector $(x,y)\in\R^{1+d}$ play the same role and hence it is more convenient to denote the partial derivatives by
$\partial_0=\partial_x$ and $\partial_i=\partial_{y_i}$ for $i\in\{1,\ldots,d\}$.

\label{Appendix:62}
    The asymptotic properties of $g$ follow directly from the fact $b(x)\sim a_\infty x^\beta$. 
   Using then $g(x,y)\underset{x\to \infty}{=}x(1+O(x^{\beta-1}))$ and $b(x) \underset{x\to \infty}=a_\infty {x}^{\beta} +o({x}^{\frac{3\beta-1}{2}})$, we have
    \begin{equation}
    \label{eq:bg}
    b(g(x,y))\underset{x\to \infty}=a_\infty x^\beta+o({x}^{\frac{3\beta-1}{2}}).
    \end{equation}
    Recall that limits as
     $x\to \infty$ are uniform in the $y$ component (as stated above Assumption~\eqref{hyp:C}).
    Moreover,
    \begin{equation}
    \label{eq:dbg}
    b'(g(x,y))\underset{x\to \infty}\sim a_\infty \beta x^{\beta-1}.
    \end{equation}
    Exact and asymptotic  expressions for the derivatives of $g$ up to order $2$ are given as follows:
    \begin{alignat}{3}
   & \partial_x g (x,y)&&=1-\frac{s_0}{2c_0}\frac{|y|^2 b'(x)}{b(x)^2} &&
    \underset{x\to \infty}{=}1+o( x^{\beta-1 } )   , \label{eq:temp:g1}\\
    &\partial_{y_i} g (x,y)&& =\frac{s_0}{c_0} \frac{y_i}{b(x) }&&\underset{x\to \infty}{=}\frac{s_0}{c_0 a_\infty} \frac{y_i}{x^{\beta}}+o(x^{\frac{-1-\beta}{2} } )\label{eq:temp:g2}\\
    &\partial_x\partial_x g(x,y)&& \underset{x\to \infty}{=} \frac{s_0}{2c_0}  \|y\|^2 \Big( \frac{2 b'(x)^2}{b(x)^3}-\frac{b''(x)}{b(x)^2} \Big)&&
    \underset{x\to \infty}{=}O(x^{\beta-2})\underset{x\to \infty}{=}o(x^{\frac{-1-\beta}{2} }), \label{eq:temp:g3}\\
    &\partial_x\partial_{y_i} g(x,y)&&= - \frac{s_0}{c_0}  y_i \frac{b'(x)}{b(x)^2}&&
    \underset{x\to \infty}{=}O(x^{-1 })\underset{x\to \infty}{=}o(x^{\frac{-1-\beta}{2} }), \label{eq:temp:g4}\\
    &\partial_{y_i}\partial_{y_j} g(x,y)&&=\frac{s_0}{c_0} \frac{\delta_{i,j}}{b(x)}&&\underset{x\to \infty}{=}\delta_{i,j} \frac{s_0}{c_0 a_\infty} x^{-\beta}+o(x^{\frac{-1-\beta}{2} }).\label{eq:temp:g5}
    \end{alignat}
    From the asymptotic  properties in~\eqref{eq:temp:g1} and \eqref{eq:temp:g2}, and the convergence of $\Sigma(z)$ as $z\to \infty$, we obtain 
    \begin{align}
    \| \Sigma^{\frac{1}{2}}\nabla g(x,y)\|_{d+1}^2&= \langle \nabla g(x,y), \Sigma(x,y)\nabla g(x,y)\rangle \nonumber \\
    &= \Sigma^{\infty}_{0,0}+ \frac{2 s_0}{c_0 a_\infty} \frac{1}{x^\beta} \sum_{i=1}^d \Sigma^{\infty}_{0,i} y_i + \frac{s_0^2}{c_0^2 a^2_\infty} \frac{1}{x^{2\beta}} \sum_{i,j=1}^d \Sigma^{\infty}_{i,j} y_i y_j + o(1). \label{eq:scalargood}\\
    &=O(1). \label{eq:scalarbad}
    \end{align}
    Equation \eqref{eq:covariance} follows directly from \eqref{eq:bg} and \eqref{eq:scalargood}.

   Recall the definition  \eqref{def:DeltaSigma} of $\Delta_\Sigma$. Using \eqref{eq:temp:g3}, \eqref{eq:temp:g4}, \eqref{eq:temp:g5}, and the fact that $\Sigma$ is bounded, we obtain 
    \[\Delta_\Sigma g(x,y)
    = \sum_{i=1}^d \Sigma_{i,i}(x,y) \frac{s_0}{c_0 a_\infty} x^{-\beta}+ o(x^{\frac{-1-\beta}{2}}).\]
    Using the fast convergence  of $\tr(\Sigma)-\Sigma_{0,0}$ to $\os$ (Assumption~\eqref{hyp:C}), we deduce 
    \begin{equation}
    \label{eq:Delta}
    \Delta_\Sigma g(x,y)=\frac{s_0 \os}{c_0 a_\infty}x^{-\beta}+ o(x^{\frac{-1-\beta}{2}}).
    \end{equation}

    Combining the asymptotic properties \eqref{eq:bg}, \eqref{eq:Delta}, \eqref{eq:dbg}, and \eqref{eq:scalarbad}, we obtain
    \begin{align*}
    \mu(x,y)&= \frac{1}{2} \big( a_\infty x^\beta + o( x^\frac{3\beta-1}{2} ) \big)\big( \frac{s_0 \os}{c_0 a_\infty}x^{-\beta} +o ( x^\frac{-\beta-1}{2} ) \big)+O(x^{\beta-1})\nonumber \\
    &= \frac{s_0 \os}{2 c_0} +o(x^{\frac{\beta-1}{2}}),
    \end{align*}
    as claimed. 

    We finally estimate $\Lambda$. Using the asymptotic behaviour of the first order derivatives of $g$, \eqref{eq:temp:g1} and \eqref{eq:temp:g2}, 
    we get $\nabla g(x,y)=(1+o(x^{\beta-1}),\frac{s_0}{c_0}\frac{y}{b(x)})$.
    The fact that $\phi$ is bounded and $ \langle\phi^{(d)}_\infty (u), u \rangle=-c_0<0$ for all $u\in\mathbb{S}^{d-1}$, and using then the fast convergence of $\phi$ (Assumption~\eqref{hyp:V+}), for $(x,y)\in\partial \cD$ we obtain 
    \begin{align*}
    \langle \phi(x,y), \nabla g(x,y) \rangle
    &= 
    \langle (s_0,\phi^{(d)}_\infty (\frac{y}{b(x)})), \nabla g(x,y) \rangle + o( x^\frac{\beta-1}{2}) \\
   & = 
   s_0+o(x^{\beta-1})+\frac{s_0}{c_0 } \langle\phi^{(d)}_\infty (\frac{y}{b(x)}), \frac{y}{b(x)} \rangle+ o( x^\frac{\beta-1}{2})\\
    &=o( x^\frac{\beta-1}{2})\quad\text{as $x\to\infty$.}
    \end{align*}
    The last estimation \eqref{eq:lambda} follows from this and \eqref{eq:bg}.

\bibliographystyle{plain}
\bibliography{bib.bib}

\end{document}